\documentclass[11pt]{amsart}
\usepackage{amsmath}
\usepackage{amsthm}
\usepackage{tikz}
\usepackage{tikz-cd}
\usepackage{hyperref}
\usepackage[left=2.5cm,right=2.5cm,top=3cm,bottom=3cm]{geometry}
\usepackage{amssymb}
\usepackage{amscd}
\usepackage{latexsym}
\usepackage{epsfig}
\usepackage{graphicx}
\usepackage{psfrag}
\usepackage{caption}
\usepackage{rotating}
\usepackage{mathtools}
\usepackage{xypic}
\usepackage{enumitem}
\usepackage{bigints}

\makeatletter
\@namedef{subjclassname@1991}{2020 Mathematics Subject Classification}
\makeatother

\allowdisplaybreaks[3]

\usepackage{url}
\usepackage{cite}
\usepackage{array,}

\usepackage{blindtext}
\usepackage{scrextend}
\addtokomafont{labelinglabel}{\sffamily}

\newtheorem{theorem}{Theorem}[subsection]
\newtheorem{proposition}[theorem]{Proposition}
\newtheorem{corollary}[theorem]{Corollary}
\newtheorem{lemma}[theorem]{Lemma}
\newtheorem{observation}[theorem]{Observation}

\theoremstyle{definition}
\newtheorem{definition}[theorem]{Definition}

\newtheorem{remark}[theorem]{Remark}
\newtheorem*{question*}{Motivating Question}

\makeatletter
\DeclareRobustCommand{\element}[1]{\@element#1\@nil}
\def\@element#1#2\@nil{%
	#1%
	\if\relax#2\relax\else\MakeLowercase{#2}\fi}
\makeatother

\title[Towards a quantization of the double via the enhanced symplectic ``category"]{Towards a quantization of the double via the enhanced symplectic ``category"}

\author[Peter Crooks]{Peter Crooks}
\author[Jonathan Weitsman]{Jonathan Weitsman}
\address{Department of Mathematics, Northeastern University, 360 Huntington Avenue, Boston, MA 02115, USA}
\email[Peter Crooks]{p.crooks@northeastern.edu}
\email[Jonathan Weitsman]{j.weitsman@northeastern.edu}
\begin{document}
	
	\maketitle

\begin{abstract}
This paper considers the enhanced symplectic ``category" for purposes of quantizing quasi-Hamiltonian $G$-spaces, where $G$ is a compact simple Lie group. Our starting point is the well-acknowledged analogy between the cotangent bundle $T^*G$ in Hamiltonian geometry and the internally fused double $D(G)=G\times G$ in quasi-Hamiltonian geometry. Guillemin and Sternberg consider the former, studing half-densities and phase functions on its so-called character Lagrangians $\Lambda_{\mathcal{O}}\subseteq T^*G$. Our quasi-Hamiltonian counterpart replaces these character Lagrangians with the universal centralizers $\Lambda_{\mathcal{C}}\longrightarrow\mathcal{C}$ of regular, $\frac{1}{k}$-integral conjugacy classes $\mathcal{C}\subseteq G$. We show each universal centralizer to be a ``quasi-Hamiltonian Lagrangian" in $D(G)$, and to come equipped with a half-density and phase function. 

At the same time, we consider a Dehn twist-induced automorphism $R:D(G)\longrightarrow D(G)$ that lacks a natural Hamiltonian analogue. Each quasi-Hamiltonian Lagrangian $R(\Lambda_{\mathcal{C}})$ is shown to have a clean intersection with every $\Lambda_{\mathcal{C}'}$, and to come equipped with a half-density and phase function of its own. This leads us to consider the possibility of a well-behaved, quasi-Hamiltonian notion of the BKS pairing between $R(\Lambda_{\mathcal{C}})$ and $\Lambda_{\mathcal{C}'}$. We construct such a pairing and study its properties. This is facilitated by the nice geometric fearures of $R(\Lambda_{\mathcal{C}})\cap\Lambda_{\mathcal{C}'}$ and a reformulation of the classical BKS pairing. Our work is perhaps the first step towards a level-$k$ quantization of $D(G)$ via the enhanced symplectic ``category".      
\end{abstract}

\tableofcontents

\section{Introduction}
\subsection{Context}
It is a premise of the geometric quantization programme that some symmetries of a symplectic manifold $M$ should induce symmetries of the quantization $Q(M)$. This can be captured via the formalism of symplectic ``categories" \cite{WeinsteinCategory}, through which geometric quantizations are related to the intersection theory of Lagrangian submanifolds. Guillemin and Sternberg \cite{GuilleminSternbergBook} study a variant of this construction, replacing the symplectic ``category" with a decorated version called the \textit{enhanced symplectic ``category"}. Lagrangian relations are thereby replaced with arrows of the form $$M_1\overset{(\Lambda,\rho)}\Longrightarrow M_2,$$ where $\Lambda$ is a Lagrangian relation and $\rho$ is a half-density on $\Lambda$. This is exemplified by a class of Lagrangian submanifolds in $T^*G$, where $G$ is a compact $1$-connected simple Lie group having Lie algebra $\mathfrak{g}$. The Lagrangian submanifolds in question are the so-called \textit{character Lagrangians} $\Lambda_{\mathcal{O}}\subseteq T^*G$, where $\mathcal{O}\subseteq\mathfrak{g}^*$ ranges over the regular integral coadjoint orbits of $G$. Guillemin and Sternberg \cite{GuilleminSternbergBook} endow each $\Lambda_{\mathcal{O}}$ with a distsinguished half-density $\rho_{\mathcal{O}}$ and phase function $\psi_{\mathcal{O}}:\Lambda_{\mathcal{O}}\longrightarrow S^1$.

Several papers \cite{Song,Shabazi,MeinrenkenLectures,LoizidesSong1,LoizidesSong2,Laurent,MeinrenkenTwisted,AMW,MeinrenkenQuant} address the quantization of quasi-Hamiltonian $G$-spaces \cite{AMM}. A recurring theme is that the approaches to geometric quantization in Hamiltonian geometry have counterparts in the quasi-Hamiltonian setting. One may therefore expect that the enhanced symplectic ``category" will lend itself to the quantization of quasi-Hamiltonian $G$-spaces. One example is the \textit{internally fused double} $D(G)=G\times G$ {\cite{AMM}, a quasi-Hamiltonian $G$-space and the quasi-Hamiltonian analogue the Hamiltonian $G$-space $T^*G$. The natural analogues of the character Lagrangians $\Lambda_{\mathcal{O}}\subseteq T^*G$ are then the universal centralizers $\Lambda_{\mathcal{C}}\subseteq D(G)$, where $\mathcal{C}\subseteq G$ ranges over the regular, $\frac{1}{k}$-integral conjugacy classes of $G$ and $k$ is a fixed positive integer. The role of $\Lambda_{\mathcal{O}}$ in a quantization of $T^*G$ via the enhanced symplectic ``category" might therefore suggest a role for $\Lambda_{\mathcal{C}}$ in quantizing $D(G)$.              
   
\subsection{Statement of results}
This paper investigates the possibility of quantizing $D(G)$ via universal centralizers $\Lambda_{\mathcal{C}}\subseteq D(G)$ and the enhanced symplectic ``category". We do not construct a geometric quantization itself; our emphasis is instead on the interesting geometric features of $D(G)$ and $\Lambda_{\mathcal{C}}$ that would be relevant to a quantization of $D(G)$ via the enhanced symplectic ``category". One such feature is a Dehn twist automorphism $R:D(G)\longrightarrow D(G)$ that lacks a natural Hamiltonian analogue. Our main technical tool is a generalized Blattner--Kostant--Sternberg (BKS) pairing \cite{Blattner} that applies in both symplectic and quasi-Hamiltonian geometry. The following is a more detailed summary of our results.

\subsubsection{A generalization of the classical BKS pairing}
We generalize the classical BKS pairing to suit our particular purposes. Our generalization applies to pairs $(M,\omega)$ with $\omega$ an arbitrary two-form on $M$. We let $M^{\circ}\subseteq M$ denote the open submanifold of points at which $\omega$ is non-degenerate, noting that $M^{\circ}$ is even-dimensional. A submanifold $\Lambda\subseteq M$ is then declared to be Lagrangian if $\Lambda\subseteq M^{\circ}$, $\dim(\Lambda)=\frac{1}{2}\dim(M^{\circ})$, and $\Lambda$ is isotropic with respect to $\omega$. This leads to the following observation.

\begin{observation}\label{Observation}
Let $(M,\omega)$ consist of a manifold $M$ and an arbitrary two-form $\omega\in\Omega^2(M)$. Suppose that $\Lambda_1,\Lambda_2\subseteq M$ are Lagrangian submanifolds in the sense defined above, and that they come equipped with respective half-densities $\rho_1,\rho_2$ and smooth maps $\psi_1:\Lambda_1\longrightarrow\mathbb{C},\psi_2:\Lambda_2\longrightarrow\mathbb{C}$. Let us also assume that $\Lambda_1$ and $\Lambda_2$ have a clean intersection in $M$, and that $\Lambda_1\cap\Lambda_2$ is compact. The usual definition of the BKS pairing  $$\mathrm{BKS}((\Lambda_1,\rho_1,\psi_1),(\Lambda_2,\rho_2,\psi_2))\in\mathbb{C}$$ for symplectic $(M,\omega)$ generalizes to our setting.
\end{observation}

Our interest in Observation \ref{Observation} stems from its applicability to quasi-Hamiltonian $G$-spaces $(M,\omega)$.
            
\subsubsection{Lagrangians in the internally fused double}
Let $G$ be a compact, connected, simply-connected, simple Lie group with Lie algebra $\mathfrak{g}$, exponential map $\exp:\mathfrak{g}\longrightarrow G$, and adjoint representation $\mathrm{Ad}:G\longrightarrow\operatorname{GL}(\mathfrak{g})$. Fix a maximal torus $T\subseteq G$ having Lie algebra $\mathfrak{t}\subseteq\mathfrak{g}$ and Weyl group $W:=N_G(T)/T$. Denote by $\langle\cdot,\cdot\rangle:\mathfrak{g}\otimes_{\mathbb{R}}\mathfrak{g}\longrightarrow\mathbb{R}$ the unique $G$-invariant inner product for which each long root has length $\sqrt{2}$, noting that $\langle\cdot,\cdot\rangle$ identifies $\mathfrak{t}$ with $\mathfrak{t}^*$. 

Use the left trivialization and $\langle\cdot,\cdot\rangle$ to identify $T^*G$ with $G\times\mathfrak{g}$. Each regular adjoint orbit $\mathcal{O}\subseteq\mathfrak{g}$ then determines a Lagrangian submanifold $\Lambda_{\mathcal{O}}\subseteq T^*G$, i.e.
$$\Lambda_{\mathcal{O}}:=\{(g,\xi)\in G\times\mathfrak{g}:\xi\in\mathcal{O}\text{ and }\mathrm{Ad}_g(\xi)=\xi\}\subseteq G\times\mathfrak{g}=T^*G.$$ One may instead consider a regular conjugacy class $\mathcal{C}\subseteq G$ and its \textit{universal centralizer} $\Lambda_{\mathcal{C}}\subseteq G\times G=D(G)$, defined by
$$\Lambda_{\mathcal{C}}:=\{(g,h)\in D(G):h\in\mathcal{C}\text{ and }ghg^{-1}=h\}.$$ We show this universal centralizer to a quasi-Hamiltonian counterpart of $\Lambda_{\mathcal{O}}$ in the following sense. 

\begin{theorem}\label{Theorem: Lagrangian}
Let $\mathcal{C}\subseteq G$ be a regular conjugacy class. The universal centralizer $\Lambda_{\mathcal{C}}$ is a Lagrangian submanifold of the internally fused double $D(G)$, i.e. $\Lambda_{\mathcal{C}}\subseteq D(G)^{\circ}$, $\dim(\Lambda_{\mathcal{C}})=n$, and $\Lambda_{\mathcal{C}}$ is isotropic with respect to the quasi-Hamiltonian $2$-form on $D(G)$.
\end{theorem}

\subsubsection{BKS pairings on the internally fused double}
Now choose a collection of simple roots and let $\mathfrak{A}\subseteq\mathfrak{t}$ be the associated fundamental Weyl alcove. Let $\beta(\mathcal{C})$ denote the unique element of $\mathfrak{A}$ satisfying $\exp(\beta(\mathcal{C}))\in\mathcal{C}$, where $\mathcal{C}\subseteq G$ is any conjugacy class. Let us also fix a positive integer $k$ and assume that $\mathcal{C}$ is a regular, $\frac{1}{k}$-integral conjugacy class, i.e. that $k\beta(\mathcal{C})$ is a regular, integral weight in $\mathfrak{t}\cong\mathfrak{t}^*$. 

Write $\mathcal{O}\subseteq\mathfrak{g}$ for the adjoint orbit of $k\beta(\mathcal{C})$, and noting that
$$\Lambda_{\mathcal{O}}\longrightarrow\Lambda_{\mathcal{C}},\quad (g,\xi)\mapsto (g,\exp\left(\frac{1}{k}\xi\right)),\quad (g,\xi)\in\Lambda_{\mathcal{O}}$$ is a diffeomorphism. On the other hand, Guillemin and Sternberg equip $\Lambda_{\mathcal{O}}$ with a canonical half-density $\rho_{\mathcal{O}}$ and phase function $\psi_{\mathcal{O}}:\Lambda_{\mathcal{O}}\longrightarrow S^1$. By means of the above-mentioned diffeomorphism, these determine a half-density $\rho_{\mathcal{C}}$ and function $\psi_{\mathcal{C}}:\Lambda_{\mathcal{C}}\longrightarrow S^1$ on $\Lambda_{\mathcal{C}}$. We will ultimately compute the BKS pairing of $(\Lambda_{\mathcal{C}},\rho_{\mathcal{C}},\psi_{\mathcal{C}})$ with another triple, constructed via the following ideas.

Certain gauge-theoretic considerations naturally endow $D(G)$ with an action of the mapping class group of the once-punctured torus, an action that respects the quasi-Hamiltonian $G$-space structure on $D(G)$. This mapping class group contains two Dehn twist automorphisms, one of which acts on $D(G)$ through the diffeomorphism
$$R:D(G)\longrightarrow D(G),\quad (g,h)\mapsto (g,hg^{-1}),\quad (g,h)\in D(G).$$ It follows that $R(\Lambda_{\mathcal{C}})$ is a Lagrangian submanifold of $D(G)$. Note also that the restricted diffeomorphism
$$R\big\vert_{\Lambda_{\mathcal{C}}}:\Lambda_{\mathcal{C}}\overset{\cong}\longrightarrow R(\Lambda_{\mathcal{C}})$$
identifies $\rho_{\mathcal{C}}$ (resp. $\psi_{\mathcal{C}}$) with a half-density $\nu_{\mathcal{C}}$ (resp. phase function $\vartheta_{\mathcal{C}}$) on $R(\Lambda_{\mathcal{C}})$. We thereby obtain the triple $(R(\Lambda_{\mathcal{C}}),\nu_{\mathcal{C}},\vartheta_{\mathcal{C}})$ alluded to in the previous paragraph. This leads to the following BKS pairing computation.

\begin{theorem}\label{Theorem: Main}
Fix a positive integer $k$ and let $\mathcal{C},\mathcal{C}'\subseteq G$ be regular, $\frac{1}{k}$-integral conjugacy classes. 
\begin{itemize}[leftmargin=20pt]
\item[(i)] The Lagrangian submanifolds $R(\Lambda_{\mathcal{C}})$ and $\Lambda_{\mathcal{C}'}$ have a clean intersection in $D(G)$.
\item[(ii)] The BKS pairing of $(R(\Lambda_{\mathcal{C}}),\nu_{\mathcal{C}},\vartheta_{\mathcal{C}})$ and $(\Lambda_{\mathcal{C}'},\rho_{\mathcal{C}'},\psi_{\mathcal{C}'})$ is given by
$$\mathrm{BKS}((R(\Lambda_{\mathcal{C}}),\nu_{\mathcal{C}},\vartheta_{\mathcal{C}}),(\Lambda_{\mathcal{C}'},\rho_{\mathcal{C}'},\psi_{\mathcal{C}'}))=k^{n-r}C(G,T)\cdot \bigg(\prod_{\alpha\in\Phi_{+}}\alpha(\beta(\mathcal{C}))\alpha(\beta(\mathcal{C}'))\bigg)^{\frac{1}{2}}\left(\sum_{w\in W}e^{2\pi i\lVert w\beta-\beta'\rVert^2}\right),$$
where $n=\dim(G)$, $r=\mathrm{rank}(G)$, $\Phi_{+}\subseteq\mathfrak{t}^*$ is the set of positive roots, $C(G,T)\in\mathbb{R}$ is a constant depending only on the root system of $(G,T)$, and $\lVert w\beta-\beta'\rVert$ is the length of $w\beta-\beta'$ with respect to $\langle\cdot,\cdot\rangle$.
\end{itemize}
\end{theorem}

\subsection{Organization}
Section \ref{Section: Quantization} uses the framework of symplectic ``categories" to motivate the investigations undertaken in this paper. Section \ref{Section: Lie-theoretic constructions} then develops the requisite Lie-theoretic machinery. This leads to Section \ref{Section: Quasi-Hamiltonian}, which studies the quasi-Hamiltonian geometry of $\Lambda_{\mathcal{C}}$ and includes a proof of Theorem \ref{Theorem: Lagrangian}. In Section \ref{Section: Main results}, we study the geometric features of $R(\Lambda_{\mathcal{C}})\cap\Lambda_{\mathcal{C}'}$ and prove Theorem \ref{Theorem: Main}. Our final section is an appendix that develops several linear-algebraic and differential-geometric properties of BKS pairings. A list of recurring notation appears at the end of this paper.   

\subsection*{Acknowledgements}
P.C. was supported by an NSERC postdoctoral fellowship [PDF--516638], while J.W. was supported by a Simons collaboration grant [\#579801].  

\section{Quantization and symplectic ``categories"}\label{Section: Quantization}
The following section emphasizes the connection between geometric quantization and the intersection theory of Lagrangian submanifolds. This connection is our impetus for studying the quasi-Hamiltonian Lagrangians $R(\Lambda_{\mathcal{C}}),\Lambda_{\mathcal{C}'}\subseteq D(G)$ in relation to one another. One can realize this connection via Weinstein's symplectic ``category" \cite{WeinsteinCategory} and an enrichment thereof \cite{GuilleminSternbergBook}; such ``categories" are related to the intersection theory of Lagrangian submanifolds, and to geometric quantization itself. We review the first of these relations, deferring to \cite{Bates} and \cite{GuilleminSternbergBook} for the latter.       

\subsection{The Weinstein symplectic ``category"}\label{Subsection: Category}
Let $(M,\omega)$ be any symplectic manifold. In what follows, we adopt the shorthand notation $M$ for $(M,\omega)$ and $\overline{M}$ for the symplectic manifold $(M,-\omega)$.

Suppose that $M_1$ and $M_2$ are symplectic manifolds. Recall that a \textit{Lagrangian relation} from $M_1$ to $M_2$ is a Lagrangian submanifold $\Lambda\subseteq \overline{M_1}\times M_2$, where $\overline{M_1}\times M_2$ is equipped with the usual product symplectic form. In this case, one writes 
$$M_1\overset{\Lambda}\Longrightarrow M_2.$$ 

Let $M_1$, $M_2$, and $M_3$ be symplectic manifolds. Given Lagrangian relations
$$M_1\overset{\Lambda_{12}}\Longrightarrow M_2\quad\text{and}\quad M_2\overset{\Lambda_{23}}\Longrightarrow M_3,$$ one defines
$$\Lambda_{23}\circ\Lambda_{12}:=\{(m_1,m_3)\in M_1\times M_3:\exists m_2\in M_2\text{ with } (m_1,m_2)\in\Lambda_{12}\text{ and } (m_2,m_3)\in\Lambda_{23}\}.$$ Note that $\Lambda_{23}\circ\Lambda_{12}$ need not be a submanifold of $M_1\times M_3$; one typically addresses this by requiring $\Lambda_{23}$ and $\Lambda_{12}$ to be \textit{cleanly composable}, an intersection-theoretic condition discussed at length in \cite[Section 4.2]{GuilleminSternbergBook}. This condition forces $\Lambda_{23}\circ\Lambda_{12}$ to be a submanifold and a Lagrangian relation, called the \textit{composition} of $\Lambda_{23}$ and $\Lambda_{12}$ and denoted
$$M_1\overset{\Lambda_{12}}\Longrightarrow M_2\overset{\Lambda_{23}}\Longrightarrow M_3.$$

It is tempting to imagine a category with symplectic manifolds as objects, Lagrangian relations as morphisms, and morphism composition as defined above. This is precluded by the existence of Lagrangian relations whose composition fails to be a submanifold. One instead refers to symplectic manifolds and Lagrangian relations as forming the \textit{Weinstein symplectic ``category"}, and views morphism composition as defined only for cleanly composable relations.      

\subsection{Examples and basic properties}

\begin{enumerate}[leftmargin=*]
	\item If $M$ is a symplectic manifold, then Lagrangian submanifolds of $\Lambda\subseteq M$ are in bijective correspondence with Lagrangian relations of the form
	$$\text{pt}\overset{\Lambda}\Longrightarrow M.$$
	\item If $M$ is a symplectic manifold, then the diagonal $\Delta_M\subseteq M\times M$ is a Lagrangian submanifold of $\overline{M}\times M$. By the previous example, it defines a Lagrangian relation
	$$\text{pt}\overset{\Delta_M}\Longrightarrow\overline{M}\times M.$$
	\item Suppose that $G$ is a compact Lie group with Lie algebra $\mathfrak{g}$. Let $M$ be a symplectic manifold endowed with a free Hamiltonian $G$-action and associated moment map $\mu:M\longrightarrow\mathfrak{g}^*$. Consider the symplectic quotient $\mu^{-1}(0)/G$ and canonical map
	$$\pi:\mu^{-1}(0)\longrightarrow\mu^{-1}(0)/G.$$ The graph of $\pi$ is then a Lagrangian relation
	$$M\overset{\Lambda}\Longrightarrow\mu^{-1}(0)/G,$$ i.e.
	$$\Lambda:=\{(m,\pi(m)):m\in\mu^{-1}(0)\}\subseteq M\times\mu_{D(G)}^{-1}(0)/G.$$ One often calls $\Lambda$ the \textit{reduction morphism}.
	\item Retain the hypotheses and notation of the previous example, except for the freeness condition on the $G$-action. Use the left trivialization to identify $T^*G$ with $G\times\mathfrak{g}^*$, and consider
	$$\Lambda_{\mu}:=\{(m,g\cdot m,g,\mu(m)):m\in M,\text{ }g\in G\}\subseteq M\times M\times G\times\mathfrak{g}^*=M\times M\times T^*G.$$ This is a Lagrangian submanifold of $M\times\overline{M}\times T^*G$, and it thereby defines a Lagrangian relation
	$$\overline{M}\times M\overset{\Lambda_{\mu}}\Longrightarrow T^*G.$$ One sometimes calls $\Lambda_{\mu}$ the \textit{moment Lagrangian}.
	\item Suppose that $G$ is a compact Lie group with Lie algebra $\mathfrak{g}$, adjoint representation $\mathrm{Ad}:G\longrightarrow\operatorname{GL}(\mathfrak{g})$, and fixed $G$-invariant inner product $\langle\cdot,\cdot\rangle:\mathfrak{g}\otimes_{\mathbb{R}}\mathfrak{g}\longrightarrow\mathbb{R}$. This inner product induces an isomorphism between the adjoint and coadjoint representations of $G$, through which we may freely identify $\mathfrak{g}$ and $\mathfrak{g}^*$. Each adjoint orbit $\mathcal{O}\subseteq\mathfrak{g}$ thereby corresponds to a coadjoint orbit, and the latter comes equipped with its canonical Kirillov--Kostant--Souriau symplectic form. It follows that $\mathcal{O}$ is a symplectic manifold, and one sees that the $G$-action on $\mathcal{O}$ is Hamiltonian with moment map the inclusion
	$\mu:\mathcal{O}\longrightarrow\mathfrak{g}$. One can also verify that the Lagrangian relations
	$$\mathrm{pt}\overset{\Delta_{\mathcal{O}}}\Longrightarrow\overline{\mathcal{O}}\times\mathcal{O}\quad\text{and}\quad \overline{\mathcal{O}}\times\mathcal{O}\overset{\Lambda_{\mu}}\Longrightarrow T^*G$$ are cleanly composable, so that we have a Lagrangian relation
	$$\text{pt}\overset{\Lambda_{\mu}\circ\Delta_{\mathcal{O}}}\Longrightarrow T^*G.$$ This amounts to having a Lagrangian submanifold $\Lambda_{\mathcal{O}}\subseteq T^*G$. If one uses the left trivialization and $\langle\cdot,\cdot\rangle$ to identify $T^*G$ with $G\times\mathfrak{g}$, then
	$$\Lambda_{\mathcal{O}}=\{(g,\xi)\in G\times\mathfrak{g}:\xi\in\mathcal{O}\text{ and }\mathrm{Ad}_g(\xi)=\xi\}.$$ This submanifold is often called the \textit{orbit Lagrangian}. 
\end{enumerate}

\subsection{The enhanced symplectic ``category"}\label{Subsection: Enhanced symplectic category}
Guillemin and Sternberg \cite{GuilleminSternbergBook} discuss an enrichment of the Weinstein symplectic ``category". This involves replacing Lagrangian relations 
$$M_1\overset{\Lambda}\Longrightarrow M_2$$
with arrows of the form
$$M_1\overset{(\Lambda,\rho)}\Longrightarrow M_2,$$ where $\Lambda\subseteq \overline{M_1}\times M_2$ is a Lagrangian relation from $M_1$ to $M_2$ and $\rho$ is a half-density on $\Lambda$. Now suppose that $M_1$, $M_2$, and $M_3$ are symplectic manifolds with arrows
\begin{equation}\label{Equation: Arrows}M_1\overset{(\Lambda_{12},\rho_{12})}\Longrightarrow M_2\quad\text{and}\quad M_2\overset{(\Lambda_{23},\rho_{23})}\Longrightarrow M_3.\end{equation} Assuming that $\Lambda_{23}$ and $\Lambda_{12}$ are cleanly composable as Lagrangian relations, Guillemin and Sternberg construct a half-density $\rho_{23}\circ\rho_{12}$ on the Lagrangian submanifold $\Lambda_{23}\circ\Lambda_{12}\subseteq\overline{M_1}\times M_3$. The composition of the arrows \eqref{Equation: Arrows} is then defined to be $(\Lambda,\rho)=(\Lambda_{23}\circ\Lambda_{12},\rho_{23}\circ\rho_{12})$, i.e. the new arrow
$$M_1\overset{(\Lambda,\rho)}\Longrightarrow M_3.$$

The \textit{enhanced symplectic ``category"} is defined to have symplectic manifolds as objects, arrows of the form
$$M_1\overset{(\Lambda,\rho)}\Longrightarrow M_2$$ as morphisms, and morphism composition as defined above. This fails to form a category, as one composes morphisms only if the underlying Lagrangian relations are cleanly composable.

\subsection{BKS pairings}\label{Subsection: BKS pairings}
The enhanced symplectic ``category" is compatible with the so-called Blattner--Kostant--Sternberg (BKS) pairing \cite{Blattner} in geometric quantization. To elaborate on this, we let $(V,\omega)$ be a $2n$-dimensional symplectic vector space over $\mathbb{R}$. Recall that a density of order $\alpha\in (0,\infty)$ on an $m$-dimensional subspace $W\subseteq V$ is a map
$\rho:W^m\longrightarrow\mathbb{C}$ satisfying
$$\rho(A(v_1),\ldots,A(v_m))=\vert\det(A)\vert^{\alpha}\rho(v_1,\ldots,v_m)$$ for all $A\in\operatorname{End}(W)$ and $(v_1,\ldots,v_m)\in W^m$. We use the term density (resp. half-density) for a density of order one (resp. $\frac{1}{2}$), and write $\vert W\vert^{\alpha}$ for the one-dimensional complex vector space of densities of order $\alpha$ on $W$. Let us also write $\vert W\vert$ for $\vert W\vert^{\alpha}$ in the case $\alpha=1$. One crucial ingredient of BKS theory is the existence of a canonical isomorphism
\begin{equation}\label{Equation: Vector space BKS}\Phi:\vert L_1\vert^{\frac{1}{2}}\otimes\vert L_2\vert^{\frac{1}{2}}\overset{\cong}\longrightarrow\vert L_1\cap L_2\vert\end{equation} for any pair of Lagrangian subspaces $L_1,L_2\subseteq V$. A more detailed discussion appears in the appendix.

Now let $(M,\omega)$ be a symplectic manifold with cleanly intersecting Lagrangian submanifolds $\Lambda_1,\Lambda_2\subseteq M$. Assume that $\Lambda_1\cap\Lambda_2$ is compact, and that $\Lambda_1$ and $\Lambda_2$ come with respective half-densities $\rho_1$ and $\rho_2$ and smooth functions $\psi_1:\Lambda_1\longrightarrow\mathbb{C}$ and $\psi_2:\Lambda_2\longrightarrow\mathbb{C}$. The half-densities $\rho_1$ and $\rho_2$ are sections of appropriate half-density bundles. One may therefore apply \eqref{Equation: Vector space BKS} on the level of tangent spaces to obtain a density $D(\rho_1,\rho_2)$ on $\Lambda_1\cap\Lambda_2$. The BKS pairing of $(\Lambda_1,\rho_1,\psi_1)$ and $(\Lambda_2,\rho_2,\psi_2)$ is then defined by
$$\mathrm{BKS}((\Lambda_1,\rho_1,\psi_1),(\Lambda_2,\rho_2,\psi_2)):=\int_{\Lambda_1\cap\Lambda_2}\psi_1\overline{\psi_2},$$
where the integration is with respect to $D(\rho_1,\rho_2)$.

\section{Lie-theoretic constructions}\label{Section: Lie-theoretic constructions}

In what follows, we introduce and develop the Lie-theoretic underpinnings of our work.

\subsection{The basics}\label{Subsection: The basics}
Let $G$ be a compact, connected, simply-connected, $n$-dimensional, rank-$r$, simple Lie group with Lie algebra $\mathfrak{g}$, exponential map $\exp:\mathfrak{g}\longrightarrow G$, and adjoint representation $\mathrm{Ad}:G\longrightarrow\operatorname{GL}(\mathfrak{g})$. Write $G_g$ for the $G$-centralizer of $g\in G$ and 
$$G_{\text{reg}}:=\{g\in G:\dim(G_g)=r\}$$ for the open, dense, conjugation-invariant subset of regular elements in $G$. A conjugacy class $\mathcal{C}\subseteq G$ shall be called \textit{regular} if $\mathcal{C}\subseteq G_{\text{reg}}$, i.e. if $\mathcal{C}$ is the conjugacy class of a regular element in $G$.
 
We likewise write $G_x\subseteq G$ (resp. $\mathfrak{g}_x\subseteq\mathfrak{g}$) for the $G$-stabilizer (resp. $\mathfrak{g}$-centralizer) of $x\in\mathfrak{g}$ and 
$$\mathfrak{g}_{\text{reg}}:=\{x\in\mathfrak{g}:\dim(\mathfrak{g}_x)=r\}$$ for the open, dense, $G$-invariant subset of regular elements in $\mathfrak{g}$. An adjoint orbit $\mathcal{O}\subseteq\mathfrak{g}$ shall be called \textit{regular} if $\mathcal{O}\subseteq\mathfrak{g}_{\text{reg}}$, i.e. if $\mathcal{O}$ is the adjoint orbit of a regular element in $\mathfrak{g}$.

Some of the preceding discussion fits into the following more general framework. Let $M$ be a $G$-manifold, i.e. a manifold equipped with a smooth left action of $G$. Each $m\in M$ then determines a $G$-stabilizer
$$G_m:=\{g\in G:g\cdot m=m\}\subseteq G$$ and $G$-orbit $$Gm:=\{g\cdot m:g\in G\}\subseteq M.$$

Let $T\subseteq G$ be a maximal torus with Lie algebra $\mathfrak{t}\subseteq\mathfrak{g}$, setting
$$T_{\text{reg}}:=T\cap G_{\text{reg}}\quad\text{and}\quad\mathfrak{t}_{\text{reg}}:=\mathfrak{t}\cap\mathfrak{g}_{\text{reg}}.$$ It follows that $t\in T$ (resp. $x\in\mathfrak{t}$) belongs to $T_{\text{reg}}$ (resp. $\mathfrak{t}_{\text{reg}}$) if and only if $G_t=T$ (resp. $\mathfrak{g}_x=\mathfrak{t}$).

Let $X^{\bullet}(T)$ denote the \textit{weight lattice} of all Lie group morphisms $T\longrightarrow S^1$. Each $\alpha\in X^{\bullet}(T)$ can be differentiated at $e\in T$ to yield an $\mathbb{R}$-linear map
$$(d\alpha)_e:\mathfrak{t}\longrightarrow T_1S^1=i\mathbb{R}.$$ The map
\begin{equation}\label{Equation: Z-module embedding}X^{\bullet}(T)\longrightarrow\mathfrak{t}^*,\quad\alpha\mapsto \frac{1}{2\pi i}(d\alpha)_e,\quad\alpha\in X^{\bullet}(T)\end{equation} is then an embedding of $\mathbb{Z}$-modules. Elements in its image are called \textit{integral weights}.

Let us set $\mathfrak{g}_{\mathbb{C}}:=\mathfrak{g}\otimes_{\mathbb{C}}\mathbb{R}$ and write $\mathrm{Ad}^{\mathbb{C}}:G\longrightarrow\operatorname{GL}(\mathfrak{g}_{\mathbb{C}})$ for the complexification of the adjoint representation. Recall that $\alpha\in X^{\bullet}(T)\setminus\{0\}$ is called a \textit{root} if
$$\mathfrak{g}_{\alpha}:=\{\xi\in\mathfrak{g}_{\mathbb{C}}:\mathrm{Ad}_t^{\mathbb{C}}(\xi)=\alpha(t)\xi\text{ for all }t\in T\}\subseteq\mathfrak{g}_{\mathbb{C}}$$
is one-dimensional. If one uses \eqref{Equation: Z-module embedding} to regard roots as being in $\mathfrak{t}^*$, then $\alpha\in\mathfrak{t}^*\setminus\{0\}$ is a root if and only if
$$\mathfrak{g}_{\alpha}:=\{\xi\in\mathfrak{g}_{\mathbb{C}}:[\eta,\xi]=2\pi i\alpha(\eta)\xi\text{ for all }\eta\in \mathfrak{t}\}\subseteq\mathfrak{g}_{\mathbb{C}}$$
is one-dimensional.
This leads to the $T$-module decomposition
$$\mathfrak{g}_{\mathbb{C}}=\mathfrak{t}_{\mathbb{C}}\oplus\bigoplus_{\alpha\in\Phi}\mathfrak{g}_{\alpha},$$
where $\mathfrak{t}_{\mathbb{C}}:=\mathfrak{t}\otimes_{\mathbb{R}}\mathbb{C}$ and $\Phi$ is the set of all roots. Choose a decomposition $\Phi=\Phi_{+}\cup\Phi_{-}$, where $\Phi_{+}$ is a set of positive roots and $\Phi_{-}=-\Phi_{+}$ is the corresponding set of negative roots. Let $\Delta\subseteq\Phi_{+}$ denote the resulting set of simple roots, and write $\alpha_0\in\Phi_{+}$ for the highest root. 

Let $\langle\cdot,\cdot\rangle:\mathfrak{g}\otimes_{\mathbb{R}}\mathfrak{g}\longrightarrow\mathbb{R}$ denote the \textit{normalized inner product} on $\mathfrak{g}$. This is the unique $G$-invariant inner product satisfying a certain normalization condition, to be stated momentarily. Note that $\langle\cdot,\cdot\rangle$ induces an isomorphism
\begin{equation}\label{Equation: Killing isomorphism}\mathfrak{g}\overset{\cong}\longrightarrow\mathfrak{g}^*,\quad\xi\mapsto\langle\xi,\cdot\rangle,\quad\xi\in\mathfrak{g}\end{equation} between the adjoint representation and its dual. Our inner product also yields a vector space isomorphism
$$\mathfrak{t}\overset{\cong}\longrightarrow\mathfrak{t}^*,\quad\xi\mapsto\xi^{\vee}:=\langle\xi,\cdot\rangle\big\vert_{\mathfrak{t}},\quad\xi\in\mathfrak{t},$$
through which $\mathfrak{t}^*$ acquires an inner product. Abusing notation slightly, we also use $\langle\cdot,\cdot\rangle$ to denote this new inner product. Our normalization condition on $\langle\cdot,\cdot\rangle:\mathfrak{g}\otimes_{\mathbb{R}}\mathfrak{g}\longrightarrow\mathbb{R}$ is then given by
$$\langle\alpha_0,\alpha_0\rangle=2$$ 

Now consider the conjugation action of $G$ on itself, as well as its canonical lift to a Hamiltonian $G$-action on $T^*G$. If one uses $\langle\cdot,\cdot\rangle$ and the left trivialization to identify $T^*G$ with $G\times\mathfrak{g}$, then this latter action is given by
\begin{equation}\label{Equation: Action on cotangent} g\cdot (h,x)=(ghg^{-1},\mathrm{Ad}_g(x))\end{equation} for all $g\in G$ and $(h,x)\in G\times\mathfrak{g}$. This is the only $G$-action on $T^*G$ that we will consider.  

Recall that $\alpha\in\mathfrak{t}^*$ is called \textit{dominant} if $\langle\alpha,\beta\rangle\geq 0$ for all $\beta\in\Delta$. We refer to $\xi\in\mathfrak{t}$ as being dominant if $\xi^{\vee}$ is so, i.e. $\xi\in\mathfrak{t}$ is dominant if and only if $\alpha(\xi)\geq 0$ for all $\alpha\in\Delta$. We likewise call an element $\xi\in\mathfrak{t}$ \textit{integral} if $\xi^{\vee}$ is an integral weight in $\mathfrak{t}^*$. An adjoint orbit $\mathcal{O}\subseteq\mathfrak{g}$ shall be called \textit{integral} if it contains an integral element of $\mathfrak{t}$.

Let $W:=N_G(T)/T$ be the Weyl group associated to $(G,T)$. Recall that $W$ acts linearly on $\mathfrak{t}$ through the adjoint representation of $G$. The induced $W$-action on $\mathfrak{t}^*$ preserves $X^{\bullet}(T)$ and $\Phi$, and we have a length function $\ell:W\longrightarrow\mathbb{Z}_{\geq 0}$ defined by
$$\ell(w):=\vert\Phi_{+}\cap w^{-1}\Phi_{-}\vert,\quad w\in W.$$

Consider the \textit{fundamental Weyl alcove} $\mathfrak{A}\subseteq\mathfrak{t}$ defined by $$\mathfrak{A}:=\{\xi\in\mathfrak{t}:\alpha_0(\xi)\leq 1\text{ and }\alpha(\xi)\geq 0\text{ for all }\alpha\in\Delta\}.$$ One has a bijection
$$\mathfrak{A}\overset{\cong}\longrightarrow\{\text{conjugacy classes in }G\},\quad \xi\mapsto\mathcal{C}_{\exp(\xi)},\quad \xi\in\mathfrak{A},$$
where $\mathcal{C}_{\exp(\xi)}\subseteq G$ denotes the conjugacy class of $\exp(\xi)$. Let $$\beta:\{\text{conjugacy classes in }G\}\overset{\cong}\longrightarrow\mathfrak{A}$$ be the inverse bijection, i.e. $\beta(\mathcal{C})$ is the unique element of $\mathfrak{A}$ that exponentiates to an element of $\mathcal{C}$. Let us also suppose that $k$ is a positive integer. We then refer to a conjugacy class $\mathcal{C}\subseteq G$ as being \textit{regular} (resp. \textit{$\frac{1}{k}$-integral}) if $\beta(\mathcal{C})$ (resp. $k\beta(\mathcal{C})$) is a regular (resp. an integral) element of $\mathfrak{t}$. These notions features prominently in our work.

\subsection{Adjoint orbits}\label{Subsection: Adjoint}
Fix $\xi\in\mathfrak{g}$ and let $\mathcal{O}\subseteq\mathfrak{g}$ denote its adjoint orbit. Note that the isomorphism \eqref{Equation: Killing isomorphism} identifies $\mathcal{O}$ with a coadjoint orbit in $\mathfrak{g}^*$. The Kirillov--Kostant--Souriau symplectic form on the latter orbit thereby corresponds to a symplectic form $\omega_{\mathcal{O}}$ on $\mathcal{O}$, which we now describe.

Let us write $V^{\perp}$ for the annihilator of a subspace $V\subseteq\mathfrak{g}$ with respect to $\langle\cdot,\cdot\rangle$. We may identify $T_{\xi}\mathcal{O}$ with $\mathfrak{g}_{\xi}^{\perp}$ in the usual way, i.e. via the isomorphism that sends each vector in $\mathfrak{g}_{\xi}^{\perp}$ to its fundamental vector field on $\mathcal{O}$ at $\xi$. Let $\omega_{\xi}$ denote the value of $\omega_{\mathcal{O}}$ at $\xi$, by which we mean
$$\omega_{\xi}:=(\omega_{\mathcal{O}})_{\xi}\in\wedge^2((\mathfrak{g}_{\xi}^{\perp})^*)\cong\wedge^2((T_{\xi}\mathcal{O})^*).$$ The form $\omega_{\mathcal{O}}$ is then characterized by being $G$-invariant and satisfying
$$\omega_{\xi}(\eta_1,\eta_2)=\langle\xi,[\eta_1,\eta_2]\rangle$$ for all $\eta_1,\eta_2\in\mathfrak{g}_{\xi}^{\perp}\cong T_{\xi}\mathcal{O}$. 

Let $\Omega_{\mathcal{O}}$ denote the highest non-zero wedge power of $\omega_{\mathcal{O}}$ and set $\Omega_{\xi}:=(\Omega_{\mathcal{O}})_{\xi}$. Note that 
$$\Omega_{\xi}\in\wedge^{n-r}((\mathfrak{g}_{\xi}^{\perp})^*)\cong\wedge^{n-r}((T_{\xi}\mathcal{O})^*)$$
if $\xi\in\mathfrak{g}_{\text{reg}}$.

\begin{lemma}\label{Lemma: Symplectic form identity}
	If $\xi\in\mathfrak{t}_{\emph{reg}}$ and $w\in W$, then
	$$\Omega_{w\xi}=(-1)^{\ell(w)}\Omega_{\xi}$$ as multilinear forms on $\mathfrak{t}^{\perp}$.
\end{lemma}

\begin{proof}
	Choose root vectors
	$\{e_{\alpha}\in\mathfrak{g}_{\alpha}\}_{\alpha\in\Phi_{+}}$ and $\{e_{-\alpha}\in\mathfrak{g}_{-\alpha}\}_{\alpha\in\Phi_{+}}$ so that $$\{i(e_{\alpha}+e_{-\alpha})\}_{\alpha\in\Phi_{+}}\cup\{e_{\alpha}-e_{-\alpha}\}_{\alpha\in\Phi_{+}}$$ is an orthonormal basis of $\mathfrak{t}^{\perp}$. Let us record these basis vectors as
	$$x_{\alpha}:=i(e_{\alpha}+e_{-\alpha})\quad\text{and}\quad y_{\alpha}:=e_{\alpha}-e_{-\alpha}$$ for each $\alpha\in\Phi_{+}$. In the interest of what lies ahead, we note the following straightforward identity:
	\begin{equation}\label{Equation: First identity}
	[\xi,x_{\alpha}]=-2\pi\alpha(\xi)y_{\alpha}
	\end{equation}
	for all $\alpha\in\Phi_{+}$.
	
    Now note that
	\begin{align*}\omega_{\xi}(x_{\alpha},x_{\beta}) & = \langle\xi,[x_{\alpha},x_{\beta}]\rangle\\
	& = \langle[\xi,x_{\alpha}],x_{\beta}\rangle\\
	& = -2\pi\alpha(\xi)\langle y_{\alpha},x_{\beta}\rangle\hspace{50pt}\text{[by \eqref{Equation: First identity}]}\\
	& = 0\hspace{111pt}\text{[by orthogonality]}
	\end{align*}
	for all $\alpha,\beta\in\Phi_{+}$. A similar argument establishes that $\omega_{\xi}(y_{\alpha},y_{\beta})=0$ for all $\alpha,\beta\in\Phi_{+}$. On the other hand,
	\begin{align*}\omega_{\xi}(x_{\alpha},y_{\beta}) & = \langle\xi,[x_{\alpha},y_{\beta}]\rangle\\
	& = \langle[\xi,x_{\alpha}],y_{\beta}\rangle\\
	& = -2\pi\alpha(\xi)\langle y_{\alpha},y_{\beta}\rangle\hspace{50pt}\text{[by \eqref{Equation: First identity}]}\\
	\end{align*}
	for all $\alpha,\beta\in\Phi_{+}$. We conclude that $\omega_{\xi}(x_{\alpha},y_{\beta})$ is zero if $\alpha\neq\beta$, and that it coincides with $-2\pi\alpha(\xi)$ if $\alpha=\beta$.
	
	The preceding calculations imply that
	\begin{equation}\label{Equation: Prelim}\omega_{\xi}=-2\pi\sum_{\alpha\in\Phi_{+}}\alpha(\xi)x_{\alpha}^*\wedge y_{\alpha}^*,\end{equation}
	where $\{x_{\alpha}^*\}_{\alpha\in\Phi_{+}}\cup\{y_{\alpha}^*\}_{\alpha\in\Phi_{+}}$ is the basis of $(\mathfrak{t}^{\perp})^*$ induced by our basis $\{x_{\alpha}\}_{\alpha\in\Phi_{+}}\cup\{y_{\alpha}\}_{\alpha\in\Phi_{+}}$ of $\mathfrak{t}^{\perp}$. Now let $\gamma$ denote the wedge product of the bilinear forms $x_{\alpha}^*\wedge y_{\alpha}^*$, with $\alpha$ ranging over all $\alpha\in\Phi_{+}$. Equation \eqref{Equation: Prelim} then implies that
	\begin{equation}\label{Equation: Nice} \Omega_{\xi}=(-2\pi)^m\bigg(\prod_{\alpha\in\Phi_{+}}\alpha(\xi)\bigg)\gamma,\end{equation}
	where $m$ is the number of positive roots. At the same time, analogous arguments establish that
	$$\Omega_{w\xi}=(-2\pi)^m\bigg(\prod_{\alpha\in\Phi_{+}}\alpha(w\xi)\bigg)\gamma$$ for all $w\in W$. It therefore suffices to prove that
	$$\prod_{\alpha\in\Phi_{+}}\alpha(w\xi)=(-1)^{\ell(w)}\prod_{\alpha\in\Phi_{+}}\alpha(\xi).$$
	
	We have
	\begin{align*}
	\prod_{\alpha\in\Phi_{+}}\alpha(w\xi) & = \prod_{\alpha\in\Phi_{+}}(w^{-1}\alpha)(\xi)\\
	& = \prod_{\alpha\in w^{-1}\Phi_{+}}\alpha(\xi)\\
	& = \bigg(\prod_{\alpha\in\Phi_{+}\cap w^{-1}\Phi_{+}}\alpha(\xi)\bigg)\bigg(\prod_{\alpha\in\Phi_{-}\cap w^{-1}\Phi_{+}}\alpha(\xi)\bigg)\\
	& = \bigg(\prod_{\alpha\in\Phi_{+}\cap w^{-1}\Phi_{+}}\alpha(\xi)\bigg)\bigg(\prod_{\alpha\in\Phi_{+}\cap w^{-1}\Phi_{-}}(-\alpha(\xi))\bigg)\\
	& = (-1)^{\ell(w)}\bigg(\prod_{\alpha\in\Phi_{+}\cap w^{-1}\Phi_{+}}\alpha(\xi)\bigg)\bigg(\prod_{\alpha\in\Phi_{+}\cap w^{-1}\Phi_{-}}\alpha(\xi)\bigg)\\
	& = (-1)^{\ell(w)}\prod_{\alpha\in\Phi_{+}}\alpha(\xi),
	\end{align*}
	where the second-last line follows from the fact that $\ell(w)=\vert\Phi_{+}\cap w^{-1}\Phi_{-}\vert$.
	This completes the proof.
\end{proof} 

Fix $\xi\in\mathfrak{t}_{\text{reg}}$ and let $\mathcal{O}\subseteq\mathfrak{g}$ denote its adjoint orbit. Consider the unique $G$-equivariant diffeomorphism $$G/T\overset{\cong}\longrightarrow\mathcal{O}$$ that sends $[e]\in G/T$ to $\xi\in\mathcal{O}$, and let $\Xi_{\xi}$ denote the pullback of $\Omega_{\mathcal{O}}$ along this diffeomorphism. It follows that $\Xi_{\xi}$ is a $G$-invariant volume form on $G/T$. Taking absolute values yields a density $\vert\Xi_{\xi}\vert$, whose square root is a half-density $\vert\Xi_{\xi}\vert^{\frac{1}{2}}$. Note that if $\xi,\xi'\in\mathfrak{t}_{\text{reg}}$, then the product $\vert\Xi_{\xi}\vert^{\frac{1}{2}}\vert\Xi_{\xi'}\vert^{\frac{1}{2}}$ is a density on $G/T$. 

\begin{lemma}\label{Lemma: Volume}
	Suppose that $\xi,\xi'\in\mathfrak{t}$ are dominant and regular. The volume of $G/T$ with respect to $\vert\Xi_{\xi}\vert^{\frac{1}{2}}\vert\Xi_{\xi'}\vert^{\frac{1}{2}}$ is
	$$\mathrm{Vol}(G/T,\vert\Xi_{\xi}\vert^{\frac{1}{2}}\vert\Xi_{\xi'}\vert^{\frac{1}{2}})=\frac{\bigg(\prod_{\alpha\in\Phi_{+}}\alpha(\xi)\alpha(\xi')\bigg)^{\frac{1}{2}}}{\prod_{\alpha\in\Phi_{+}}\langle\alpha,\rho\rangle},$$
	where $\rho:=\frac{1}{2}\sum_{\alpha\in\Phi_{+}}\alpha$.
\end{lemma}

\begin{proof}
	Note that evaluating $\vert\Xi_{\xi}\vert^{\frac{1}{2}}$ (resp. $\vert\Xi_{\xi'}\vert^{\frac{1}{2}}$) at $[e]\in G/T$ produces a half-density $(\vert\Xi_{\xi}\vert^{\frac{1}{2}})_{[e]}$ (resp. $(\vert\Xi_{\xi'}\vert^{\frac{1}{2}})_{[e]}$) on $T_{[e]}(G/T)=\mathfrak{t}^{\perp}$. To describe these half-densities, we recall the notation used in the proof of Lemma \ref{Lemma: Symplectic form identity}. The arguments used to establish \eqref{Equation: Nice} then imply that
	$$(\vert\Xi_{\xi}\vert^{\frac{1}{2}})_{[e]}=(2\pi)^{\frac{m}{2}}\bigg(\prod_{\alpha\in\Phi_{+}}\alpha(\xi)\bigg)^{\frac{1}{2}}\vert\gamma\vert^{\frac{1}{2}}\quad\text{and}\quad(\vert\Xi_{\xi'}\vert^{\frac{1}{2}})_{[e]}=(2\pi)^{\frac{m}{2}}\bigg(\prod_{\alpha\in\Phi_{+}}\alpha(\xi')\bigg)^{\frac{1}{2}}\vert\gamma\vert^{\frac{1}{2}}.$$ The value of $\vert\Xi_{\xi}\vert^{\frac{1}{2}}\vert\Xi_{\xi'}\vert^{\frac{1}{2}}$ at $[e]$ is therefore given by  \begin{equation}\label{Equation: One} (\vert\Xi_{\xi}\vert^{\frac{1}{2}}\vert\Xi_{\xi'}\vert^{\frac{1}{2}})_{[e]}=(2\pi)^m\bigg(\prod_{\alpha\in\Phi_{+}}\alpha(\xi)\alpha(\xi')\bigg)^{\frac{1}{2}}\vert\gamma\vert.\end{equation}
	
	Now consider the unique $G$-invariant Riemannian metric on $G/T$ that coincides with $\langle\cdot,\cdot\rangle$ on $T_{[e]}(G/T)=\mathfrak{t}^{\perp}$, and write $\theta$ for the $G$-invariant density on $G/T$ induced by this metric. Let us also recall that $\{x_{\alpha}\}_{\alpha\in\Phi_{+}}\cup\{y_{\alpha}\}_{\alpha\in\Phi_{+}}$ is an orthonormal basis of $\mathfrak{t}^{\perp}$ with respect to $\langle\cdot,\cdot\rangle$, and that $\gamma$ is the wedge product of all the skew-symmetric bilinear forms $x_{\alpha}^*\wedge y_{\alpha}^*$. This last sentence implies that $\theta_{[e]}=\vert\gamma\vert$, where $\theta_{[e]}$ denotes the value of $\theta$ at $[e]$. The statement \eqref{Equation: One} then becomes
	$$(\vert\Xi_{\xi}\vert^{\frac{1}{2}}\vert\Xi_{\xi'}\vert^{\frac{1}{2}})_{[e]}=(2\pi)^m\bigg(\prod_{\alpha\in\Phi_{+}}\alpha(\xi)\alpha(\xi')\bigg)^{\frac{1}{2}}\theta_{[e]}.$$ Since $\vert\Xi_{\xi}\vert^{\frac{1}{2}}\vert\Xi_{\xi'}\vert^{\frac{1}{2}}$ and $\theta$ are $G$-invariant, it follows that
	$$\vert\Xi_{\xi}\vert^{\frac{1}{2}}\vert\Xi_{\xi'}\vert^{\frac{1}{2}}=(2\pi)^m\bigg(\prod_{\alpha\in\Phi_{+}}\alpha(\xi)\alpha(\xi')\bigg)^{\frac{1}{2}}\theta.$$ We conclude that
	\begin{equation}\label{Equation: Conclusion}\mathrm{Vol}(G/T,\vert\Xi_{\xi}\vert^{\frac{1}{2}}\vert\Xi_{\xi'}\vert^{\frac{1}{2}})=(2\pi)^m\bigg(\prod_{\alpha\in\Phi_{+}}\alpha(\xi)\alpha(\xi')\bigg)^{\frac{1}{2}}\mathrm{Vol}(G/T,\theta),\end{equation} where $\mathrm{Vol}(G/T,\theta)$ is the volume of $G/T$ with respect to $\theta$. At the same time, Equation (20) in \cite{MeinrenkenWoodward} implies that
	\begin{equation}\label{Equation: Conclusion'}\mathrm{Vol}(G/T,\theta)=\frac{1}{(2\pi)^m\prod_{\alpha\in\Phi_{+}}\langle\alpha,\rho\rangle}.\end{equation} The desired result now follows from \eqref{Equation: Conclusion} and \eqref{Equation: Conclusion'}.
\end{proof}

\subsection{Haar measures on maximal tori}\label{Subsection: Haar}
Let $T'\subseteq G$ be any maximal torus, possibly different from the maximal torus $T$ fixed throughout this manuscript. Write $\Theta_{T'}\in\Omega^r(T')$ for the left-invariant $r$-form that induces the Haar measure on $T'$. The value of $\Theta_{T'}$ at the identity is then a vector
$$(\Theta_{T'})_{e}\in\wedge^r((\mathfrak{t'})^*),$$
where $\mathfrak{t}'$ is the Lie algebra of $T'$.

\begin{proposition}
	Suppose that $T',T''\subseteq G$ are maximal tori with respective Lie algebras $\mathfrak{t}',\mathfrak{t}''\subseteq\mathfrak{g}$. If $\{\zeta_1',\ldots,\zeta_r'\}$ and $\{\zeta_1'',\ldots,\zeta_r''\}$ are orthonormal bases of $\mathfrak{t}'$ and $\mathfrak{t}''$, respectively, then
	$$\vert(\Theta_{T'})_e(\zeta_1',\ldots,\zeta_r')\vert=\vert(\Theta_{T''})_e(\zeta_1'',\ldots,\zeta_r'')\vert.$$  
\end{proposition}

\begin{proof}
	Choose an element $g\in G$ satisfying $T''=gT'g^{-1}$. It follows that conjugation by $g$ defines a Lie group isomorphism $T'\overset{\cong}\longrightarrow T''$, so that
	\begin{equation}\label{Equation: Pre}(\Theta_{T'})_e(\zeta_1',\ldots,\zeta_r')=(\Theta_{T''})_e(\mathrm{Ad}_g(\zeta_1'),\ldots,\mathrm{Ad}_g(\zeta_r')).\end{equation} We also note that $\{\mathrm{Ad}_g(\zeta_1'),\ldots,\mathrm{Ad}_g(\zeta_r')\}$ is an orthonormal basis of $\mathfrak{t}''$, owing to the fact that $\langle\cdot,\cdot\rangle$ is $G$-invariant. The bases $\{\mathrm{Ad}_g(\zeta_1'),\ldots,\mathrm{Ad}_g(\zeta_r')\}$ and $\{\zeta_1'',\ldots,\zeta_r''\}$ are therefore related by a linear automorphism $\mathfrak{t}''\longrightarrow\mathfrak{t''}$ having determinant $1$ or $-1$, implying that
	\begin{equation}\label{Equation: Pre2}\vert(\Theta_{T''})_e(\mathrm{Ad}_g(\zeta_1'),\ldots,\mathrm{Ad}_g(\zeta_r'))\vert=\vert(\Theta_{T''})_e(\zeta_1'',\ldots,\zeta_r'')\vert.\end{equation} The desired result now follows from \eqref{Equation: Pre} and \eqref{Equation: Pre2}.
\end{proof}

One consequence is that the positive real number
$$\kappa(G):=\vert(\Theta_{T'})_e(\zeta_1',\ldots,\zeta_r')\vert$$ does not depend on the choice of maximal torus $T'\subseteq G$ and orthonormal basis $\{\zeta_1',\ldots,\zeta_r'\}$ of the Lie algebra of $T'$.

\subsection{The phase function on $\Lambda_{\mathcal{O}}$}\label{Subsection: Phase}
Let $\mathcal{O}\subseteq\mathfrak{g}$ be a regular, integral adjoint orbit and recall the Lagrangian submanifold $\Lambda_{\mathcal{O}}\subseteq T^*G$ from Section \ref{Subsection: Category}. Note that if $(g,\xi)\in\Lambda_{\mathcal{O}}$, then $G_{\xi}$ is a maximal torus in $G$. It follows that the exponential map $\exp:\mathfrak{g}\longrightarrow G$ restricts to a surjective group morphism $\mathfrak{g}_{\xi}\longrightarrow G_{\xi}$. We may therefore choose $\eta\in\mathfrak{g}_{\xi}$ satisfying $\exp(\eta)=g$ and define $$\psi_{\mathcal{O}}(g,\xi):=e^{2\pi i\langle\xi,\eta\rangle}\in S^1\subseteq\mathbb{C}.$$ Our integrality hypotheses implies that $\psi_{\mathcal{O}}(g,\xi)$ is independent of $\eta$, and we thereby obtain a well-defined map
$$\psi_{\mathcal{O}}:\Lambda_{\mathcal{O}}\longrightarrow S^1.$$
Guillemin and Sternberg show $\psi_{\mathcal{O}}$ to be a phase function for the Lagrangian submanifold $\Lambda_{\mathcal{O}}\subseteq T^*G$.

Now recall the $G$-action on $T^*G$ defined in \eqref{Equation: Action on cotangent}. This action leaves $\Lambda_{\mathcal{O}}$ invariant, in which context we have the following result.

\begin{lemma}
If $\mathcal{O}\subseteq\mathfrak{g}$ is a regular, integral adjoint orbit, then the phase function $\psi_{\mathcal{O}}:\Lambda_{\mathcal{O}}\longrightarrow S^1$ is $G$-invariant.
\end{lemma}

\begin{proof}
	Suppose that $(g,\xi)\in\Lambda_{\mathcal{O}}$ and $h\in G$. Choose $\eta\in\mathfrak{g}_{\xi}$ for which $g=\exp(\eta)$, and consider the vectors $\xi'=\mathrm{Ad}_h(\xi)$ and $\eta'=\mathrm{Ad}_h(\eta)$. We then have $\eta'\in\mathfrak{g}_{\xi'}$ and $\exp(\eta')=hgh^{-1}$.  It follows that
	\begin{align*}
	\psi_{\mathcal{O}}(h\cdot(g,\xi)) & = \psi_{\mathcal{O}}(hgh^{-1},\xi')\\
	& = e^{2\pi i\langle\xi',\eta'\rangle}\\
	& = e^{2\pi i\langle\xi,\eta\rangle}\\
	& = \psi_{\mathcal{O}}(g,\xi),\\
	\end{align*}
	where the second-last line is a consequence of $\langle\cdot,\cdot\rangle$ being $G$-invariant. This completes the proof.
\end{proof}

\subsection{The half-density on $\Lambda_{\mathcal{O}}$}\label{Subsection: Half-density}
We now describe the half-density on $\Lambda_{\mathcal{O}}$ constructed in \cite{GuilleminSternbergBook}, where $\mathcal{O}\subseteq\mathfrak{g}$ is a regular adjoint orbit. To this end, recall the $G$-action on $\Lambda_{\mathcal{O}}$ mentioned in Section \ref{Subsection: Phase}. Let $G(g,\xi)\subseteq\Lambda_{\mathcal{O}}$ denote the $G$-orbit of $(g,\xi)\in\Lambda_{\mathcal{O}}$; this is consistent with the notation for $G$-orbits explained in Section \ref{Subsection: The basics}.

\begin{lemma}\label{Lemma: Regular orbit}
	If $\mathcal{O}\subseteq\mathfrak{g}$ is a regular adjoint orbit, then
	\begin{equation}\label{Equation: Tangent space decomposition}T_{(g,\xi)}\Lambda_{\mathcal{O}}=T_{(g,\xi)}(G(g,\xi))\oplus T_{(g,\xi)}(G_{\xi}\times\{\xi\})\end{equation} for all $(g,\xi)\in\Lambda_{\mathcal{O}}$. The $G$-action on $\Lambda_{\mathcal{O}}$ respects these tangent space decompositions. 
\end{lemma}

\begin{proof}
	Since $\mathcal{O}$ is a regular, the group $G_{\xi}$ is abelian. A straightforward calculation then reveals that $G_{\xi}$ is the $G$-stabilizer of $(g,\xi)\in\Lambda_{\mathcal{O}}$. It follows that $T_{(g,\xi)}(G(g,\xi))$ has dimension $n-\dim(G_{\xi})$, while we see that $T_{(g,\xi)}(G_{\xi}\times\{\xi\})$ has dimension $\dim(G_{\xi})$. We also know $T_{(g,\xi)}\Lambda_{\mathcal{O}}$ to have dimension $n$, as $\Lambda_{\mathcal{O}}$ is a Lagrangian submanifold of $T^*G$. In light of these last two sentences, $$T_{(g,\xi)}\Lambda_{\mathcal{O}}=T_{(g,\xi)}(G(g,\xi))\oplus T_{(g,\xi)}(G_{\xi}\times\{\xi\})$$ holds if and only if $$T_{(g,\xi)}(G(g,\xi))\cap T_{(g,\xi)}(G_{\xi}\times\{\xi\})=\{0\}.$$
	
	Note that each vector in $T_{(g,\xi)}(G(g,\xi))$ has the form $((\eta_G)_g,(\eta_{\mathcal{O}})_{\xi})$ for $\eta\in\mathfrak{g}$, where $\eta_G$ (resp. $\eta_{\mathcal{O}}$) denotes the fundamental vector field for $\eta$ with respect to the conjugation action (resp. adjoint action) of $G$ on $G$ (resp. $\mathcal{O}$). With this in mind, let $\eta\in\mathfrak{g}$ be such that $$((\eta_G)_g,(\eta_{\mathcal{O}})_{\xi})\in T_{(g,\xi)}(G_{\xi}\times\{\xi\}).$$ It follows that $(\eta_{\mathcal{O}})_{\xi}=0$, while a straightforward calculation reveals that $(\eta_{\mathcal{O}})_{\xi}=[\eta,\xi]$. We conclude that $\eta\in\mathfrak{g}_{\xi}$. Since $G_{\xi}$ is abelian and contains $g$, this is easily seen to imply that $(\eta_G)_g=0$. Hence $$((\eta_G)_g,(\eta_{\mathcal{O}})_{\xi})=(0,0),$$ and one deduces that $T_{(g,\xi)}(G(g,\xi))$ and $T_{(g,\xi)}(G_{\xi}\times\{\xi\})$ intersect trivially. In light of the previous paragraph, we have $$T_{(g,\xi)}\Lambda_{\mathcal{O}}=T_{(g,\xi)}(G(g,\xi))\oplus T_{(g,\xi)}(G_{\xi}\times\{\xi\}).$$
	
	It remains to prove that the $G$-action on $\Lambda_{\mathcal{O}}$ respects our tangent space decompositions. We thus suppose that $(g,\xi)\in\Lambda_{\mathcal{O}}$ and $h\in G$, and we set $$(g',\xi'):=(hgh^{-1},\mathrm{Ad}_h(\xi)).$$ Our task is to verify that the action of $h$ sends $T_{(g,\xi)}(G(g,\xi))$ and $T_{(g,\xi)}(G_{\xi}\times\{\xi\})$ to $T_{(g',\xi')}(G(g',\xi'))$ and $T_{(g',\xi')}(G_{\xi'}\times\{\xi'\})$, respectively. But this follows from the observation that $h$ sends $G (g,\xi)$ and $G_{\xi}\times\{\xi\}$ to $G (g',\xi')$ and $hG_{\xi}h^{-1}\times\{\mathrm{Ad}_h(\xi)\}=G_{\xi'}\times\{\xi'\}$, respectively. 
\end{proof}

This result has the following immediate implication: if one has a half-density on each of $T_{(g,\xi)}(G (g,\xi))$ and $T_{(g,\xi)}(G_{\xi}\times\{\xi\})$, then one also has a half-density on $T_{(g,\xi)}\Lambda_{\mathcal{O}}$ (see \eqref{Equation: Half-density iso}). This principle gives rise to a half-density $\rho_{\mathcal{O}}$ on $\Lambda_{\mathcal{O}}$, as we now explain. 

Let $\mathcal{O}\subseteq\mathfrak{g}$ be a regular adjoint orbit and fix $(g,\xi)\in\Lambda_{\mathcal{O}}$. Recall that $G_{\xi}$ is the $G$-stabilizer of $(g,\xi)$ in $\Lambda_{\mathcal{O}}$, as explained in the proof above. One straightforward consequence is that 
\begin{equation}\label{Equation: Orbit isomorphism}
G(g,\xi)\longrightarrow\mathcal{O},\quad (h,\eta)\mapsto\eta
\end{equation}
defines a $G$-equivariant diffeomorphism. This diffeomorphism identifies $T_{(g,\xi)}(G (g,\xi))$ with $T_{\xi}\mathcal{O}$, and we recall that the symplectic structure on $\mathcal{O}$ induces a half-density $\vert\Omega_{\xi}\vert^{\frac{1}{2}}$ on $T_{\xi}\mathcal{O}\cong\mathfrak{g}_{\xi}^{\perp}$. The tangent space $T_{(g,\xi)}(G (g,\xi))$ thereby inherits a half-density. We also have an obvious identification of $T_{(g,\xi)}(G_{\xi}\times\{\xi\})$ with $T_gG_{\xi}$, and the Haar measure on $G_{\xi}$ induces a half-density $\vert(\Theta_{G_{\xi}})_g\vert^{\frac{1}{2}}$ on $T_gG_{\xi}$ (see Section \ref{Subsection: Haar}). These considerations produce a half-density on $T_{(g,\xi)}(G_{\xi}\times\{\xi\})$, and it combines with that on $T_{(g,\xi)}(G(g,\xi))$ to yield a half-density on the direct sum $T_{(g,\xi)}(G(g,\xi))\oplus T_{(g,\xi)}(G_{\xi}\times\{\xi\})=T_{(g,\xi)}\Lambda_{\mathcal{O}}$ (see \eqref{Equation: Half-density iso}). This new half-density is precisely $\rho_{\mathcal{O}}$ at the point $(g,\xi)$.

This next result uses the notation and discussion in Sections \ref{Subsection: Adjoint} and \ref{Subsection: Haar}. 

\begin{proposition}\label{Proposition: Half-density description}
	Let $\mathcal{O}\subseteq\mathfrak{g}$ be a regular adjoint orbit and suppose that $(g,\xi)\in\Lambda_{\mathcal{O}}$. Suppose also that $\{\eta_1,\ldots,\eta_{n-r}\}$ and $\{\zeta_1,\ldots,\zeta_r\}$ are bases of $\mathfrak{g}_{\xi}^{\perp}$ and $\mathfrak{g}_{\xi}$, respectively. Write $\overline{\eta_j}$ for the fundamental vector field on $T^*G$ associated to $\eta_j$ via the $G$-action \eqref{Equation: Action on cotangent}, i.e. $$\overline{\eta_j}:=(\eta_j)_{T^*G}$$ for all $j\in\{1,\ldots,n-r\}$. Let us also define the vector $$(\widehat{\zeta_k})_{(g,\xi)}:=((dL_g)_e(\zeta_k),0)\in T_gG_{\xi}\oplus\{0\}=T_{(g,\xi)}(G_{\xi}\times\{\xi\})\subseteq T_{(g,\xi)}\Lambda_{\mathcal{O}}$$ for each $k\in\{1,\ldots,r\}$, where $L_g:G\longleftrightarrow G$ denotes left translation by $g$. The following statements then hold.
	\begin{itemize}
		\item[(i)] The sets
		$$\{(\overline{\eta_1})_{(g,\xi)},\ldots,(\overline{\eta_{n-r}})_{(g,\xi)}\}\quad\text{and}\quad \{(\widehat{\zeta_1})_{(g,\xi)},\ldots,(\widehat{\zeta_r})_{(g,\xi)}\}$$ are bases of $T_{(g,\xi)}(G (g,\xi))$ and $T_{(g,\xi)}(G_{\xi}\times\{\xi\})$, respectively.
		\item[(ii)] The set
		$$\{(\overline{\eta_1})_{(g,\xi)},\ldots,(\overline{\eta_{n-r}})_{(g,\xi)},(\widehat{\zeta_1})_{(g,\xi)},\ldots,(\widehat{\zeta_r})_{(g,\xi)}\}$$ is a basis of $T_{(g,\xi)}\Lambda_{\mathcal{O}}$.
		\item[(iii)] The value of $(\rho_{\mathcal{O}})_{(g,\xi)}$ on this basis is
		$$\kappa(G)^{\frac{1}{2}}\cdot\big\vert\Omega_{\xi}\big(\eta_1,\ldots,\eta_{n-r}\big)\big\vert^{\frac{1}{2}}.$$
	\end{itemize} 
\end{proposition}

\begin{proof}
	We begin by proving (i). To this end, recall that the isomorphism $$T_{(g,\xi)}(G(g,\xi))\overset{\cong}\longrightarrow T_{\xi}\mathcal{O}$$ obtained by differentiating \eqref{Equation: Orbit isomorphism} at $\xi$. This isomorphism sends  $(\overline{\eta_1})_{(g,\xi)},\ldots,(\overline{\eta_{n-r}})_{(g,\xi)}$ to $\eta_1,\ldots,\eta_{n-r}\in\mathfrak{g}_{\xi}^{\perp}\cong T_{\xi}\mathcal{O}$, respectively. It follows that $\{(\overline{\eta_1})_{(g,\xi)},\ldots,(\overline{\eta_{n-r}})_{(g,\xi)}\}$ is a basis of $T_{(g,\xi)}(G (g,\xi))$. On the other hand, $\{\zeta_1,\ldots,\zeta_r\}$ being a basis of $\mathfrak{g}_{\xi}$ implies that $\{(dL_g)_e(\zeta_1),\ldots,(dL_g)_e(\zeta_r)\}$ is a basis of $T_gG_{\xi}$. Under the obvious identification $$T_gG_{\xi}\cong T_{(g,\xi)}(G_{\xi}\times\{\xi\}),$$ this basis corresponds to the ordered set $$\{(\widehat{\zeta_1})_{(g,\xi)},\ldots,(\widehat{\zeta_r})_{(g,\xi)}\}\subseteq T_{(g,\xi)}(G_{\xi}\times\{\xi\}).$$ We conclude that $\{(\widehat{\zeta_1})_{(g,\xi)},\ldots,(\widehat{\zeta_r})_{(g,\xi)}\}$ is a basis of $T_{(g,\xi)}(G_{\xi}\times\{\xi\})$, completing our proof of (i).
	
	Part (ii) follows from (i) and Lemma \ref{Lemma: Regular orbit}. To prove (iii), recall our proof of (i) and the definition of $(\rho_{\mathcal{O}})_{(g,\xi)}$ given in the paragraph preceding this proposition. These considerations imply that $(\rho_{\mathcal{O}})_{(g,\xi)}$ takes the value
	$$\big\vert\Omega_{\xi}\big(\eta_1,\ldots,\eta_{n-r}\big)\big\vert^{\frac{1}{2}}\big\vert(\Theta_{G_{\xi}})_g\big((dL_g)_e(\zeta_1),\ldots,(dL_g)_e(\zeta_r)\big)\big\vert^{\frac{1}{2}}$$ on the indicated basis.
\end{proof}

We have the following additional property of $\rho_{\mathcal{O}}$.

\begin{proposition}\label{Proposition: G-invariance of half-density}
	If $\mathcal{O}\subseteq\mathfrak{g}$ is a regular adjoint orbit, then $\rho_{\mathcal{O}}$ is invariant under the $G$-action on $\Lambda_{\mathcal{O}}$.
\end{proposition}

\begin{proof}
	Suppose that $(g,\xi)\in\Lambda_{\mathcal{O}}$ and $h\in G$, and set $$(g',\xi'):=(hgh^{-1},\mathrm{Ad}_h(\xi)).$$ We now make three remarks. The first requires us to equip $T_gG_{\xi}$ and $T_{g'}G_{\xi'}$ with the half-densities arising from the Haar measures on $G_{\xi}$ and $G_{\xi'}$, respectively. Since conjugation by $h$ defines a group isomorphism from $G_{\xi}$ to $G_{\xi'}$, it identifies the half-density on $T_{g'}G_{\xi'}$ with that on $T_gG_{\xi}$. To begin our second remark, we note that $h$ preserves the symplectic form on $\mathcal{O}$. It follows that $h$ identifies the half-density on $T_{\xi'}\mathcal{O}$ with that on $T_{\xi}\mathcal{O}$, where these two half-densities are induced by the symplectic form on $\mathcal{O}$. Our final remark is that $h$ respects the tangent space decompositions \eqref{Equation: Tangent space decomposition}, a consequence of Lemma \ref{Lemma: Regular orbit}. By the description of $\rho_{\mathcal{O}}$ given in Proposition \ref{Proposition: Half-density description}(iii), these three remarks show $h$ to preserve $\rho_{\mathcal{O}}$. We conclude that $\rho_{\mathcal{O}}$ is $G$-invariant. 
\end{proof}

\section{Some quasi-Hamiltonian geometry}\label{Section: Quasi-Hamiltonian}  
We now discuss the salient parts of quasi-Hamiltonian geometry \cite{AMM}, emphasizing that of the universal centralizers $\Lambda_{\mathcal{C}}$.

\subsection{Quasi-Hamiltonian $G$-spaces}  
Suppose that $M$ is a $G$-manifold and write $g\cdot m\in M$ for the action of $g\in G$ on a point $m\in M$. Each $\xi\in\mathfrak{g}$ then determines a \textit{fundamental vector field} $\xi_M$ on $M$, defined by
$$(\xi_M)_m:=\frac{d}{dt}\bigg\vert_{t=0}\big(\exp(t\xi)\cdot m\big),\quad m\in M.$$ 

Now let $M$ be an arbitrary manifold and consider $\Omega^1(M;\mathfrak{g})$, the real vector space of $\mathfrak{g}$-valued differential $1$-forms on $M$. We have a bilinear pairing $$(\cdot,\cdot):\Omega^1(M;\mathfrak{g})\otimes_{\mathbb{R}}\Omega^1(M;\mathfrak{g})\longrightarrow\Omega^2(M)$$
defined by $$(\alpha,\beta)_m(v_1,v_2)=\langle\alpha_m(v_1),\beta_m(v_2)\rangle-\langle\alpha_m(v_2),\beta_m(v_1)\rangle$$ for all $\alpha,\beta\in\Omega^1(M;\mathfrak{g})$, $m\in M$, and $v_1,v_2\in T_mM$. In the case $M=G$, we let $\theta^L,\theta^R\in\Omega^1(G;\mathfrak{g})$ denote the left and right-invariant Maurer--Cartan forms, respectively. One then has
$$(\theta^L)_g((dL_g)_e(\xi))=\xi\quad\text{and}\quad(\theta^R)_g((dR_g)_e(\xi))=\xi$$
for all $g\in G$ and $\xi\in\mathfrak{g}$, where $L_g:G\longrightarrow G$ and $R_g:G\longrightarrow G$ are left and right translation by $g$, respectively.

\begin{definition}\label{Definition: q-Hamiltonian G-space}
A \textit{quasi-Hamiltonian} $G$\textit{-space} consists of a $G$-manifold $M$, a $G$-invariant $2$-form $\omega\in\Omega^2(M)$, and a smooth map $\mu:M\longrightarrow G$, subject to the following conditions:
\begin{itemize}
	\item[(i)] $d\omega=-\mu^*\chi$, where $\chi\in\Omega^3(G)$ is the Cartan $3$-form;
	\item[(ii)] $\iota_{\xi_M}\omega=\frac{1}{2}\langle\theta^L+\theta^R,\xi\rangle$;
	\item[(iii)] $\mathrm{ker}(\omega_m)=\{(\xi_M)_m:\xi\in\mathrm{ker}(\mathrm{Ad}_{\mu(m)}+\mathrm{id}_{\mathfrak{g}})\}$ for all $m\in M$;
	\item[(iv)] $\mu$ is $G$-equivariant with respect to the given action on $M$ and the conjugation action on $G$.
\end{itemize}
One refers to $\mu$ and $\omega$ as the moment map and quasi-Hamiltonian form, respectively.
\end{definition}

Suppose that this definition is satisfied and write $$M^{\circ}:=\{m\in M:\mathrm{ker}(\omega_m)=\{0\}\}$$ for the open, $G$-invariant, even-dimensional submanifold of $M$ on which $\omega$ is non-degenerate. Note that $\omega$ need not restrict to a symplectic form on $M^{\circ}$; this restricted form need not be closed. Nevertheless, $(T_mM,\omega_m)$ is a symplectic vector space for all $m\in M^{\circ}$. 

\begin{remark}\label{Remark: Non-degeneracy}
We also observe that $\mu^{-1}(e)\subseteq M^{\circ}$, as follows from the condition (iii).   
\end{remark}

\begin{definition}\label{Definition: quasi-Hamiltonian Lagrangian}
Consider the quasi-Hamiltonian $G$-space described in Definition \ref{Definition: q-Hamiltonian G-space}. Let $L\subseteq M$ be a submanifold with inclusion map $j:L\longrightarrow M$. We call $L$ an \textit{isotropic submanifold} of $M$ if $j^*\omega=0$. If $L\subseteq M^{\circ}$ and $\dim(L)=\frac{1}{2}\dim(M^{\circ})$ also hold, we call $L$ a \textit{quasi-Hamiltonian Lagrangian submanifold} of $M$. 
\end{definition}

\subsection{The internally fused double}
Examples of quasi-Hamiltonian $G$-spaces include the so-called \textit{internally fused double} $D(G)$, a construction described in \cite{AMM}. To obtain it, one first notes that $D(G)=G\times G$ as a manifold. Now let $G$ act on $D(G)$ via
$$g\cdot(h,k):=(ghg^{-1},gkg^{-1}),\quad g\in G,\text{ }(h,k)\in D(G).$$ The map $$\mu_{D(G)}:D(G)\longrightarrow G,\quad (g,h)\mapsto ghg^{-1}h^{-1},\quad (g,h)\in D(G)$$ is then $G$-equivariant with respect to the conjugation action on $G$. Let us also consider the two-form $\omega_{D(G)}\in\Omega^2(D(G))$ defined by
$$\omega_{D(G)}:=\frac{1}{2}\bigg( \pi_1^*\theta^L,\pi_2^*\theta^R\bigg)+\frac{1}{2}\bigg(\pi_1^*\theta^R,\pi_2^*\theta^L\bigg)+\frac{1}{2}\bigg( (\pi_1\pi_2)^*\theta^L,(\pi_1^{-1}\pi_2^{-1})^*\theta^R\bigg),$$ where $\pi_1,\pi_2:D(G)\longrightarrow G$ are the projections onto the first and second factors, respectively, and $\pi_1\pi_2,\pi_1^{-1}\pi_2^{-1}:D(G)\longrightarrow G$ are defined by $$(\pi_1\pi_2)(g,h)=gh,\quad (\pi_1^{-1}\pi_2^{-1})(g,h)=g^{-1}h^{-1},\quad (g,h)\in D(G).$$ A straightforward calculation shows $\omega$ to be $G$-invariant, an instance of the following more substantial result (cf. \cite[Example 6.1]{AMM}). 

\begin{proposition}
The $G$-manifold $D(G)$, two-form $\omega_{D(G)}$, and map $\mu_{D(G)}$ constitute a quasi-Hamiltonian $G$-space.
\end{proposition}

One calls this quasi-Hamiltonian $G$-space the \textit{internally fused double of} $G$.

\subsection{The gauge-theoretic realization of $D(G)$}\label{Subsection: Gauge theory}
We now recall a gauge-theoretic realization of $D(G)$ described in \cite{AMM}. The reader will find a more comprehensive account in Section 9 of the aforementioned reference. 

Let $\Sigma$ denote a one-holed torus and fix a real number $\lambda>1$. Write $\mathcal{A}_{\text{flat}}(\Sigma)$ for the Banach manifold of flat connections on the trivial principal $G$-bundle $\Sigma\times G\longrightarrow\Sigma$ of Sobolev class $\lambda$, and write $\mathcal{G}(\Sigma)$ for the Banach Lie group of all maps $\Sigma\longrightarrow G$ having Sobolev class
$\lambda+1$. This group acts on $\mathcal{A}_{\text{flat}}(\Sigma)$ in a canonical way (see \cite[Section 9.1]{AMM}), and one calls it the \textit{gauge group}. Now fix a point $x$ on the boundary of $\Sigma$, and consider the \textit{restricted gauge group} 
$$\mathcal{G}(\Sigma)_{\text{res}}:=\{f\in\mathcal{G}(\Sigma):f(x)=e\}.$$ This subgroup of $\mathcal{G}(\Sigma)$ acts freely on $\mathcal{A}_{\text{flat}}(\Sigma)$, and the quotient
$$\mathcal{M}(\Sigma):=\mathcal{A}_{\text{flat}}(\Sigma)/G(\Sigma)_{\text{res}}$$ is a finite-dimensional manifold. The action of $\mathcal{G}(\Sigma)$ on $\mathcal{A}_{\text{flat}}(\Sigma)$ descends to an action of $\mathcal{G}(\Sigma)/\mathcal{G}(\Sigma)_{\text{res}}\cong G$ on $M(\Sigma)$. One then has the following specialized version of \cite[Theorem 9.1]{AMM}.

\begin{theorem}\label{Theorem: Canonical way}
	The $G$-manifold $\mathcal{M}(\Sigma)$ is a quasi-Hamiltonian $G$-space in a canonical way.
\end{theorem}

Now choose loops $\gamma_1,\gamma_2:[0,1]\longrightarrow\Sigma$ at $x$ that satisfy the following properties:
\begin{itemize}
	\item[(i)] the images of $\gamma_1$ and $\gamma_2$ intersect only at $x$;
	\item[(ii)] the homotopy classes $[\gamma_1]$ and $[\gamma_2]$ freely generate $\pi_1(\hat{\Sigma},x)$ as an abelian group, where $\hat{\Sigma}$ is the torus obtained by capping off the boundary of $\Sigma$;
	\item[(iii)] we have $[\gamma_1][\gamma_2][\gamma_1]^{-1}[\gamma_2]^{-1}=[\delta]$ in $\pi_1(\Sigma,x)$, where $[\delta]$ is a generator of $\pi_1(\partial\Sigma,x)$.    
\end{itemize}

These choices give rise to a well-defined map
$$\Phi_{\gamma_1,\gamma_2}:M(\Sigma)\longrightarrow D(G),\quad [\nabla]\mapsto (\mathrm{Hol}_{\gamma_1}(\nabla),\mathrm{Hol}_{\gamma_2}(\nabla)),\quad [\nabla]\in\mathcal{M}(\Sigma),$$
where $\mathrm{Hol}_{\gamma_1}(\nabla),\mathrm{Hol}_{\gamma_2}(\nabla)\in G$ are the holonomies of $\nabla\in\mathcal{A}_{\text{flat}}(\Sigma)$ around $\gamma_1$ and $\gamma_2$, respectively. This leads to the following special case of \cite[Theorem 9.3]{AMM}.

\begin{theorem}\label{Theorem: q-Hamiltonian isomorphism}
	Let $\gamma_1,\gamma_2:[0,1]\longrightarrow\Sigma$ be as described above, and equip $\mathcal{M}(\Sigma)$ with the quasi-Hamiltonian $G$-space structure from Theorem \ref{Theorem: Canonical way}. The map $\Phi_{\gamma_1,\gamma_2}: \mathcal{M}(\Sigma)\longrightarrow D(G)$ is then an isomorphism of quasi-Hamiltonian $G$-spaces. 
\end{theorem}

Now let $\mathrm{Diff}^+(\Sigma)$ denote the group of all orientation-preserving diffeomorphisms $\Sigma\longrightarrow\Sigma$, endowed with the $C^1$-topology. Consider the subgroup $$\mathrm{Diff}^+(\Sigma)_{\text{res}}:=\{\phi\in \mathrm{Diff}^+(\Sigma):\phi(x)=x\}$$ and its identify component $(\mathrm{Diff}^+(\Sigma)_{\text{res}})^{\circ}$. The quotient group $$\Gamma(\Sigma):=\mathrm{Diff}^+(\Sigma)_{\text{res}}/(\mathrm{Diff}^+(\Sigma)_{\text{res}})^{\circ}$$ is then called the \textit{mapping class group}. One has a well-defined action of $\Gamma(\Sigma)$ on $\mathcal{M}(\Sigma)$ given by
$$[\phi]\cdot[\nabla]:=[(\phi^{-1})^*\nabla],\quad[\phi]\in\Gamma(\Sigma),\text{ }[\nabla]\in\mathcal{M}(\Sigma),$$
where $(\phi^{-1})^*\nabla$ denotes the pullback of $\nabla$ along the diffeomorphism $\phi^{-1}:\Sigma\longrightarrow\Sigma$. This gives context for the following result, which is explained in \cite[Section 9.4]{AMM}.
\begin{proposition}
	The above-defined action of $\Gamma(\Sigma)$ respects the quasi-Hamiltonian $G$-space structure on $\mathcal{M}(\Sigma)$. 
\end{proposition}

The action of $\mathrm{Diff}^+(\Sigma)_{\text{res}}$ on $\pi_1(\Sigma,x)$ descends to an action of $\Gamma(\Sigma)$ by group automorphisms. This is amounts to $$\chi:\Gamma(\Sigma)\longrightarrow\mathrm{Aut}(\pi_1(\Sigma,x)),\quad [\phi]\mapsto\phi_*,\quad [\phi]\in\Gamma(\Sigma)$$ being a well-defined group morphism, where $\phi_*:\pi_1(\Sigma,x)\longrightarrow\pi_1(\Sigma,x)$ is the map of fundamental groups induced by $\phi\in\mathrm{Diff}^+(\Sigma)_{\text{res}}$. One then has the following well-known fact (cf. \cite[Section 1.1.1]{Goldman}).

\begin{proposition}
	The group morphism $\chi:\Gamma(\Sigma)\longrightarrow\mathrm{Aut}(\pi_1(\Sigma,x))$ is injective.
\end{proposition}

Now let $\gamma_1$ and $\gamma_2$ be as chosen earlier. The \textit{Dehn twist} about $\gamma_1$ is then the unique element $[\vartheta]\in\Gamma(\Sigma)$ that satisfies
\begin{equation}\label{Equation: Fundamental relations}\vartheta_*([\gamma_1])=[\gamma_1]\quad\text{and}\quad \vartheta_*([\gamma_2])=[\gamma_2][\gamma_1].\end{equation} Let $S:\mathcal{M}(\Sigma)\longrightarrow\mathcal{M}(\Sigma)$ be the diffeomorphism through which $[\vartheta]$ acts on $\mathcal{M}(\Sigma)$, i.e.
$$S([\nabla])=[\vartheta]\cdot[\nabla],\quad[\nabla]\in\mathcal{M}(\Sigma).$$ If we identify $\mathcal{M}(\Sigma)$ with $D(G)$ via $\Phi_{\gamma_1,\gamma_2}$, then $S$ identifies with a diffeomorphism $R:D(G)\longrightarrow D(G)$. We thereby obtain the commutative diagram  
\begin{equation}\label{Equation: q-Hamiltonian diagram}\begin{tikzcd}
\mathcal{M}(\Sigma)\arrow[d, "S"] \arrow[r, "\Phi_{\gamma_1,\gamma_2}"] & D(G)\arrow[d, "R"] \\
\mathcal{M}(\Sigma)\arrow[r, "\Phi_{\gamma_1,\gamma_2}"]
& D(G)
\end{tikzcd}.\end{equation}

\begin{proposition}
	The diffeomorphism $R:D(G)\longrightarrow D(G)$ is given by
	$$R(g,h)=(g,hg^{-1}),\quad (g,h)\in D(G),$$
	and it is a quasi-Hamiltonian $G$-space automorphism.
\end{proposition}

\begin{proof}
	Since $\Gamma(\Sigma)$ acts on $\mathcal{M}(\Sigma)$ by quasi-Hamiltonian $G$-space automorphisms, $S$ is an automorphism of this type. This observation combines with Theorem \ref{Theorem: q-Hamiltonian isomorphism} and the diagram \eqref{Equation: q-Hamiltonian diagram} to imply that $R$ is a quasi-Hamiltonian $G$-space automorphism.
	
	It remains only to verify that $R(g,h)=(g,hg^{-1})$ for all $(g,h)\in D(G)$. To this end, fix $(g,h)\in D(G)$ and let $[\nabla]\in\mathcal{M}(\Sigma)$ be such that $\Phi_{\gamma_1,\gamma_2}([\nabla])=(g,h)$. This amounts to the conditions
	\begin{equation}\label{Equation: Elements}\mathrm{Hol}_{\gamma_1}(\nabla)=g\quad\text{and}\quad \mathrm{Hol}_{\gamma_2}(\nabla)=h.\end{equation}
	
	We have
	\begin{equation}\label{Equation: Holonomy relations}\mathrm{Hol}_{\gamma_1}(\nabla)=\mathrm{Hol}_{\vartheta\circ\gamma_1}((\vartheta^{-1})^*\nabla)\quad\text{and}\quad \mathrm{Hol}_{\gamma_2}(\nabla)=\mathrm{Hol}_{\vartheta\circ\gamma_2}((\vartheta^{-1})^*\nabla).\end{equation}
	On the other hand, the relations \eqref{Equation: Fundamental relations} imply that
	$$\mathrm{Hol}_{\vartheta\circ\gamma_1}((\vartheta^{-1})^*\nabla)=\mathrm{Hol}_{\gamma_1}((\vartheta^{-1})^*\nabla)\quad\text{and}\quad \mathrm{Hol}_{\vartheta\circ\gamma_2}((\vartheta^{-1})^*\nabla)=\mathrm{Hol}_{\gamma_2}((\vartheta^{-1})^*\nabla)\mathrm{Hol}_{\gamma_1}((\vartheta^{-1})^*\nabla).$$ These combine with \eqref{Equation: Elements} and \eqref{Equation: Holonomy relations} to imply that
	$$\mathrm{Hol}_{\gamma_1}((\vartheta^{-1})^*\nabla)=g\quad\text{and}\quad \mathrm{Hol}_{\gamma_2}((\vartheta^{-1})^*\nabla)=hg^{-1}.$$
	Using this and the commutativity of \eqref{Equation: q-Hamiltonian diagram}, we obtain  
	\begin{align*}
	R(g,h) & = R(\Phi_{\gamma_1,\gamma_2}([\nabla]))\\
	& = \Phi_{\gamma_1,\gamma_2}(S([\nabla]))\\
	& = \Phi_{\gamma_1,\gamma_2}([(\vartheta^{-1})^*\nabla])\\
	& = (\mathrm{Hol}_{\gamma_1}((\vartheta^{-1})^*\nabla), \mathrm{Hol}_{\gamma_2}((\vartheta^{-1})^*\nabla))\\
	& = (g,hg^{-1})
	\end{align*} 
\end{proof}

\subsection{The symplectic submanifold $T\times T$}
Recall our fixed maximal torus $T\subseteq G$ and its Lie algebra $\mathfrak{t}\subseteq\mathfrak{g}$. The manifold $T\times T$ then carries a symplectic form $\omega_{T\times T}\in\Omega^2(T\times T)$, defined as follows on each tangent space $T_{(g,h)}(T\times T)=T_gT\oplus T_hT$:
\begin{equation}\label{Equation: Form on TxT}(\omega_{T\times T})_{(g,h)}\bigg(\bigg((dL_g)_e(\xi_1),(dL_h)_e(\xi_2)\bigg),\bigg((dL_g)_e(\eta_1),(dL_h)_e(\eta_2)\bigg)\bigg)=\langle\xi_1,\eta_2\rangle-\langle\eta_1,\xi_2\rangle\end{equation} for all $\xi_1,\xi_2,\eta_1,\eta_2\in\mathfrak{t}$. This symplectic form turns out to be compatible with the quasi-Hamiltonian form on $D(G)$.

\begin{proposition}\label{Proposition: Pullback}
We have $j^*(\omega_{D(G)})=\omega_{T\times T}$, where $j:T\times T\longrightarrow D(G)$ is the inclusion map.
\end{proposition}

\begin{proof}
Let us fix elements $g,h\in T$ and $\xi_1,\xi_2,\eta_1,\eta_2\in\mathfrak{t}$. Our task is to prove that $$(\omega_{D(G)})_{(g,h)}\bigg(\bigg((dL_g)_e(\xi_1),(dL_h)_e(\xi_2)\bigg),\bigg((dL_g)_e(\eta_1),(dL_h)_e(\eta_2)\bigg)\bigg)=\langle\xi_1,\eta_2\rangle-\langle\eta_1,\xi_2\rangle.$$ To this end, consider the two-forms on $D(G)$ defined by $$\alpha:=\frac{1}{2}\bigg( \pi_1^*\theta^L,\pi_2^*\theta^R\bigg),\quad\beta:=\frac{1}{2}\bigg( \pi_1^*\theta^R,\pi_2^*\theta^L\bigg),\quad\text{and}\quad\gamma:=\frac{1}{2}\bigg( (\pi_1\pi_2)^*\theta^L,(\pi_1^{-1}\pi_2^{-1})^*\theta^R\bigg).$$ One has $\omega_{D(G)}=\alpha+\beta+\gamma$, and it therefore suffices to prove the following identities:
\begin{subequations}
\begin{align}
& \alpha_{(g,h)}\bigg(\bigg((dL_g)_e(\xi_1),(dL_h)_e(\xi_2)\bigg),\bigg((dL_g)_e(\eta_1),(dL_h)_e(\eta_2)\bigg)\bigg)=\frac{1}{2}(\langle\xi_1,\eta_2\rangle-\langle\eta_1,\xi_2\rangle);\label{Equation: First}\\
& \beta_{(g,h)}\bigg(\bigg((dL_g)_e(\xi_1),(dL_h)_e(\xi_2)\bigg),\bigg((dL_g)_e(\eta_1),(dL_h)_e(\eta_2)\bigg)\bigg)=\frac{1}{2}(\langle\xi_1,\eta_2\rangle-\langle\eta_1,\xi_2\rangle);\label{Equation: Second}\\
& \gamma_{(g,h)}\bigg(\bigg((dL_g)_e(\xi_1),(dL_h)_e(\xi_2)\bigg),\bigg((dL_g)_e(\eta_1),(dL_h)_e(\eta_2)\bigg)\bigg)=0.\label{Equation: Third}
\end{align}
\end{subequations}

We begin by verifying \eqref{Equation: First}. First note that $T$ being abelian forces $(dL_k)_e(x)=(dR_k)_e(x)$ to hold for all $k\in T$ and $x\in\mathfrak{t}$. We conclude that \begin{equation}\label{Equation: Useful}(\theta^L)_k((dL_k)_e(x))=x=(\theta^R)_k((dL_k)_e(x))\end{equation} for all $k\in T$ and $x\in\mathfrak{t}$, and this implies the following identities:
\begin{subequations}
\begin{align*}
&(\theta^L)_g((dL_g)_e(\xi_1))=\xi_1;\\
&(\theta^L)_g((dL_g)_e(\eta_1))=\eta_1;\\
&(\theta^R)_h((dL_h)_e(\xi_2))=\xi_2;\\
&(\theta^R)_h((dL_h)_e(\eta_2))=\eta_2.
\end{align*}
\end{subequations}
It follows that 
\begin{align*}
& \alpha_{(g,h)}\bigg(\bigg((dL_g)_e(\xi_1),(dL_h)_e(\xi_2)\bigg),\bigg((dL_g)_e(\eta_1),(dL_h)_e(\eta_2)\bigg)\bigg)\\
& = \frac{1}{2}\bigg(\bigg\langle (\theta^L)_g((dL_g)_e(\xi_1)),(\theta^R)_h((dL_h)_e(\eta_2))\bigg\rangle-\bigg\langle (\theta^L)_g((dL_g)_e(\eta_1)),(\theta^R)_h((dL_h)_e(\xi_2))\bigg\rangle\bigg)\\
& = \frac{1}{2}(\langle\xi_1,\eta_2\rangle-\langle\eta_1,\xi_2\rangle).
\end{align*}
The verification of \eqref{Equation: Second} proceeds analogously.

It remains only to prove \eqref{Equation: Third}. To this end, it is straightforward to verify the following identities: 
\begin{subequations}
\begin{align*}
&d(\pi_1\pi_2)_{(g,h)}\bigg((dL_g)_e(\xi_1),(dL_h)_e(\xi_2)\bigg)=(dL_{gh})_e(\xi_1+\xi_2);\\
&d(\pi_1\pi_2)_{(g,h)}\bigg((dL_g)_e(\eta_1),(dL_h)_e(\eta_2)\bigg)=(dL_{gh})_e(\eta_1+\eta_2);\\
&d(\pi_1^{-1}\pi_2^{-1})_{(g,h)}\bigg((dL_g)_e(\xi_1),(dL_h)_e(\xi_2)\bigg)=-(dL_{g^{-1}h^{-1}})_e(\xi_1+\xi_2);\\
&d(\pi_1^{-1}\pi_2^{-1})_{(g,h)}\bigg((dL_g)_e(\eta_1),(dL_h)_e(\eta_2)\bigg)=-(dL_{g^{-1}h^{-1}})_e(\eta_1+\eta_2).
\end{align*}
\end{subequations}
One may also use \eqref{Equation: Useful} to deduce the following identities:
\begin{subequations}
	\begin{align*}
	&(\theta^L)_{gh}((dL_{gh})_e(\xi_1+\xi_2))=\xi_1+\xi_2;\\
	&(\theta^L)_{gh}((dL_{gh})_e(\eta_1+\eta_2))=\eta_1+\eta_2;\\
	&(\theta^R)_{g^{-1}h^{-1}}((dL_{g^{-1}h^{-1}})_e(\xi_1+\xi_2))=\xi_1+\xi_2;\\
	&(\theta^R)_{g^{-1}h^{-1}}((dL_{g^{-1}h^{-1}})_e(\eta_1+\eta_2))=\eta_1+\eta_2.
	\end{align*}
\end{subequations}
These last two sentences imply that
\begin{align*}
& \gamma_{(g,h)}\bigg(\bigg((dL_g)_e(\xi_1),(dL_h)_e(\xi_2)\bigg),\bigg((dL_g)_e(\eta_1),(dL_h)_e(\eta_2)\bigg)\bigg)\\
&=-\frac{1}{2}\bigg\langle (\theta^L)_{gh}((dL_{gh})_e(\xi_1+\xi_2)),(\theta^R)_{g^{-1}h^{-1}}((dL_{g^{-1}h^{-1}})_e(\eta_1+\eta_2))\bigg\rangle\\
&+\frac{1}{2}\bigg\langle (\theta^L)_{gh}((dL_{gh})_e(\eta_1+\eta_2)),(\theta^R)_{g^{-1}h^{-1}}((dL_{g^{-1}h^{-1}})_e(\xi_1+\xi_2))\bigg\rangle\\
& = -\frac{1}{2}\langle\xi_1+\xi_2,\eta_1+\eta_2\rangle+\frac{1}{2}\langle\eta_1+\eta_2,\xi_1+\xi_2\rangle\\
& = 0,
\end{align*}
completing the proof.
\end{proof}

\subsection{A special instance of quasi-Hamiltonian reduction}\label{Subsection: A special}
Recall the moment map $\mu_{D(G)}:D(G)\longrightarrow G$ and set $D(G)_{\text{reg}}:=G\times G_{\text{reg}}$. Let us also define
$$\mu_{D(G)}^{-1}(e)_{\text{reg}}:=\mu_{D(G)}^{-1}(e)\cap D(G)_{\text{reg}}.$$

\begin{lemma}\label{Lemma: Maximal torus}
If $x\in\mu_{D(G)}^{-1}(e)_{\emph{reg}}$, then $G_x$ is a maximal torus in $G$.
\end{lemma}

\begin{proof}
We have $x=(g,h)$ for some $g\in G$ and $h\in G_{\text{reg}}$ satisfying $ghg^{-1}h^{-1}=e$. It follows that $G_h$ is a maximal torus in $G$ containing $g$, and we conclude that $G_h\subseteq G_g$. Since $G_x=G_g\cap G_h$, we conclude that $G_x=G_h$. This completes the proof.
\end{proof}

This leads to the following additional fact. To this end, recall that $n$ and $r$ denote the dimension and rank of $G$, respectively.

\begin{lemma}\label{Lemma: Differential}
The differential of $\mu_{D(G)}$ has constant rank $n-r$ on $\mu_{D(G)}^{-1}(e)_{\emph{reg}}$.
\end{lemma}

\begin{proof}
Suppose that $(g,h)\in\mu_{D(G)}^{-1}(e)_{\text{reg}}$ and set $x:=(g,h)$ for notational simplicity. An application of \cite[Proposition 4.1]{AMM} then yields
$$\mathrm{ker}((d\mu_{D(G)})_{x})^{(\omega_{D(G)})_x}=\{(\xi_{D(G)})_{x}:\xi\in\mathfrak{g}\},$$ where the the left-hand side denotes the annihilator of $\mathrm{ker}((d\mu_{D(G)})_{x})\subseteq T_x(D(G))$ under $(\omega_{D(G)})_x$. The right-hand side has dimension $n-\dim(G_x)$, which by Lemma \ref{Lemma: Maximal torus} coincides with $n-r$. On the other hand, the non-degeneracy of $(\omega_{D(G)})_x$ (see Remark \ref{Remark: Non-degeneracy}) forces $\mathrm{ker}((d\mu_{D(G)})_{x})^{(\omega_{D(G)})_x}$ to have dimension $\dim(D(G))-\dim(\mathrm{ker}((d\mu_{D(G)})_x)$. It follows that 
$$\dim(D(G))-\dim(\mathrm{ker}((d\mu_{D(G)})_x))=n-r,$$ or equivalently
$$\dim(\mathrm{ker}((d\mu_{D(G)})_x))=n+r.$$
This amounts to $(d\mu_{D(G)})_x$ having rank $n-r$.
\end{proof}

\begin{proposition}\label{Proposition: Submanifold}
The subset $\mu_{D(G)}^{-1}(e)_{\emph{reg}}\subseteq D(G)$ is a locally closed, $G$-invariant submanifold of dimension $n+r$.
\end{proposition}

\begin{proof}
Since $\mu_{D(G)}^{-1}(e)$ and $D(G)_{\text{reg}}$ are $G$-invariant subsets of $D(G)$, the same is true of $\mu_{D(G)}^{-1}(e)_{\text{reg}}=\mu_{D(G)}^{-1}(e)\cap (D(G)_{\text{reg}})$. The rest of this proposition follows immediately from Lemma \ref{Lemma: Differential}.
\end{proof}

The action of $G$ on $\mu_{D(G)}^{-1}(e)_{\text{reg}}$ merits further study. In particular, Lemma \ref{Lemma: Maximal torus} implies that the $G$-stabilizers of any two points in $\mu_{D(G)}^{-1}(e)_{\text{reg}}$ are conjugate. The topological quotient $\mu_{D(G)}^{-1}(e)_{\text{reg}}/G$ therefore carries a unique manifold structure for which the quotient map
$$q:\mu_{D(G)}^{-1}(e)_{\text{reg}}\longrightarrow \mu_{D(G)}^{-1}(e)_{\text{reg}}/G$$  
is a submersion.

\begin{proposition}\label{Proposition: Symplectic quotient}
Let $j:\mu_{D(G)}^{-1}(e)_{\emph{reg}}\longrightarrow D(G)$ denote the inclusion map. The quotient manifold $\mu_{D(G)}^{-1}(e)_{\emph{reg}}/G$ then carries a unique symplectic form $\overline{\omega_{D(G)}}$ for which $q^*\overline{\omega_{D(G)}}=j^*\omega_{D(G)}$.
\end{proposition}

\begin{proof}
The arguments given in the proof of \cite[Proposition 5.1]{AMM} apply in our setting.
\end{proof}

\begin{remark}
The statement of \cite[Proposition 5.1]{AMM} does not apply to the setting of Proposition \ref{Proposition: Symplectic quotient}; it would apply only if the $G$-action on $\mu_{D(G)}^{-1}(e)_{\text{reg}}$ were free. One must examine the proof of \cite[Proposition 5.1]{AMM} to deduce Proposition \ref{Proposition: Symplectic quotient}.
\end{remark}

Now consider the Weyl group $W:=N_G(T)/T$ and its action on $T\times T$ defined by
$$w\cdot (s,t):=(wsw^{-1},wtw^{-1}),\quad w\in W,\text{ }(s,t)\in T\times T.$$ The subset $T\times T_{\text{reg}}\subseteq T\times T$ is $W$-invariant, and we observe that $W$ acts freely on $T\times T_{\text{reg}}$.
We also observe that the inclusion $T\times T_{\text{reg}}\longrightarrow\mu_{D(G)}^{-1}(e)_{\text{reg}}$ descends to a well-defined smooth map
$$\varphi:(T\times T_{\text{reg}})/W\longrightarrow\mu_{D(G)}^{-1}(e)_{\text{reg}}/G.$$At the same time, observe that the symplectic form $\omega_{T\times T}\in\Omega^2(T\times T)$ is $W$-invariant. It therefore induces a symplectic form $\overline{\omega_{T\times T}}$ on $(T\times T_{\text{reg}})/W$. 

\begin{proposition}\label{Proposition: Symplectomorphism}
The map $\varphi:(T\times T_{\emph{reg}})/W\longrightarrow\mu_{D(G)}^{-1}(e)_{\emph{reg}}/G$ is a symplectomorphism from $((T\times T_{\emph{reg}})/W,\overline{\omega_{T\times T}})$ to $(\mu_{D(G)}^{-1}(e)_{\emph{reg}}/G,\overline{\omega_{D(G)}})$.
\end{proposition}

\begin{proof}
We begin by verifying that $\varphi^*(\overline{\omega_{D(G)}})=\overline{\omega_{T\times T}}$. Consider the commutative diagram
$$\begin{tikzcd}
T\times T_{\text{reg}}\arrow[d, "r"] \arrow[r, "k"] & \mu_{D(G)}^{-1}(e)_{\text{reg}}\arrow[r, "j"]\arrow[d, "q"] & D(G) \\
(T\times T_{\text{reg}})/W\arrow[r, "\varphi"]
& \mu_{D(G)}^{-1}(e)_{\text{reg}}/G
\end{tikzcd},$$
where $j$ and $k$ (resp. $q$ and $r$) are the natural inclusions (resp. quotient maps). We have
\begin{align*}
r^*(\varphi^*(\overline{\omega_{D(G)}})) & = k^*(q^*(\overline{\omega_{D(G)}}))\\
& = k^*(j^*(\omega_{D(G)}))\\
& = \ell^*(\omega_{D(G)}),
\end{align*}
where $\ell:T\times T_{\text{reg}}\longrightarrow D(G)$ is the inclusion. Proposition \ref{Proposition: Pullback} implies that $\ell^*(\omega_{D(G)})$ is the restriction of $\omega_{T\times T}\in\Omega^2(T\times T)$ to the open subset $T\times T_{\text{reg}}\subseteq T\times T$, and we know this restriction to be $r^*(\overline{\omega_{T\times T}})$. It follows that $r^*(\varphi^*(\overline{\omega_{D(G)}}))=r^*(\overline{\omega_{T\times T}})$, yielding $\varphi^*(\overline{\omega_{D(G)}})=\overline{\omega_{T\times T}}$.

It remains to verify that $\varphi$ is a diffeomorphism. We first apply Lemma \ref{Lemma: Maximal torus} and Proposition \ref{Proposition: Submanifold} to conclude that
$\mu_{D(G)}^{-1}(e)_{\text{reg}}/G$ is $2r$-dimensional. The manifold $(T\times T_{\text{reg}})/W$ is also $2r$-dimensional, while the condition $\varphi^*(\overline{\omega_{D(G)}})=\overline{\omega_{T\times T}}$ forces $\varphi$ to be an immersion. We conclude that $\varphi$ is a local diffeomorphism, and are thereby reduced to proving that $\varphi$ is bijective.

To establish surjectivity, we suppose that $[(g,h)]\in\mu_{D(G)}^{-1}(e)_{\text{reg}}/G$. Note that $khk^{-1}\in T_{\text{reg}}$ for some $k\in G$, and that $kgk^{-1}$ commutes with $khk^{-1}$. It follows that $kgk^{-1}\in T$ and
$$\varphi([(kgk^{-1},khk^{-1})])=[(g,h)],$$ proving surjectivity.

To verify injectivity, assume that $(s,t),(s',t')\in T\times T_{\text{reg}}$ satisfy $[(s,t)]=[(s',t')]$ in $\mu_{D(G)}^{-1}(e)_{\text{reg}}/G$. This assumption amounts to the existence of $g\in G$ for which $s'=gsg^{-1}$ and $t'=gtg^{-1}$. We may therefore find $w\in W$ with $t'=wtw^{-1}$. Now choose a lift $h\in N_G(T)$ of $w$ and note that
$$gtg^{-1}=t'=hth^{-1}.$$ Hence $$(g^{-1}h)t(g^{-1}h)^{-1}=t,$$ and this combines with the fact that $t\in T_{\text{reg}}$ to yield $g^{-1}h\in T$. It follows that
$$s'=gsg^{-1}=g((g^{-1}h)s(g^{-1}h)^{-1})g^{-1}=hsh^{-1}=wsw^{-1},$$ so that $[(s',t')]=[(wsw^{-1},wtw^{-1})]=[(s,t)]$ in $(T\times T_{\text{reg}})/W$. This shows $\varphi$ to be injective, completing the proof.  
\end{proof}

Fix a point $(s,t)\in T\times T_{\text{reg}}$ and set $x:=(s,t)$. Write $\overline{x}$ and $[x]$ for the equivalence classes of $x$ in $(T\times T_{\text{reg}})/W$ and $\mu_{D(G)}^{-1}(e)_{\text{reg}}/G$, respectively. The differential of $\varphi$ at $\overline{x}$ is then a symplectic vector space isomorphism
$$(d\varphi)_{\overline{x}}:T_{\overline{x}}((T\times T_{\text{reg}})/W)\overset{\cong}\longrightarrow T_{[x]}(\mu_{D(G)}^{-1}(e)_{\text{reg}}/G)$$
with respect to the forms $(\overline{\omega_{T\times T}})_{\overline{x}}$ and $(\overline{\omega_{D(G)}})_{[x]}$. At the same time, we have $$T_{\overline{x}}((T\times T_{\text{reg}})/W)=T_x(T\times T_{\text{reg}})=T_x(T\times T)\quad\text{and}\quad T_{[x]}(\mu_{D(G)}^{-1}(e)_{\text{reg}}/G)=\frac{T_x(\mu_{D(G)}^{-1}(e)_{\text{reg}})}{T_x(Gx)}.$$ Our isomorphism $(d\varphi)_{\overline{x}}$ then takes the form
$$(d\varphi)_{\overline{x}}:T_x(T\times T)\longrightarrow \frac{T_x(\mu_{D(G)}^{-1}(e)_{\text{reg}})}{T_x(Gx)},\quad v\mapsto [v],\quad v\in T_x(T\times T),$$
where $[v]\in \frac{T_x(\mu_{D(G)}^{-1}(e)_{\text{reg}})}{T_x(Gx)}$ denotes the equivalence class of $v\in T_x(T\times T)\subseteq T_x(\mu_{D(G)}^{-1}(e)_{\text{reg}})$. We also note that $(\overline{\omega_{T\times T}})_{\overline{x}}=(\omega_{T\times T})_x$ as bilinear forms on $T_{\overline{x}}((T\times T_{\text{reg}})/W)=T_x(T\times T)$. The condition that $(d\varphi)_{\overline{x}}$ respects symplectic forms then becomes the statement
\begin{equation}\label{Equation: Relation}(\overline{\omega_{D(G)}})_{[x]}([v],[w])=(\omega_{T\times T})_x(v,w)\end{equation} for all $v,w\in T_x(T\times T)$.

\subsection{Universal centralizers of conjugacy classes} 
Recall the Lagrangian submanifolds $\Lambda_{\mathcal{O}}\subseteq T^*G$ discussed in Section \ref{Section: Lie-theoretic constructions}. The quasi-Hamiltonian counterparts of the $\Lambda_{\mathcal{O}}$ are the
$$\Lambda_{\mathcal{C}}:=\{(g,h)\in D(G):h\in\mathcal{C}\text{ and }ghg^{-1}=h\},$$ where $\mathcal{C}\subseteq G$ ranges over the regular conjugacy classes. Note that
$$\Lambda_{\mathcal{C}}\longrightarrow\mathcal{C},\quad (g,h)\mapsto h,\quad (g,h)\in\Lambda_{\mathcal{C}}$$
defines a fibre bundle in which the fibre over each $h\in\mathcal{C}$ is the maximal torus $G_h$. In light of this, we call $\Lambda_{\mathcal{C}}$ the \textit{universal centralizer} of $\mathcal{C}$. This object exemplifies Definition \ref{Definition: quasi-Hamiltonian Lagrangian}, as our next result demonstrates.

\begin{theorem}\label{Theorem: q-Lagrangian}
If $\mathcal{C}\subseteq G$ is a regular conjugacy class, then the universal centralizer $\Lambda_{\mathcal{C}}$ is a $G$-invariant, quasi-Hamiltonian Lagrangian submanifold of $D(G)$.
\end{theorem}

\begin{proof}
Our $G$-invariance claim follows from a direct calculation. To establish the rest of our proposition, we recall the requirements of Definition \ref{Definition: quasi-Hamiltonian Lagrangian}. This leads us to observe that $\Lambda_{\mathcal{C}}\subseteq\mu_{D(G)}^{-1}(e)\subseteq (D(G))^{\circ}$. We also recall that $\Lambda_{\mathcal{C}}$ is a fibre bundle over $\mathcal{C}$, and that the fibres are maximal tori of $G$.
It follows that $$\dim(\Lambda_{\mathcal{C}})=\dim(\mathcal{C})+r=(n-r)+r=n=\frac{1}{2}\dim(D(G)).$$ These considerations reduce us to proving that $\Lambda_{\mathcal{C}}$ is isotropic in $D(G)$.

Consider the commutative diagram
$$\begin{tikzcd}
\Lambda_{\mathcal{C}}\arrow[d, "r"] \arrow[r, "k"] & \mu_{D(G)}^{-1}(e)_{\text{reg}}\arrow[r, "j"]\arrow[d, "q"] & D(G) \\
\Lambda_{\mathcal{C}}/G\arrow[r, "\ell"]
& \mu_{D(G)}^{-1}(e)_{\text{reg}}/G
\end{tikzcd},$$
where $j$, $k$, and $\ell$ (resp. $q$ and $r$) are the natural inclusions (resp. quotient maps). Observe that $k^*(j^*\omega_{D(G)})$ is the pullback of $\omega_{D(G)}$ to $\Lambda_{\mathcal{C}}$. On the other hand, Proposition \ref{Proposition: Symplectic quotient} and the commutitivity of our diagram imply that 
$$k^*(j^*\omega_{D(G)})=k^*(q^*\overline{\omega_{D(G)}})=r^*(\ell^*\overline{\omega_{D(G)}}).$$ It therefore suffices to prove that $\Lambda_{\mathcal{C}}/G$ is isotropic in the symplectic manifold 
$\mu_{D(G)}^{-1}(e)_{\text{reg}}/G$. 

Recall the symplectomorphism $\varphi:(T\times T_{\text{reg}})/W\longrightarrow\mu_{D(G)}^{-1}(e)_{\text{reg}}/G$ discussed in Proposition \ref{Proposition: Symplectomorphism}. Since we have
$$\varphi\bigg((T\times (T\cap\mathcal{C}))/W\bigg)=\Lambda_{\mathcal{C}}/G,$$
it suffices to prove that $(T\times (T\cap\mathcal{C}))/W$ is isotropic in $(T\times T_{\text{reg}})/W$. This is equivalent to $T\times (T\cap\mathcal{C})$ being isotropic in $T\times T$, which follows from the formula \eqref{Equation: Form on TxT} and the fact that $T\cap\mathcal{C}$ is finite. The proof is complete.  
\end{proof}

\section{BKS pairings on the internally fused double}\label{Section: Main results}
Recall the diffeomorphism $R:D(G)\longrightarrow D(G)$ discussed in Section \ref{Subsection: Gauge theory}. We now examine the geometry of $R(\Lambda_{\mathcal{C}})\cap\Lambda_{\mathcal{C}'}$, where $\mathcal{C},\mathcal{C}'\subseteq G$ are regular conjugacy classes. This will give context for our BKS pairing computation, as described in the introduction of this paper.

\subsection{The geometry of $R(\Lambda_{\mathcal{C}})\cap\Lambda_{\mathcal{C}'}$}\label{Subsection: The geometry} 
Recall the fundamental Weyl alcove $\mathcal{A}\subseteq\mathfrak{t}$ and the associated notation discussed in Section \ref{Subsection: The basics}. Let $\mathcal{C},\mathcal{C}'\subseteq G$ be regular conjugacy classes and set
$$\beta:=\beta(\mathcal{C})\in\mathfrak{A}\quad\text{and}\quad\beta':=\beta(\mathcal{C}')\in\mathfrak{A}.$$ Let us also recall that $\Lambda_{\mathcal{C}}$ and $\Lambda_{\mathcal{C}'}$ are $G$-invariant submanifolds of $D(G)$, and that the diffeomorphism $R:D(G)\longrightarrow D(G)$ is $G$-equivariant. These considerations force the intersection $R(\Lambda_{\mathcal{C}})\cap\Lambda_{\mathcal{C}'}$ to be $G$-invariant, i.e. $R(\Lambda_{\mathcal{C}})\cap\Lambda_{\mathcal{C}'}$ is a union of $G$-orbits. To find these $G$-orbits, we first observe that each $w\in W$ determines a point
$$z_w:=(\exp(w\beta-\beta'),\exp(\beta'))\in R(\Lambda_{\mathcal{C}})\cap\Lambda_{\mathcal{C}'}.$$ It follows that 
$$\Gamma_w:=Gz_w$$ is a $G$-orbit in $R(\Lambda_{\mathcal{C}})\cap\Lambda_{\mathcal{C}'}$ for all $w\in W$.
\begin{proposition}\label{Proposition: Orbit description}
	The association $w\mapsto\Gamma_w$ is a bijection from $W$ to the set of $G$-orbits in $R(\Lambda_{\mathcal{C}})\cap\Lambda_{\mathcal{C}'}$.
\end{proposition}

\begin{proof}
	We begin with a proof of injectivity, letting $u,w\in W$ be such that $\Gamma_u=\Gamma_w$. This amounts to the existence of $g\in G$ for which $g\cdot z_u=z_w$, i.e.
	$$(g\exp(u\beta-\beta')g^{-1},g\exp(\beta')g^{-1})=(\exp(w\beta-\beta'),\exp(\beta')).$$ Since $\beta'$ is a regular element of $\mathfrak{t}$, the condition $g\exp(\beta')g^{-1}=\exp(\beta')$ implies that $g\in T$. It follows that $\exp(u\beta-\beta')=\exp(w\beta-\beta')$, or equivalently that $\exp(u\beta)=\exp(w\beta)$. The regularity of $\beta$ then forces $u=w$ to hold, proving injectivity.
	
	Our claim of surjectivity is equivalent to the following assertion: each point in $R(\Lambda_{\mathcal{C}})\cap\Lambda_{\mathcal{C}'}$ lies in the $G$-orbit of $z_w$ for a suitable $w\in W$. To verify this assertion, we suppose that $(g,h)\in R(\Lambda_{\mathcal{C}})\cap\Lambda_{\mathcal{C}'}$. It follows that $h\in\mathcal{C}'$, i.e. $(g,h)$ lies in the $G$-orbit of a point having $\exp(\beta')$ as its second coordinate. We may therefore assume that $h=\exp(\beta')$. Since $g$ commutes with $h=\exp(\beta')$, we must have $g\in T$. We also have $(g,\exp(\beta'))\in R(\Lambda_{\mathcal{C}})$, forcing $\exp(\beta')g\in\mathcal{C}$ to hold. These last two sentences imply that $\exp(\beta')g\in T\cap\mathcal{C}$, while one knows that $T\cap\mathcal{C}=\{\exp(w\beta):w\in W\}$. Hence $g=\exp(w\beta-\beta')$ for some $w\in W$, and we see that $(g,h)=z_w$. This completes the proof.
\end{proof}

\begin{corollary}\label{Corollary: Easy}
	The intersection $R(\Lambda_{\mathcal{C}})\cap\Lambda_{\mathcal{C}'}$ is an $(n-r)$-dimensional submanifold of $D(G)$. Its connected components coincide with its $G$-orbits.  
\end{corollary}

\begin{proof}
	Note that $T$ is the $G$-stabilizer of $z_w$ for each $w\in W$. It follows that $\Gamma_w$ has dimension $n-r$ for all $w\in W$. This combines with Proposition \ref{Proposition: Orbit description} to imply the desired results.
\end{proof}

Observe that $R(\Lambda_{\mathcal{C}})\subseteq\mu_{D(G)}^{-1}(e)$ and $\Lambda_{\mathcal{C}'}\subseteq\mu_{D(G)}^{-1}(e)_{\text{reg}}$, implying that $R(\Lambda_{\mathcal{C}})\cap\Lambda_{\mathcal{C}'}\subseteq\mu_{D(G)}^{-1}(e)_{\text{reg}}$. In what follows, we fix $x\in R(\Lambda_{\mathcal{C}})\cap\Lambda_{\mathcal{C}'}$ and consider the three subspaces $T_xR(\Lambda_{\mathcal{C}})$, $T_x\Lambda_{\mathcal{C}'}$, and $T_x(\mu_{D(G)}^{-1}(e)_{\text{reg}})$ of $T_x(D(G))$.

\begin{lemma}\label{Lemma: Clean intersection}
	Suppose that $x\in R(\Lambda_{\mathcal{C}})\cap\Lambda_{\mathcal{C}'}$ and let $Gx\subseteq R(\Lambda_{\mathcal{C}})\cap\Lambda_{\mathcal{C}'}$ denote the $G$-orbit through $x$. The following statements hold:
	\begin{itemize}
		\item[(i)]$T_xR(\Lambda_{\mathcal{C}})\cap T_x\Lambda_{\mathcal{C}'}=T_x(R(\Lambda_{\mathcal{C}})\cap\Lambda_{\mathcal{C}'})=T_x(Gx)$;
		\item[(ii)]$T_xR(\Lambda_{\mathcal{C}})+T_x\Lambda_{\mathcal{C}'}=T_x(\mu_{D(G)}^{-1}(e)_{\emph{reg}})$. 
	\end{itemize}
\end{lemma}

\begin{proof}
	Corollary \ref{Corollary: Easy} implies that $$T_x(R(\Lambda_{\mathcal{C}})\cap\Lambda_{\mathcal{C}'})=T_x(Gx),$$ one of the equalities in (i). Using this equality together with the inclusion	$T_x(R(\Lambda_{\mathcal{C}})\cap\Lambda_{\mathcal{C}'})\subseteq T_xR(\Lambda_{\mathcal{C}})\cap T_x\Lambda_{\mathcal{C}'}$ and Corollary \ref{Corollary: Easy}, we obtain
	\begin{align}
	\dim(T_xR(\Lambda_{\mathcal{C}})+T_x\Lambda_{\mathcal{C}'}) & = \dim(T_xR(\Lambda_{\mathcal{C}}))+\dim(T_x\Lambda_{\mathcal{C}'})-\dim(T_xR(\Lambda_{\mathcal{C}})\cap T_x\Lambda_{\mathcal{C}'})\label{Equation: N3}\\
	& \geq \dim(T_xR(\Lambda_{\mathcal{C}}))+\dim(T_x\Lambda_{\mathcal{C}'})-\dim(T_x(R(\Lambda_{\mathcal{C}})\cap\Lambda_{\mathcal{C}'}))\label{Equation: N4}\\
	& = n+n-\dim(Gx)\label{Equation: N1}\\
	& = 2n-(n-r)\\
	& = n+r.
	\end{align}
	We also note that the inclusions $R(\Lambda_{\mathcal{C}})\subseteq \mu_{D(G)}^{-1}(e)$ and $\Lambda_{\mathcal{C}'}\subseteq\mu_{D(G)}^{-1}(e)$ force $T_xR(\Lambda_{\mathcal{C}})+T_x\Lambda_{\mathcal{C}'}\subseteq T_x(\mu_{D(G)}^{-1}(e)_{\text{reg}})$ to hold. Applying Proposition \ref{Proposition: Submanifold}, one obtains 
	\begin{equation}\label{Equation: N2}\dim(T_xR(\Lambda_{\mathcal{C}})+T_x\Lambda_{\mathcal{C}'})\leq n+r.\end{equation}
	It follows that \eqref{Equation: N4} and \eqref{Equation: N2} are equalities. Lines \eqref{Equation: N3}--\eqref{Equation: N2} now imply that $$\dim(T_xR(\Lambda_{\mathcal{C}})\cap T_x\Lambda_{\mathcal{C}'})=\dim(T_x(R(\Lambda_{\mathcal{C}})\cap\Lambda_{\mathcal{C}'})),$$ or equivalently
	$$T_x(R(\Lambda_{\mathcal{C}})\cap\Lambda_{\mathcal{C}'})=T_xR(\Lambda_{\mathcal{C}})\cap T_x\Lambda_{\mathcal{C}'}.$$ This completes our proof of (i).
	On the other hand, the equality \eqref{Equation: N2}, the inclusion $T_xR(\Lambda_{\mathcal{C}})+T_x\Lambda_{\mathcal{C}'}\subseteq T_x(\mu_{D(G)}^{-1}(e)_{\text{reg}})$, Proposition \ref{Proposition: Submanifold}, and a dimension count yield  
	$$\dim(T_xR(\Lambda_{\mathcal{C}})+T_x\Lambda_{\mathcal{C}'})=\dim(T_x(\mu_{D(G)}^{-1}(e)_{\text{reg}})).$$ This proves (ii). 
\end{proof}

\subsection{Half-densities on $R(\Lambda_{\mathcal{C}})$ and $\Lambda_{\mathcal{C}'}$}\label{Subsection: Half-densities}
Recall the brief discussion of conjugacy classes in Section \ref{Subsection: The basics}. Let us now fix a positive integer $k$ and consider a regular, $\frac{1}{k}$-integral conjugacy class $\mathcal{C}\subseteq G$. Let us also set $$\beta:=\beta(\mathcal{C})\in\mathfrak{A}$$ and write $\mathcal{O}\subseteq\mathfrak{g}$ for the adjoint orbit of $k\beta$. One then has a $G$-equivariant diffeomorphism
\begin{equation}\label{Equation: Equivariant diffeomorphism}
\varphi_{\mathcal{O}}^{\mathcal{C}}:\Lambda_{\mathcal{O}}\longrightarrow\Lambda_{\mathcal{C}},\quad (g,\xi)\mapsto(g,\exp\left(\frac{1}{k}\xi\right)),\quad (g,\xi)\in\Lambda_{\mathcal{O}}.
\end{equation}

\begin{lemma}\label{Lemma: Compatibility}
	The following assertions are true.
	\begin{itemize}
		\item[(i)] Suppose that $(g,h)\in\Lambda_{\mathcal{C}}$, and let $\xi\in\mathcal{O}$ be the unique element for which $(g,h)=\varphi_{\mathcal{O}}^{\mathcal{C}}(g,\xi)$. The diffeomorphism $\varphi_{\mathcal{O}}^{\mathcal{C}}$ then satisfies
		$$\varphi_{\mathcal{O}}^{\mathcal{C}}(G(g,\xi))=G(g,h)\quad\text{and}\quad\varphi_{\mathcal{O}}^{\mathcal{C}}(G_{\xi}\times\{\xi\})=G_h\times\{h\}.$$
		\item[(ii)] We have the direct sum decomposition \begin{equation}\label{Equation: New direct}T_{(g,h)}\Lambda_{\mathcal{C}}=T_{(g,h)}(G (g,h))\oplus T_{(g,h)}(G_h\times\{h\})\end{equation} for all $(g,h)\in\Lambda_{\mathcal{C}}$. The $G$-action on $\Lambda_{\mathcal{C}}$ respects these tangent space decompositions.
	\end{itemize} 
\end{lemma} 

\begin{proof}
	Since $\varphi_{\mathcal{O}}^{\mathcal{C}}$ is $G$-equivariant and satisfies $(g,h)=\varphi_{\mathcal{O}}^{\mathcal{C}}(g,\xi)$, we must have $\varphi_{\mathcal{O}}^{\mathcal{C}}(G(g,\xi))=G(g,h)$. The second identity $\varphi_{\mathcal{O}}^{\mathcal{C}}(G_{\xi}\times\{\xi\})=G_h\times\{h\}$ follows from the definition of $\varphi_{\mathcal{O}}^{\mathcal{C}}$ and the observation that $$G_{\xi}=G_{\frac{1}{k}\xi}=G_{\exp\left(\frac{1}{k}\xi\right)}=G_h.$$ This proves (i). The assertion (ii) is a consequence of (i) and Lemma \ref{Lemma: Regular orbit}.
\end{proof}

Now recall the half-density $\rho_{\mathcal{O}}$ on $\Lambda_{\mathcal{O}}$ constructed in Section \ref{Subsection: Half-density}. Let $\rho_{\mathcal{C}}$ denote the half-density on $\Lambda_{\mathcal{C}}$ corresponding to $\rho_{\mathcal{O}}$ under the diffeomorphism \eqref{Equation: Equivariant diffeomorphism}. This new half-density admits the following description. 
To prepare for it, we invite the reader to recall the notation and discussion in Sections \ref{Subsection: Adjoint} and \ref{Subsection: Haar}.

\begin{proposition}\label{Proposition: Half-density on conjugacy class}
	Fix $(g,h)\in\Lambda_{\mathcal{C}}$ and let $\xi\in\mathcal{O}$ be the unique element satisfying $\exp(\frac{1}{k}\xi)=h$. Assume that $\{\eta_1,\ldots,\eta_{n-r}\}$ and $\{\zeta_1,\ldots,\zeta_r\}$ are bases of $\mathfrak{g}_{\xi}^{\perp}$ and $\mathfrak{g}_{\xi}$, respectively. Write $\widetilde{\eta_j}$ for the fundamental vector field on $D(G)$ associated to $\eta_j$, i.e. $$\widetilde{\eta_j}:=(\eta_j)_{D(G)}$$ for all $j\in\{1,\ldots,n-r\}$. Let us also define the vector $$(\dot{\zeta_k})_{(g,h)}:=((dL_g)_e(\zeta_k),0)\in T_gG_{\xi}\oplus\{0\}=T_{(g,\xi)}(G_{h}\times\{h\})\subseteq T_{(g,h)}\Lambda_{\mathcal{C}}$$ for each $k\in\{1,\ldots,r\}$. The following statements then hold.
	\begin{itemize}
		\item[(i)] The sets 
		$$\{(\widetilde{\eta_1})_{(g,h)},\ldots,(\widetilde{\eta_{n-r}})_{(g,h)}\}\quad\text{and}\quad\{(\dot{\zeta_1})_{(g,h)},\ldots,(\dot{\zeta_r})_{(g,h)}\}$$ are bases of $T_{(g,h)}(G(g,h))$ and $T_{(g,h)}(G_{h}\times\{h\})$, respectively.
		\item[(ii)] The set
		$$\{(\widetilde{\eta_1})_{(g,h)},\ldots,(\widetilde{\eta_{n-r}})_{(g,h)},(\dot{\zeta_1})_{(g,h)},\ldots,(\dot{\zeta_r})_{(g,h)}\}$$ is a basis of $T_{(g,h)}\Lambda_{\mathcal{C}}$.
		\item[(iii)] The value of $(\rho_{\mathcal{C}})_{(g,h)}$ on this basis is
		$$\kappa(G)^{\frac{1}{2}}\cdot\big\vert\Omega_{\xi}\big(\eta_1,\ldots,\eta_{n-r}\big)\big\vert^{\frac{1}{2}}.$$
	\end{itemize} 
\end{proposition}

\begin{proof}
	Recall the notation in Proposition \ref{Proposition: Half-density description} and consider the differential of $\varphi_{\mathcal{O}}^{\mathcal{C}}$ at $(g,\xi)\in\Lambda_{\mathcal{O}}$, i.e.
	$$(d\varphi_{\mathcal{O}}^{\mathcal{C}})_{(g,\xi)}:T_{(g,\xi)}\Lambda_{\mathcal{O}}\overset{\cong}\longrightarrow T_{(g,h)}\Lambda_{\mathcal{C}}.$$ Since $\varphi_{\mathcal{O}}^{\mathcal{C}}$ is $G$-equivariant, we must have
	$$(d\varphi_{\mathcal{O}}^{\mathcal{C}})_{(g,\xi)}((\overline{\eta_j})_{(g,\xi)})=(\widetilde{\eta_j})_{(g,h)}$$ for all $j\in\{1,\ldots,n-r\}$. It is also straightforward to verify that
	$$(d\varphi_{\mathcal{O}}^{\mathcal{C}})_{(g,\xi)}((\widehat{\zeta_k})_{(g,\xi)})=(\dot{\zeta_k})_{(g,h)}$$ for all $k\in\{1,\ldots,r\}$. The desired results now follow from Proposition \ref{Proposition: Half-density description} and Lemma \ref{Lemma: Compatibility}. 
\end{proof}

\begin{corollary}\label{Corollary: Compatibility}
	We have
	\begin{equation}\label{Equation: New new direct}T_{(g,h)}R(\Lambda_{\mathcal{C}})=T_{(g,h)}(G (g,h))\oplus T_{(g,h)}(R(G_{hg}\times\{hg\}))\end{equation} for all $(g,h)\in R(\Lambda_{\mathcal{C}})$. The $G$-action on $R(\Lambda_{\mathcal{C}})$ respects these tangent space decompositions.
\end{corollary}

\begin{proof}
	Suppose that $(g,h)\in R(\Lambda_{\mathcal{C}})$. We then have $R^{-1}(g,h)=(g,hg)\in\Lambda_{\mathcal{C}}$, in which context Lemma \ref{Lemma: Compatibility}(ii) yields
	$$T_{(g,hg)}\Lambda_{\mathcal{C}}=T_{(g,hg)}(G (g,hg))\oplus T_{(g,hg)}(G_{hg}\times\{hg\}).$$ It follows that
	$$T_{(g,h)}R(\Lambda_{\mathcal{C}})=T_{(g,h)}(R(G(g,hg)))\oplus T_{(g,h)}(R(G_{hg}\times\{hg\})).$$ On the other hand, the $G$-equivariance of $R$ implies that $R(G(g,hg))=G (g,h)$. Hence $$T_{(g,h)}R(\Lambda_{\mathcal{C}})=T_{(g,h)}(G (g,h))\oplus T_{(g,h)}(R(G_{hg}\times\{hg\})),$$ as desired.
	
	The arguments in the previous paragraph lead to the following observation: the tangent space decompositions \eqref{Equation: New new direct} are the result of applying $R$ to the decompositions \eqref{Equation: New direct}. This combines with Lemma \ref{Lemma: Compatibility}(ii) and the fact that $R$ is a $G$-equivariant diffeomorphism to imply that $G$ respects the decompositions \eqref{Equation: New new direct}. The proof is therefore complete. 
\end{proof}

\begin{lemma}\label{Lemma: Computation}
	Fix $(g,h)\in D(G)$ and make the identifications
	$$T_{(g,h)}(D(G))=T_gG\oplus T_hG\quad\text{and}\quad T_{(g,hg)}(D(G))=T_gG\oplus T_{hg}G.$$ We then have $$(dR)_{(g,hg)}((dL_g)_e(\zeta),0)=((dL_g)_e(\zeta),-(dL_h)_e(\mathrm{Ad}_g(\zeta)))$$ for all $\zeta\in\mathfrak{g}$, where
	$$(dR)_{(g,hg)}:T_gG\oplus T_{hg}G\longrightarrow T_gG\oplus T_hG$$ is the differential of $R$ at $(g,hg)$.
\end{lemma}

\begin{proof}
	We have
	\begin{align*}
	(dR)_{(g,hg)}((dL_g)_e(\zeta),0) & = \frac{d}{dt}\bigg\vert_{t=0}R(g\exp(t\zeta),hg)\\
	& = \frac{d}{dt}\bigg\vert_{t=0}(g\exp(t\zeta),hg\exp(t\zeta)^{-1}g^{-1})\\
	& = \frac{d}{dt}\bigg\vert_{t=0}(g\exp(t\zeta),h\exp(-t\mathrm{Ad}_g(\zeta)))\\
	& = ((dL_g)_e(\zeta),-(dL_h)_e(\mathrm{Ad}_g(\zeta))).
	\end{align*}
\end{proof}

Now note that $R$ identifies $\rho_{\mathcal{C}}$ with a half-density $\nu_{\mathcal{C}}$ on $R(\Lambda_{\mathcal{C}})$. This leads to the following result. 

\begin{proposition}\label{Proposition: Half-density on shifted conjugacy class}
	Fix $(g,h)\in R(\Lambda_{\mathcal{C}})$ and let $\xi\in\mathcal{O}$ be the unique element satisfying $\exp(\frac{1}{k}\xi)=hg$. Assume that $\{\eta_1,\ldots,\eta_{n-r}\}$ and $\{\zeta_1,\ldots,\zeta_r\}$ are bases of $\mathfrak{g}_{\xi}^{\perp}$ and $\mathfrak{g}_{\xi}$, respectively. Write $\widetilde{\eta_j}$ for the fundamental vector field on $D(G)$ associated to $\eta_j$, i.e. $$\widetilde{\eta_j}:=(\eta_j)_{D(G)}$$ for all $j\in\{1,\ldots,n-r\}$. Let us also define the vector $$(\zeta_k')_{(g,h)}:=((dL_g)_e(\zeta_k),-(dL_h)_e(\zeta_k))\in T_{(g,h)}R(G_{hg}\times\{hg\})\subseteq T_{(g,h)}R(\Lambda_{\mathcal{C}})$$ for each $k\in\{1,\ldots,r\}$. The following statements then hold.
	\begin{itemize}
		\item[(i)] The sets 
		$$\{(\widetilde{\eta_1})_{(g,h)},\ldots,(\widetilde{\eta_{n-r}})_{(g,h)}\}\quad\text{and}\quad\{(\zeta_1')_{(g,h)},\ldots,(\zeta_r')_{(g,h)}\}$$ are bases of $T_{(g,h)}(G(g,h))$ and $T_{(g,h)}R(G_{hg}\times\{hg\})$, respectively.
		\item[(ii)] The set
		$$\{(\widetilde{\eta_1})_{(g,h)},\ldots,(\widetilde{\eta_{n-r}})_{(g,h)},(\zeta_1')_{(g,h)},\ldots,(\zeta_r')_{(g,h)}\}$$ is a basis of $T_{(g,h)}R(\Lambda_{\mathcal{C}})$.
		\item[(iii)] The value of $(\nu_{\mathcal{C}})_{(g,h)}$ on this basis is
		$$\kappa(G)^{\frac{1}{2}}\cdot\big\vert\Omega_{\xi}\big(\eta_1,\ldots,\eta_{n-r}\big)\big\vert^{\frac{1}{2}}.$$
	\end{itemize} 
\end{proposition}

\begin{proof}
	Recall the notation used in Proposition \ref{Proposition: Half-density on conjugacy class}. Lemma \ref{Lemma: Computation} then yields
	$$(dR)_{(g,hg)}((\dot{\zeta_k})_{(g,hg)})=(\zeta_k')_{(g,h)}$$ for all $k\in\{1,\ldots,r\}$. At the same time, the $G$-equivariance of $R$ implies that
	$$(dR)_{(g,hg)}((\widetilde{\eta}_j)_{(g,hg)})=(\widetilde{\eta}_j)_{(g,h)}$$ for all $j\in\{1,\ldots,n-r\}$. The desired results now follow from Proposition \ref{Proposition: Half-density on conjugacy class} and Corollary \ref{Corollary: Compatibility}.
\end{proof}

\subsection{Phase functions on $R(\Lambda_{\mathcal{C}})$ and $\Lambda_{\mathcal{C}'}$}\label{Subsection: Phase functions}
Fix a positive integer $k$ and let $\mathcal{C}\subseteq G$ be a regular, $\frac{1}{k}$-integral conjugacy class. Let $\mathcal{O}\subseteq\mathfrak{g}$ be the adjoint orbit of $k\beta(\mathcal{C})$, and recall the $G$-equivariant diffeomorphism $\varphi_{\mathcal{O}}^{\mathcal{C}}:\Lambda_{\mathcal{O}}\overset{\cong}\longrightarrow\Lambda_{\mathcal{C}}$ defined in \eqref{Equation: Equivariant diffeomorphism}. Let us also recall the $G$-invariant phase function $\psi_{\mathcal{O}}:\Lambda_{\mathcal{O}}\longrightarrow S^1$ considered in Section \ref{Subsection: Phase}. We then have $\psi_{\mathcal{O}}=(\varphi_{\mathcal{O}}^{\mathcal{C}})^*\psi_{\mathcal{C}}$, where $\psi_{\mathcal{C}}:\Lambda_{\mathcal{C}}\longrightarrow S^1$ is the $G$-invariant map defined as follows:
$$\psi_{\mathcal{C}}(g,h)=e^{2\pi i\langle\xi,\eta\rangle},$$
where $\xi\in\mathcal{O}$ is the unique element for which $\exp(\frac{1}{k}\xi)=h$ and $\eta\in\mathfrak{g}_{\xi}$ satisfies $g=\exp(\eta)$.
At the same time, one can pull $\psi_{\mathcal{C}}$ back along the $G$-equivariant diffeomorphism
$$R^{-1}\bigg\vert_{R(\Lambda_{\mathcal{C}})}:R(\Lambda_{\mathcal{C}})\overset{\cong}\longrightarrow\Lambda_{\mathcal{C}}$$ to obtain a $G$-invariant map $\vartheta_{\mathcal{C}}:R(\Lambda_{\mathcal{C}})\longrightarrow S^1$. It follows that
$$\vartheta_{\mathcal{C}}(g,h)=\psi_{\mathcal{C}}(g,hg)$$ for all $(g,h)\in R(\Lambda_{\mathcal{C}})$.

\subsection{Some technical lemmas}
Fix a positive integer $k$ and let $\mathcal{C},\mathcal{C}'\subseteq G$ be regular, $\frac{1}{k}$-integral conjugacy classes. Write $\mathcal{O},\mathcal{O}'\subseteq\mathfrak{g}$ for the adjoint orbits of $k\beta(\mathcal{C}),k\beta(\mathcal{C}')$, respectively. 
Section \ref{Subsection: Half-densities} uses this information to construct half-densities $\nu_{\mathcal{C}}$ and $\rho_{\mathcal{C}'}$ on $R(\Lambda_{\mathcal{C}})$ and $\Lambda_{\mathcal{C}'}$, respectively. Let us also recall that $R(\Lambda_{\mathcal{C}})$ and $\Lambda_{\mathcal{C}'}$ are cleanly intersecting Lagrangian submanifolds of $D(G)$ (see Lemma \ref{Lemma: Clean intersection}). It now follows from Section \ref{Subsection: The BKS density} that $\nu_{\mathcal{C}}$ and $\rho_{\mathcal{C}'}$ determine a density $D(\nu_{\mathcal{C}},\rho_{\mathcal{C}'})$ on $R(\Lambda_{\mathcal{C}})\cap\Lambda_{\mathcal{C}'}$. Note that    
\begin{equation}\label{Equation: Fact}D(\nu_{\mathcal{C}},\rho_{\mathcal{C}'})_x=\Phi_x((\nu_{\mathcal{C}})_x\otimes (\rho_{\mathcal{C}'})_x)\end{equation}
for all $x\in R(\Lambda_{\mathcal{C}})\cap\Lambda_{\mathcal{C}'}$, where
\begin{equation}\label{Equation: The map}\Phi_x:\vert T_xR(\Lambda_{\mathcal{C}})\vert^{\frac{1}{2}}\otimes\vert T_x\Lambda_{\mathcal{C}'}\vert^{\frac{1}{2}}\longrightarrow\vert T_x(R(\Lambda_{\mathcal{C}})\cap\Lambda_{\mathcal{C}'})\vert\end{equation}
is the map \eqref{Equation: Canonical isomorphism}.

In what follows, we describe the density $D(\nu_{\mathcal{C}},\rho_{\mathcal{C}'})$ on $R(\Lambda_{\mathcal{C}})\cap\Lambda_{\mathcal{C}'}$. 

\begin{lemma}\label{Lemma: G-invariance of BKS density}
The density $D(\nu_{\mathcal{C}},\rho_{\mathcal{C}'})$ is invariant under the action of $G$ on $R(\Lambda_{\mathcal{C}})\cap\Lambda_{\mathcal{C}'}$.
\end{lemma}

\begin{proof}
Recall that $\rho_{\mathcal{O}'}$ is $G$-invariant (see Proposition \ref{Proposition: G-invariance of half-density}), and that $\rho_{\mathcal{C}'}$ corresponds to $\rho_{\mathcal{O}'}$ under the $G$-equivariant diffeomorphism $\varphi_{\mathcal{O}'}^{\mathcal{C}'}:\Lambda_{\mathcal{O}'}\overset{\cong}\longrightarrow\Lambda_{\mathcal{C}'}$ from \eqref{Equation: Equivariant diffeomorphism}. These considerations imply that $\rho_{\mathcal{C}'}$ is $G$-invariant, and an analogous argument establishes the invariance of $\rho_{\mathcal{C}}$. At the same time, we recall that $\nu_{\mathcal{C}}$ corresponds to $\rho_{\mathcal{C}}$ under the $G$-equivariant diffeomorphism 
$$R\vert_{\Lambda_{\mathcal{C}}}:\Lambda_{\mathcal{C}}\overset{\cong}\longrightarrow R(\Lambda_{\mathcal{C}}).$$ We conclude that $\nu_{\mathcal{C}}$ is also $G$-invariant, and the desired result now follows from Lemma \ref{Lemma: G-invariance of densities}. 
\end{proof}

We also need the following elementary lemma.

\begin{lemma}\label{Lemma: Stabilizers}
	If $g,h\in G$ are commuting elements for which $h\in G_{\emph{reg}}$ and $hg\in G_{\emph{reg}}$, then $G_{hg}=G_h$.
\end{lemma}

\begin{proof}
	Note that $G_{hg}$ is a maximal torus in $G$. We also have $h\in G_{hg}$, owing to the fact that $h$ and $g$ commute. These last two sentences force $G_{hg}\subseteq G_h$ to hold. Since $G_h$ is also a maximal torus, we obtain $G_{hg}=G_h$.
\end{proof}

Now fix $(g,h)\in R(\Lambda_{\mathcal{C}})\cap\Lambda_{\mathcal{C}'}$ and set $x:=(g,h)$. Lemma \ref{Lemma: Stabilizers} then gives the equality $G_h=G_{hg}$ of maximal tori. No generality will be lost in taking this maximal torus to be our fixed maximal torus $T$. Applications of Lemma \ref{Lemma: Compatibility} Corollary \ref{Corollary: Compatibility} then yield
$$T_xR(\Lambda_{\mathcal{C}})=T_x(Gx)\oplus T_xR(T\times\{hg\})\quad\text{and}\quad T_x\Lambda_{\mathcal{C}}=T_x(Gx)\oplus T_x(T\times\{h\}).$$ By Lemma \ref{Lemma: Clean intersection}(i), these statements may be written as \begin{equation}\label{Equation: Tangent space decompositions}T_xR(\Lambda_{\mathcal{C}})=T_x(R(\Lambda_{\mathcal{C}})\cap\Lambda_{\mathcal{C'}})\oplus T_xR(T\times\{hg\})\quad\text{and}\quad T_x\Lambda_{\mathcal{C}}=T_x(R(\Lambda_{\mathcal{C}})\cap\Lambda_{\mathcal{C'}})\oplus T_x(T\times\{h\}).\end{equation} These combine with Corollary \ref{Corollary: Computational} to imply that 
\begin{equation}\label{Equation: Observation}T_xR(T\times\{hg\})\oplus T_x(T\times\{h\})\longrightarrow \frac{T_xR(\Lambda_{\mathcal{C}})+T_x\Lambda_{\mathcal{C}'}}{T_x(R(\Lambda_{\mathcal{C}})\cap\Lambda_{\mathcal{C}'})},\quad (v_1,v_2)\mapsto[v_1-v_2]\end{equation}
is a vector space isomorphism. Let us consider the induced isomorphism
\begin{equation}\label{Equation: Image isomorphism}\vert T_xR(T\times\{hg\})\oplus T_x(T\times\{h\})\vert^{\frac{1}{2}}\overset{\cong}\longrightarrow \bigg\vert\frac{T_xR(\Lambda_{\mathcal{C}})+T_x\Lambda_{\mathcal{C}'}}{T_x(R(\Lambda_{\mathcal{C}})\cap\Lambda_{\mathcal{C}'})}\bigg\vert^{\frac{1}{2}}.\end{equation}
 
\begin{lemma}\label{Lemma: Complicated half-density}
Use the notation of Propositions \ref{Proposition: Half-density on conjugacy class} and \ref{Proposition: Half-density on shifted conjugacy class}, and let $$\rho\in \vert T_xR(T\times\{hg\})\oplus T_x(T\times\{h\})\vert^{\frac{1}{2}}$$ be the unique half-density satisfying
$$\rho\bigg(\big((\zeta_1')_x,0\big),\ldots,\big((\zeta_r')_x,0\big),\big(0,(\dot{\zeta_1})_x\big),\ldots,\big(0,(\dot{\zeta_r})_x\big)\bigg)=\kappa(G).$$ If $\rho'$ is the image of $\rho$ under \eqref{Equation: Image isomorphism}, then $$\rho'\bigg(\big[\big((dL_g)_e(\zeta_1),0\big)\big],\ldots,\big[\big((dL_g)_e(\zeta_r),0\big)\big],\big[\big(0,(dL_h)_e(\zeta_1)\big)\big],\ldots,\big[\big(0,(dL_h)_e(\zeta_r)\big)\big]\bigg)=\kappa(G).$$ 
\end{lemma}

\begin{proof}
The images of $$\big((\zeta_1')_x,0\big),\ldots,\big((\zeta_r')_x,0\big),\big(0,(\dot{\zeta_1})_x\big),\ldots,\big(0,(\dot{\zeta_r})_x\big)$$ under \eqref{Equation: Observation} are $$\big[(\zeta_1')_x\big],\ldots,\big[(\zeta_r')_x\big],-\big[(\dot{\zeta_1})_x\big],\ldots,-\big[(\dot{\zeta_r})_x\big],$$ respectively. It now follows from the definition of \eqref{Equation: Image isomorphism} that 
\begin{align*}\rho'\bigg(\big[(\zeta_1')_x\big],\ldots,\big[(\zeta_r')_x\big],-\big[(\dot{\zeta_1})_x\big],\ldots,-\big[(\dot{\zeta_r})_x\big]\bigg) & =\rho\bigg(\big((\zeta_1')_x,0\big),\ldots,\big((\zeta_r')_x,0\big),\big(0,(\dot{\zeta_1})_x\big),\ldots,\big(0,(\dot{\zeta_r})_x\big)\bigg)\\ &=\kappa(G).\end{align*}
This is precisely the statement that
$$\rho'\bigg(\big[(\zeta_1')_x\big],\ldots,\big[(\zeta_r')_x\big],\big[(\dot{\zeta_1})_x\big],\ldots,\big[(\dot{\zeta_r})_x\big]\bigg)=\kappa(G).$$

Now consider the linear automorphism $$\frac{T_xR(\Lambda_{\mathcal{C}})+T_x\Lambda_{\mathcal{C}'}}{T_x(R(\Lambda_{\mathcal{C}})\cap\Lambda_{\mathcal{C}'})}\longrightarrow\frac{T_xR(\Lambda_{\mathcal{C}})+T_x\Lambda_{\mathcal{C}'}}{T_x(R(\Lambda_{\mathcal{C}})\cap\Lambda_{\mathcal{C}'})}$$ that sends the basis $$\{\big[(\zeta_1')_x\big],\ldots,\big[(\zeta_r')_x\big],\big[(\dot{\zeta_1})_x\big],\ldots,\big[(\dot{\zeta_r})_x\big]\}$$ to the basis
$$\{\big[(\dot{\zeta_1})_x\big],\ldots,\big[(\dot{\zeta_r})_x\big],-\big[(\zeta_1')_x\big]+\big[(\dot{\zeta_1})_x\big],\ldots,-\big[(\zeta_r')_x\big]+\big[(\dot{\zeta_r})_x\big]\}.$$ This automorphism has a determinant of absolute value one, so that 
\begin{align*}
& \rho'\bigg(\big[(\dot{\zeta_1})_x\big],\ldots,\big[(\dot{\zeta_r})_x\big],-\big[(\zeta_1')_x\big]+\big[(\dot{\zeta_1})_x\big],\ldots,-\big[(\zeta_r')_x\big]+\big[(\dot{\zeta_r})_x\big]\bigg)\\
& =  \rho'\bigg(\big[(\zeta_1')_x\big],\ldots,\big[(\zeta_r')_x\big],\big[(\dot{\zeta_1})_x\big],\ldots,\big[(\dot{\zeta_r})_x\big]\bigg)\\
& = \kappa(G).
\end{align*}
It remains only to observe that 
$$\big[(\dot{\zeta_k})_x\big]=\big[\big((dL_g)_e(\zeta_k),0)\big]\quad\text{and}\quad -\big[(\zeta_k')_x\big]+\big[(\dot{\zeta_k})_x\big]=\big[\big(0,(dL_h)_e(\zeta_k)\big)\big]$$ for all $k\in\{1,\ldots,r\}$.
\end{proof}

We now scrutinize the codomain of \eqref{Equation: Observation} in greater detail. Section \ref{Subsection: The BKS density} explains that this codomain is a symplectic vector space, and that the symplectic form induces an isomorphism 
\begin{equation}\label{Equation: Push}\bigg\vert\frac{T_xR(\Lambda_{\mathcal{C}})+T_x\Lambda_{\mathcal{C}'}}{T_x(R(\Lambda_{\mathcal{C}})\cap\Lambda_{\mathcal{C}'})}\bigg\vert^{\frac{1}{2}}\overset{\cong}\longrightarrow\mathbb{C}.\end{equation}
At the same time, Lemma \ref{Lemma: Clean intersection} implies that 
\begin{equation}\label{Equation: Value}\frac{T_xR(\Lambda_{\mathcal{C}})+T_x\Lambda_{\mathcal{C}'}}{T_x(R(\Lambda_{\mathcal{C}})\cap\Lambda_{\mathcal{C}'})}=\frac{T_x(\mu_{D(G)}^{-1}(e)_{\text{reg}})}{T_x(Gx)}=T_{[x]}(\mu_{D(G)}^{-1}(e)_{\text{reg}}/G).\end{equation} The symplectic form $(\overline{\omega_{D(G)}})_{[x]}$ on $T_{[x]}(\mu_{D(G)}^{-1}(e)_{\text{reg}}/G)$ (see Section \ref{Subsection: A special}) then agrees with the above-discussed symplectic form on $\frac{T_xR(\Lambda_{\mathcal{C}})+T_x\Lambda_{\mathcal{C}'}}{T_x(R(\Lambda_{\mathcal{C}})\cap\Lambda_{\mathcal{C}'})}$, as each is obtained from $(\omega_{D(G)})_x$ in the same way.

\begin{lemma}\label{Lemma: Complicated image}
Suppose that $\{\zeta_1,\ldots,\zeta_r\}$ is an orthonormal basis of $\mathfrak{t}$. If we have
$$\rho'\in \bigg\vert\frac{T_xR(\Lambda_{\mathcal{C}})+T_x\Lambda_{\mathcal{C}'}}{T_x(R(\Lambda_{\mathcal{C}})\cap\Lambda_{\mathcal{C}'})}\bigg\vert^{\frac{1}{2}},$$ then the image of $\rho'$ under \eqref{Equation: Push} is equal to
$$\rho'\bigg(\big[\big((dL_g)_e(\zeta_1),0\big)\big],\ldots,\big[\big((dL_g)_e(\zeta_r),0\big)\big],\big[\big(0,(dL_h)_e(\zeta_1)\big)\big],\ldots,\big[\big(0,(dL_h)_e(\zeta_r)\big)\big]\bigg).$$
\end{lemma}

\begin{proof}
Let $a\in\mathbb{C}$ denote the image of $\rho'$ under \eqref{Equation: Push}, and recall the paragraph preceding the statement of this lemma. This paragraph implies that $$\rho'=a\vert\Omega_{[x]}\vert^{\frac{1}{2}},$$ where $\Omega_{[x]}:=((\overline{\omega_{D(G)}})_{[x]})^r$ is the highest non-zero wedge power of $(\overline{\omega_{D(G)}})_{[x]}$. It therefore suffices to prove that 
$$\vert\Omega_{[x]}\vert\bigg(\big[\big((dL_g)_e(\zeta_1),0\big)\big],\ldots,\big[\big((dL_g)_e(\zeta_r),0\big)\big],\big[\big(0,(dL_h)_e(\zeta_1)\big)\big],\ldots,\big[\big(0,(dL_h)_e(\zeta_r)\big)\big]\bigg)=1.$$

Consider the symplectic form $(\omega_{T\times T})_x$ on $T_x(T\times T)$ from Section \ref{Subsection: A special}, and its highest non-zero wedge power $\Xi_x:=((\omega_{T\times T})_{x})^r$. The relation \eqref{Equation: Relation} implies that
\begin{align*}
& \Omega_{[x]}\bigg(\big[\big((dL_g)_e(\zeta_1),0\big)\big],\ldots,\big[\big((dL_g)_e(\zeta_r),0\big)\big],\big[\big(0,(dL_h)_e(\zeta_1)\big)\big],\ldots,\big[\big(0,(dL_h)_e(\zeta_r)\big)\big]\bigg)\\
& = \Xi_{x}\bigg(\big((dL_g)_e(\zeta_1),0\big),\ldots,\big((dL_g)_e(\zeta_r),0\big),\big(0,(dL_h)_e(\zeta_1)\big),\ldots,\big(0,(dL_h)_e(\zeta_r)\big)\bigg).
\end{align*}
By virtue of the definition \eqref{Equation: Form on TxT} and the fact that $\{\zeta_1,\ldots,\zeta_r\}$ is orthonormal, the right-hand side has an absolute value of $1$. This completes the proof.
\end{proof}

\subsection{The main results}
Fix a positive integer $k$ and recall the notation in Section \ref{Subsection: The basics} associated with conjugacy classes. Let $\mathcal{C},\mathcal{C}'\subseteq G$ be regular, $\frac{1}{k}$-integral conjugacy classes and set
$$\beta:=\beta(\mathcal{C})\in\mathfrak{A}\quad\text{and}\quad\beta':=\beta(\mathcal{C}')\in\mathfrak{A}.$$ One then recalls the definition and properties of $z_w\in R(\Lambda_{\mathcal{C}})\cap\Lambda_{\mathcal{C}'}$ from Section \ref{Subsection: The geometry}, where $w\in W$. Let us also recall that $n$ and $r$ denote the dimension and rank of $G$, respectively.

\begin{proposition}\label{Proposition: Main computation}
Fix $w\in W$, set $x:=z_w\in R(\Lambda_{\mathcal{C}})\cap\Lambda_{\mathcal{C}'}$, and let $\{\eta_1,\ldots,\eta_{n-r}\}$ and be a basis of $\mathfrak{t}^{\perp}$. Let $$\widetilde{\eta_j}:=(\eta_j)_{D(G)}$$ be the fundamental vector field on $D(G)$ associated to $\eta_j$ for each $j\in\{1,\ldots,n-r\}$. The following statements then hold.
\begin{itemize}
\item[(i)] The set 
$\{(\widetilde{\eta_1})_{x},\ldots,(\widetilde{\eta_{n-r}})_{x}\}$ is a basis of $T_x(R(\Lambda_{\mathcal{C}})\cap\Lambda_{\mathcal{C}'})$.
\item[(ii)] The value of $D(\nu_{\mathcal{C}},\rho_{\mathcal{C}'})_x$ on this basis is
$$k^{n-r}\kappa(G)\cdot\vert\Omega_{\beta}(\eta_1,\ldots,\eta_{n-r})\vert^{\frac{1}{2}}\vert\Omega_{\beta'}(\eta_1,\ldots,\eta_{n-r})\vert^{\frac{1}{2}}.$$
\end{itemize}
\end{proposition}

\begin{proof}
Part (i) follows from the identity $T_x(R(\Lambda_{\mathcal{C}})\cap\Lambda_{\mathcal{C}'})=T_x(Gx)$ (see Lemma \ref{Lemma: Clean intersection}(i)) and Proposition \ref{Proposition: Half-density on conjugacy class}(i).

We now verify (ii). Note that if we write $z_w=(g,h)$, then $h=\exp(\beta')$ and $hg=\exp(w\beta)$. If $\{\zeta_1,\ldots,\zeta_r\}$ is an orthonormal basis of $\mathfrak{t}$, then the previous sentence combines with Propositions \ref{Proposition: Half-density on conjugacy class} and \ref{Proposition: Half-density on shifted conjugacy class} to imply 
$$(\rho_{\mathcal{C}'})_x((\widetilde{\eta_1})_{x},\ldots,(\widetilde{\eta_{n-r}})_{x},(\dot{\zeta_1})_{x},\ldots,(\dot{\zeta_r})_{x})=\kappa(G)^{\frac{1}{2}}\cdot\big\vert\Omega_{k\beta'}\big(\eta_1,\ldots,\eta_{n-r}\big)\big\vert^{\frac{1}{2}}$$
and 
$$(\nu_{\mathcal{C}})_x((\widetilde{\eta_1})_{x},\ldots,(\widetilde{\eta_{n-r}})_{x},(\zeta_1')_{x},\ldots,(\zeta_r')_{x})=\kappa(G)^{\frac{1}{2}}\cdot\big\vert\Omega_{w(k\beta)}\big(\eta_1,\ldots,\eta_{n-r}\big)\big\vert^{\frac{1}{2}}.$$ Lemma \ref{Lemma: Symplectic form identity} tells us that $\vert\Omega_{w(k\beta)}\vert^{\frac{1}{2}}=\vert\Omega_{k\beta}\vert^{\frac{1}{2}}$ as half-densities on $\mathfrak{t}^{\perp}$, allowing us to write
\begin{equation}\label{Equation: Form of nu}(\rho_{\mathcal{C}'})_x((\widetilde{\eta_1})_{x},\ldots,(\widetilde{\eta_{n-r}})_{x},(\dot{\zeta_1})_{x},\ldots,(\dot{\zeta_r})_{x})=\kappa(G)^{\frac{1}{2}}\cdot\big\vert\Omega_{k\beta'}\big(\eta_1,\ldots,\eta_{n-r}\big)\big\vert^{\frac{1}{2}}\end{equation}
and 
\begin{equation}\label{Equation: Form of lambda}(\nu_{\mathcal{C}})_x((\widetilde{\eta_1})_{x},\ldots,(\widetilde{\eta_{n-r}})_{x},(\zeta_1')_{x},\ldots,(\zeta_r')_{x})=\kappa(G)^{\frac{1}{2}}\cdot\big\vert\Omega_{k\beta}\big(\eta_1,\ldots,\eta_{n-r}\big)\big\vert^{\frac{1}{2}}.\end{equation}

Now recall the tangent space decompositions \eqref{Equation: Tangent space decompositions} and set
$$V_1:=T_xR(T\times\{hg\})\quad\text{and}\quad V_2:=T_x(T\times\{h\}).$$ Let us also recall the map \eqref{Equation: The map} and the fact \eqref{Equation: Fact}. These last two sentences combine with Proposition \ref{Proposition: Two BKS} to yield
$$D(\nu_{\mathcal{C}},\rho_{\mathcal{C}'})_x=\Phi_x((\nu_{\mathcal{C}})_x\otimes (\rho_{\mathcal{C}'})_x)=\Phi_x^{V_1,V_2}((\nu_{\mathcal{C}})_x\otimes (\rho_{\mathcal{C}'})_x),$$
where $$\Phi_x^{V_1,V_2}:\vert T_xR(\Lambda_{\mathcal{C}})\vert^{\frac{1}{2}}\otimes\vert T_x\Lambda_{\mathcal{C}'}\vert^{\frac{1}{2}}\longrightarrow\vert T_x(R(\Lambda_{\mathcal{C}})\cap\Lambda_{\mathcal{C}'})\vert$$ is defined in \eqref{Equation: List2} and \eqref{Equation: Nicest}. One uses this definition along with \eqref{Equation: Form of nu}, \eqref{Equation: Form of lambda}, Proposition \ref{Proposition: Half-density on conjugacy class}, and Proposition \ref{Proposition: Half-density on shifted conjugacy class} to obtain
$$D(\nu_{\mathcal{C}},\rho_{\mathcal{C}'})_x((\widetilde{\eta_1})_{x},\ldots,(\widetilde{\eta_{n-r}})_{x})=a\cdot\big\vert\Omega_{k\beta}\big(\eta_1,\ldots,\eta_{n-r}\big)\big\vert^{\frac{1}{2}}\big\vert\Omega_{k\beta'}\big(\eta_1,\ldots,\eta_{n-r}\big)\big\vert^{\frac{1}{2}},$$
where $a\in\mathbb{C}$ is defined as follows: let $\rho$ be the half-density in Lemma \ref{Lemma: Complicated half-density} and define $a\in\mathbb{C}$ to be its image under the composition of \eqref{Equation: Image isomorphism} and \eqref{Equation: Push}. By Lemmas \ref{Lemma: Complicated half-density} and \ref{Lemma: Complicated image}, this image is precisely $\kappa(G)$.
We conclude that
\begin{align*}D(\nu_{\mathcal{C}},\rho_{\mathcal{C}'})_x((\widetilde{\eta_1})_{x},\ldots,(\widetilde{\eta_{n-r}})_{x}) & = \kappa(G)\cdot\vert\Omega_{k\beta}(\eta_1,\ldots,\eta_{n-r})\vert^{\frac{1}{2}}\vert\Omega_{k\beta'}(\eta_1,\ldots,\eta_{n-r})\vert^{\frac{1}{2}}\\
& = k^{n-r}\kappa(G)\cdot\vert\Omega_{\beta}(\eta_1,\ldots,\eta_{n-r})\vert^{\frac{1}{2}}\vert\Omega_{\beta'}(\eta_1,\ldots,\eta_{n-r})\vert^{\frac{1}{2}},\end{align*}
completing the proof.
\end{proof}

Fix $w\in W$ and note that $T$ is the $G$-stabilizer of $z_w\in R(\Lambda_{\mathcal{C}})\cap\Lambda_{\mathcal{C}'}$. It follows that
\begin{equation}\label{Equation: Pulling}\varphi_w:G/T\longrightarrow R(\Lambda_{\mathcal{C}})\cap\Lambda_{\mathcal{C}'},\quad [g]\mapsto g\cdot z_w,\quad [g]\in G/T\end{equation}
is a well-defined embedding of $G$-manifolds with image equal to the connected component $\Gamma_w\subseteq R(\Lambda_{\mathcal{C}})\cap\Lambda_{\mathcal{C}'}$ (see Section \ref{Subsection: The geometry}). Let $(\varphi_w)^*D(\nu_{\mathcal{C}},\rho_{\mathcal{C}'})$ denote the density on $G/T$ obtained by pulling $D(\nu_{\mathcal{C}},\rho_{\mathcal{C}'})$ back along $\varphi_w$. Using the notation set just prior to Lemma \ref{Lemma: Volume}, one describes this new density as follows. 

\begin{proposition}\label{Proposition: Pullback}
If $w\in W$, then 
\begin{equation}\label{Equation: Verify}(\varphi_w)^*D(\nu_{\mathcal{C}},\rho_{\mathcal{C}'})=k^{n-r}\kappa(G)\cdot\vert\Xi_{\beta}\vert^{\frac{1}{2}}\vert\Xi_{\beta'}\vert^{\frac{1}{2}}.\end{equation}
\end{proposition}

\begin{proof}
Lemma \ref{Lemma: G-invariance of BKS density} and the $G$-equivariance of $\varphi_w$ imply that the density $(\varphi_w)^*D(\nu_{\mathcal{C}},\rho_{\mathcal{C}'})$ is $G$-invariant. Since the density $\kappa(G)\cdot\vert\Xi_{\beta}\vert^{\frac{1}{2}}\vert\Xi_{\beta'}\vert^{\frac{1}{2}}$ is also $G$-invariant, we are reduced to verifying \eqref{Equation: Verify} at $[e]\in G/T$. Now choose a basis $\{\eta_1,\ldots,\eta_{n-r}\}$ of $\mathfrak{t}^{\perp}$. Identify $T_{[e]}(G/T)$ with $\mathfrak{t}^{\perp}$ in the usual way, noting that $((\varphi_w)^*D(\nu_{\mathcal{C}},\rho_{\mathcal{C}'}))_{[e]}$ and $\kappa(G)\cdot\vert(\Xi_{\beta})_{[e]}\vert^{\frac{1}{2}}\vert(\Xi_{\beta'})_{[e]}\vert^{\frac{1}{2}}$ are then densities on $\mathfrak{t}^{\perp}$. Our task is to verify that $$((\varphi_w)^*D(\nu_{\mathcal{C}},\rho_{\mathcal{C}'}))_{[e]}(\eta_1,\ldots,\eta_{n-r})=\kappa(G)\cdot\vert(\Xi_{\beta})_{[e]}(\eta_1,\ldots,\eta_{n-r})\vert^{\frac{1}{2}}\vert(\Xi_{\beta'})_{[e]}(\eta_1,\ldots,\eta_{n-r})\vert^{\frac{1}{2}}.$$

Consider the differential of $\varphi_w$ at $[e]$, i.e. the vector space isomorphism
$$(d\varphi_w)_{[e]}:\mathfrak{t}^{\perp}\overset{\cong}\longrightarrow T_x(R(\Lambda_{\mathcal{C}})\cap\Lambda_{\mathcal{C}'})$$
with $x:=z_w$. One readily verifies that
$$(d\varphi_w)_{[e]}(\eta_j)=(\widetilde{\eta_j})_x$$ for all $j\in\{1,\ldots,n-r\}$. Hence
\begin{align*}
((\varphi_w)^*D(\nu_{\mathcal{C}},\rho_{\mathcal{C}'}))_{[e]}(\eta_1,\ldots,\eta_{n-r}) & = D(\nu_{\mathcal{C}},\rho_{\mathcal{C}'})_x((\widetilde{\eta_1})_{x},\ldots,(\widetilde{\eta_{n-r}})_{x})\\
& = k^{n-r}\kappa(G)\cdot\vert(\Xi_{\beta})_{[e]}(\eta_1,\ldots,\eta_{n-r})\vert^{\frac{1}{2}}\vert(\Xi_{\beta'})_{[e]}(\eta_1,\ldots,\eta_{n-r})\vert^{\frac{1}{2}},
\end{align*}
where the last line follows from Proposition \ref{Proposition: Main computation}(ii).
\end{proof}

We may now formulate and prove the main result of this section. To this end, recall the definition of the BKS pairing given in Section \ref{Subsection: BKS pairings}. This definition continues to make sense for a quasi-Hamiltonian $G$-space and cleanly intersecting quasi-Hamiltonian Lagrangian submanifolds thereof. One thereby obtains the following Lie-theoretic formula for the BKS pairing of $(R(\Lambda_{\mathcal{C}}),\nu_{\mathcal{C}},\vartheta_{\mathcal{C}})$ and $(\Lambda_{\mathcal{C}'},\rho_{\mathcal{C}'},\psi_{\mathcal{C}'})$.    

\begin{theorem}\label{Theorem: BKS}
The BKS pairing of $(R(\Lambda_{\mathcal{C}}),\nu_{\mathcal{C}},\vartheta_{\mathcal{C}})$ and $(\Lambda_{\mathcal{C}'},\rho_{\mathcal{C}'},\psi_{\mathcal{C}'})$ is given by
$$\mathrm{BKS}((R(\Lambda_{\mathcal{C}}),\nu_{\mathcal{C}},\vartheta_{\mathcal{C}}),(\Lambda_{\mathcal{C}'},\rho_{\mathcal{C}'},\psi_{\mathcal{C}'}))=k^{n-r}C(G,T)\cdot \bigg(\prod_{\alpha\in\Phi_{+}}\alpha(\beta)\alpha(\beta')\bigg)^{\frac{1}{2}}\left(\sum_{w\in W}e^{2\pi i\lvert w\beta-\beta'\rVert^2}\right),$$
where $C(G,T)\in\mathbb{R}$ is a constant depending only on the root system of $(G,T)$, $\Phi_{+}\subseteq\mathfrak{t}^*$ is the set of positive roots, and $\lvert w\beta-\beta'\rVert$ is the length of $w\beta-\beta'$ with respect to $\langle\cdot,\cdot\rangle$.
\end{theorem}

\begin{proof}
Recall the $G$-orbit decomposition of $R(\Lambda_{\mathcal{C}})\cap\Lambda_{\mathcal{C}'}$ given in Proposition \ref{Proposition: Orbit description}, and the fact that $\vartheta_{\mathcal{C}}$ and $\psi_{\mathcal{C}'}$ are $G$-invariant. We also observe that
$$\vartheta_{\mathcal{C}}(z_w)=\psi_{\mathcal{C}}(R^{-1}(z_w))=\psi_{\mathcal{C}}(\exp(w\beta-\beta'),\exp(w\beta))=e^{2\pi i\langle w\beta,w\beta-\beta'\rangle}$$
and 
$$\psi_{\mathcal{C}'}(z_w)=\psi_{\mathcal{C}}(\exp(w\beta-\beta'),\exp(\beta'))=e^{2\pi i\langle \beta',w\beta-\beta'\rangle}$$ for all $w\in W$. These last two sentences imply that $\vartheta_{\mathcal{C}}$ (resp. $\psi_{\mathcal{C}'}$) takes the constant value $e^{2\pi i\langle w\beta,w\beta-\beta'\rangle}$ (resp. $e^{2\pi i\langle \beta',w\beta-\beta'\rangle}$) on $\Gamma_w$ for each $w\in W$. A straightforward computation now reveals that
$\vartheta_{\mathcal{C}}\overline{\psi_{\mathcal{C}'}}$ takes the constant value $e^{2\pi i\lvert w\beta-\beta'\rVert^2}$ on $\Gamma_w$ for each $w\in W$.

Now note that
\begin{align*}& \mathrm{BKS}((R(\Lambda_{\mathcal{C}}),\nu_{\mathcal{C}},\vartheta_{\mathcal{C}}),(\Lambda_{\mathcal{C}'},\rho_{\mathcal{C}'},\psi_{\mathcal{C}'}))\\ & = \bigintss_{R(\beta_{\mathcal{C}})\cap\Lambda_{\mathcal{C}'}}\vartheta_{\mathcal{C}}\overline{\psi_{\mathcal{C}'}}D(\nu_{\mathcal{C}},\rho_{\mathcal{C}'})\\
& = \sum_{w\in W}\bigintss_{\Gamma_w}\vartheta_{\mathcal{C}}\overline{\psi_{\mathcal{C}'}}D(\nu_{\mathcal{C}},\rho_{\mathcal{C}'})\hspace{10pt}\text{(by Proposition \ref{Proposition: Orbit description} and Corollary \ref{Corollary: Easy})}\\
& = \sum_{w\in W}e^{2\pi i\lvert w\beta-\beta'\rVert^2}\bigintss_{\Gamma_w}D(\nu_{\mathcal{C}},\rho_{\mathcal{C}'})\hspace{10pt}\text{(by the previous paragraph)}\\
& = \sum_{w\in W}e^{2\pi i\lvert w\beta-\beta'\rVert^2}\bigintss_{G/T}(\varphi_w)^*D(\nu_{\mathcal{C}},\rho_{\mathcal{C}'})\hspace{10pt}\text{(see \eqref{Equation: Pulling})}\\
& = k^{n-r}\kappa(G)\left(\bigintss_{G/T}\vert\Xi_{\beta}\vert^{\frac{1}{2}}\vert\Xi_{\beta'}\vert^{\frac{1}{2}}\right)\left(\sum_{w\in W}e^{2\pi i\lvert w\beta-\beta'\rVert^2}\right)\hspace{10pt}\text{(by Proposition \ref{Proposition: Pullback})}\\
& = k^{n-r}\kappa(G)\frac{\bigg(\prod_{\alpha\in\Phi_{+}}\alpha(\beta)\alpha(\beta')\bigg)^{\frac{1}{2}}}{\prod_{\alpha\in\Phi_{+}}\langle\alpha,\rho\rangle}\left(\sum_{w\in W}e^{2\pi i\lvert w\beta-\beta'\rVert^2}\right)\hspace{10pt}\text{(by Lemma \ref{Lemma: Volume})}.
\end{align*}
It just remains to observe that the real number
$$\frac{\kappa(G)}{\prod_{\alpha\in\Phi_{+}}\langle\alpha,\rho\rangle}$$
depends only on the root system of $(G,T)$.
\end{proof}  

\section*{Appendix: Aspects of the BKS half-density}\label{Appendix}
\renewcommand{\thesubsection}{A\arabic{subsection}}
\setcounter{subsection}{0}
\subsection{Some preliminaries}
Let $V$ be an $n$-dimensional real vector space and fix a positive real number $\alpha$. Recall that a \textit{density of order} $\alpha$ on $V$ is a map $\rho:V^{\oplus n}\longrightarrow\mathbb{C}$ satisfying
$$\rho(A(v_1),\ldots,A(v_n))=\vert\det(A)\vert^{\alpha}\rho(v_1,\ldots,v_n)$$ for all $(v_1,\ldots,v_n)\in V^{\oplus n}$ and $A\in\mathrm{End}(V)$. Pointwise addition and scalar multiplication yield a complex vector space structure on the set of densities of order $\alpha$, and we let $\vert V\vert^{\alpha}$ denote the resulting complex vector space. Note that if $\{v_1,\ldots,v_n\}$ is a basis of $V$, then each $\rho\in\vert V\vert^{\alpha}$ is completely determined by the complex number $\rho(v_1,\ldots,v_n)$. It follows that $\vert V\vert^{\alpha}$ is one-dimensional for all $d$.

Now let $V$ be as above and suppose that $\alpha_1$ and $\alpha_2$ are positive real numbers. We may consider the pointwise product of $\rho_1\in\vert V\vert^{\alpha_1}$ and $\rho_2\in\vert V\vert^{\alpha_2}$, i.e. the map
$$\rho_1\rho_2:V^{\oplus n}\longrightarrow\mathbb{C},\quad (v_1,\ldots,v_n)\mapsto \rho_1(v_1,\ldots,v_n)\rho_2(v_1,\ldots,v_n),\quad (v_1,\ldots,v_n)\in V^{\oplus n}.$$ One readily verifies that $\rho_1\rho_2\in\vert V\vert^{\alpha_1+\alpha_2}$, and that we have an isomorphism
$$\vert V\vert^{\alpha_1}\otimes\vert V\vert^{\alpha_2}\overset{\cong}\longrightarrow\vert V\vert^{\alpha_1+\alpha_2},\quad \rho_1\otimes\rho_2\mapsto\rho_1\rho_2,\quad\rho_1\in\vert V\vert^{\alpha_1},\text{ }\rho_2\in\vert V\vert^{\alpha_2}.$$ 

One uses the term \textit{density} (resp. \textit{half-density}) in reference to a density of order $1$ (resp. $\frac{1}{2}$) and writes $\vert V\vert$ for $\vert V\vert^{\alpha}$ when $\alpha=1$. The previous paragraph then specializes to yield an isomorphism \begin{equation}\label{Equation: Multiplication iso}\vert V\vert^{\frac{1}{2}}\otimes\vert V\vert^{\frac{1}{2}}\overset{\cong}\longrightarrow\vert V\vert,\quad \rho\otimes\rho'\longrightarrow\rho\rho',\quad\rho,\rho'\in\vert V\vert^{\frac{1}{2}}.\end{equation}

Let $\phi:V\longrightarrow V'$ be an isomorphism of $n$-dimensional real vector spaces, and fix a positive real number $\alpha$. Given any $\rho\in\vert V'\vert^{\alpha}$, define $\phi^*(\rho)\in\vert V\vert^{\alpha}$ by the formula
$$\phi^*(\rho)(v_1,\ldots,v_n)=\rho(\phi(v_1),\ldots,\phi(v_n)),\quad (v_1,\ldots,v_n)\in V^{\oplus n}.$$ This defines a complex vector space isomorphism
$$\phi^*:\vert V'\vert^{\alpha}\overset{\cong}\longrightarrow\vert V\vert^{\alpha}.$$

\subsection{Exact sequences}\label{Subsection: Exact sequences}
Suppose that one has an exact sequence
$$0\longrightarrow U\overset{i}\longrightarrow V\overset{j}\longrightarrow W\longrightarrow 0$$ of finite-dimensional real vector spaces. This sequence canonically determines a vector space isomorphism
\begin{equation}\label{Equation: Half-density iso}\vert U\vert^{\frac{1}{2}}\otimes\vert W\vert^{\frac{1}{2}}\overset{\cong}\longrightarrow\vert V\vert^{\frac{1}{2}},\end{equation}
described as follows. Fix half-densities $\rho_U\in \vert U\vert^{\frac{1}{2}}$ and $\rho_W\in \vert W\vert^{\frac{1}{2}}$, and let $\rho_V$ denote the image of $\rho_U\otimes\rho_W$ under \eqref{Equation: Half-density iso}. To define $\rho_V$, we choose a subspace $W'\subseteq V$ with the property that $V=\mathrm{ker}(j)\oplus W'$. It follows that $j\vert_{W'}:W'\longrightarrow W$ is an isomorphism. Choose bases $\{u_1,\ldots,u_m\}$ and $\{w_1,\ldots,w_n\}$ for $U$ and $W$, respectively, noting that $$\{v_1,\ldots,v_{m+n}\}:=\{i(u_1),\ldots,i(u_m),(j\vert_{W'})^{-1}(w_1),\ldots,(j\vert_{W'})^{-1}(w_n)\}$$ is a basis of $V$. One then has
$$\rho_V(v_1,\ldots,v_{m+n})=\rho_U(u_1,\ldots,u_m)\rho_W(w_1,\ldots,w_n),$$
a condition that completely determines $\rho_V$.

\begin{lemma}
	Suppose that $U$, $V$, and $W$ are finite-dimensional real vector spaces, and that we have two exact sequences
	$$0\longrightarrow U\overset{i}\longrightarrow V\overset{j}\longrightarrow W\longrightarrow 0\quad\text{and}\quad 0\longrightarrow U\overset{i'}\longrightarrow V\overset{j'}\longrightarrow W\longrightarrow 0.$$ Let $$\theta:\vert U\vert^{\frac{1}{2}}\otimes\vert W\vert^{\frac{1}{2}}\overset{\cong}\longrightarrow\vert V\vert^{\frac{1}{2}}\quad\text{and}\quad\theta': \vert U\vert^{\frac{1}{2}}\otimes\vert W\vert^{\frac{1}{2}}\overset{\cong}\longrightarrow\vert V\vert^{\frac{1}{2}}$$
	be the isomorphisms induced by the first and second sequences, respectively. Let $k:V\longrightarrow V$ be an isomorphism that makes the squares in 
	$$\begin{tikzcd}
	0\arrow[r] & U \arrow[r, "i"] \arrow[d, "\mathrm{id}_U"] & V \arrow[r, "j"] \arrow[d, "k"] & W \arrow[r] \arrow[d, "\mathrm{id}_W"]  & 0 \\
	0\arrow[r]
	& U \arrow[r, "i'"] & V \arrow[r, "j'"] & W \arrow[r] & 0
	\end{tikzcd}$$
	commute, where $\mathrm{id}_U$ and $\mathrm{id}_W$ are the identity maps. We then have 
	$$\theta=\vert\det(k)\vert^{\frac{1}{2}}\theta'.$$
\end{lemma}

\begin{proof}
	Let $\rho_U\in\vert U\vert^{\frac{1}{2}}$ and $\rho_W\in\vert W\vert^{\frac{1}{2}}$ be given, and set $\rho_V:=\theta(\rho_U\otimes\rho_W)$ and $\rho_V':=\theta'(\rho_U\otimes\rho_W)$. Our task is to prove that 
	$$\rho_V=\vert\det(k)\vert^{\frac{1}{2}}\rho_V'.$$ We begin by choosing a complement $C$ of $\mathrm{ker}(j)$ in $V$, noting that $C':=k(C)$ is a complement of $\mathrm{ker}(j')$ in $V$. Let us also choose bases $\{u_1,\ldots,u_m\}$ and $\{w_1,\ldots,w_n\}$ of $U$ and $W$, respectively. We then have bases
	$$\{v_1,\ldots,v_{m+n}\}:=\{i(u_1),\ldots,i(u_m),(j\vert_{C})^{-1}(w_1),\ldots,(j\vert_{C})^{-1}(w_n)\}$$
	and $$\{v_1',\ldots,v_{m+n}'\}:=\{i'(u_1),\ldots,i'(u_m),(j'\vert_{C'})^{-1}(w_1),\ldots,(j'\vert_{C'})^{-1}(w_n)\}$$
	of $V$. Our description of \eqref{Equation: Half-density iso} implies that
	\begin{equation}\label{Equation: anice}\rho_V(v_1,\ldots,v_{n+m})=\rho_U(u_1,\ldots,u_m)\rho_W(w_1,\ldots,w_n)=\rho_V'(v_1',\ldots,v_{m+n}').\end{equation} At the same time, note that
	\begin{equation}\label{Equation: bnice}v_{p}'=i'(u_p)=k(i(u_p))=k(v_p)\end{equation} for all $p\in\{1,\ldots,m\}$. We also have $$j'(k(v_{m+q}))=j'(k((j\vert_{C})^{-1}(w_q)))=j((j\vert_{C})^{-1}(w_q))=w_q$$ for all $q\in\{1,\ldots,n\}$, and we note that $k(v_{m+q})\in C'$ for all $q\in\{1,\ldots,n\}$. This last sentence tells us that
\begin{equation}\label{Equation: cnice}k(v_{m+q})=(j'\vert_{C'})^{-1}(w_q)=v'_{m+q}\end{equation} for all $q\in\{1,\ldots,n\}$. Hence
	\begin{align*}
	\rho_V(v_1,\ldots,v_{n+m}) & = \rho_V'(v_1',\ldots,v_{n+m}')\hspace{100pt}\text{[by \eqref{Equation: anice}]}\\
	& = \rho_{V}'(k(v_1),\ldots,k(v_{m+n}))\hspace{70pt}\text{[by \eqref{Equation: bnice} and \eqref{Equation: cnice}]}\\
	& = \vert\det(k)\vert^{\frac{1}{2}}\rho_{V}'(v_1,\ldots,v_{m+n}).
	\end{align*}
	Since $\{v_1,\ldots,v_{m+n}\}$ is a basis of $V$, this establishes that $\rho_V=\vert\det(k)\vert^{\frac{1}{2}}\rho_V'$.
\end{proof}

One also has the following result. Its proof is straightforward and omitted.

\begin{lemma}\label{Lemma: Linear algebra}
	Suppose that
	$$\begin{tikzcd}
	0\arrow[r] & U \arrow[r] \arrow[d, "k"] & V \arrow[r] \arrow[d, "\ell"] & W \arrow[r] \arrow[d, "m"]  & 0 \\
	0\arrow[r]
	& U' \arrow[r] & V' \arrow[r] & W' \arrow[r] & 0
	\end{tikzcd}$$
	is a diagram of finite-dimensional real vector spaces and linear maps. Assume that the rows are exact, the vertical maps are isomorphisms, and that the two squares commute. Let $$\theta:\vert U\vert^{\frac{1}{2}}\otimes\vert W\vert^{\frac{1}{2}}\overset{\cong}\longrightarrow\vert V\vert^{\frac{1}{2}}\quad\text{and}\quad\theta': \vert U'\vert^{\frac{1}{2}}\otimes\vert W'\vert^{\frac{1}{2}}\overset{\cong}\longrightarrow\vert V'\vert^{\frac{1}{2}}$$
	be the isomorphisms induced by the top and bottom rows, respectively. The diagram
	$$\begin{tikzcd}
	\vert U'\vert^{\frac{1}{2}}\otimes\vert W'\vert^{\frac{1}{2}}\arrow[d, "k^*\otimes m^*"] \arrow[r, "\theta' "] & \vert V'\vert^{\frac{1}{2}}\arrow[d, "\ell^*"] \\
	\vert U\vert^{\frac{1}{2}}\otimes \vert W\vert^{\frac{1}{2}}\arrow[r, "\theta"]
	& \vert V\vert^{\frac{1}{2}}
	\end{tikzcd}$$
	then commutes.
\end{lemma}

\begin{lemma}\label{Lemma: Linear algebra2}
	Suppose that $$0\longrightarrow U\overset{i}\longrightarrow V\overset{j}\longrightarrow W\longrightarrow 0\quad\text{and}\quad 0\longrightarrow U'\overset{i'}\longrightarrow V'\overset{j'}\longrightarrow W'\longrightarrow 0$$ are exact sequences of finite-dimensional real vector spaces and linear maps. Denote by $$\theta:\vert U\vert^{\frac{1}{2}}\otimes\vert W\vert^{\frac{1}{2}}\overset{\cong}\longrightarrow\vert V\vert^{\frac{1}{2}}\quad\text{and}\quad\theta': \vert U'\vert^{\frac{1}{2}}\otimes\vert W'\vert^{\frac{1}{2}}\overset{\cong}\longrightarrow\vert V'\vert^{\frac{1}{2}}$$ the isomorphisms induced by the first and second sequences, respectively. Now consider the exact sequence $$0\longrightarrow U\oplus U'\xrightarrow{i\oplus i'} V\oplus V'\xrightarrow{j\oplus j'} W\oplus W'\longrightarrow 0$$ and the induced isomorphism
	$$\theta'':\vert U\oplus U'\vert^{\frac{1}{2}} \otimes\vert W\oplus W'\vert^{\frac{1}{2}}\overset{\cong}\longrightarrow\vert V\oplus V'\vert^{\frac{1}{2}}.$$
	Let us also consider the canonical isomorphisms
	$$\xi:\vert U\vert^{\frac{1}{2}}\otimes\vert U'\vert^{\frac{1}{2}}\overset{\cong}\longrightarrow\vert U\oplus U'\vert^{\frac{1}{2}},\quad\eta: \vert W\vert^{\frac{1}{2}}\otimes\vert W'\vert^{\frac{1}{2}}\overset{\cong}\longrightarrow \vert W\oplus W'\vert^{\frac{1}{2}},\quad\text{and}\quad \zeta: \vert V\vert^{\frac{1}{2}}\otimes\vert V'\vert^{\frac{1}{2}}\overset{\cong}\longrightarrow \vert V\oplus V'\vert^{\frac{1}{2}}.$$ 
	We then have a commutative diagram
	
	$$\begin{tikzcd}
	\vert V\vert^{\frac{1}{2}}\otimes \vert V'\vert^{\frac{1}{2}}\arrow[d, "(\theta\otimes\theta')^{-1}"] \arrow[r, "\zeta"] & \vert V\oplus V'\vert^{\frac{1}{2}}\arrow[d, "(\theta'')^{-1}"] \\
	(\vert U\vert^{\frac{1}{2}}\otimes \vert W\vert^{\frac{1}{2}})\otimes(\vert U'\vert^{\frac{1}{2}}\otimes\vert W'\vert^{\frac{1}{2}})\arrow[r, "\tau "]
	& \vert U\oplus U'\vert^{\frac{1}{2}}\otimes \vert W\oplus W'\vert^{\frac{1}{2}}
	\end{tikzcd},$$
	
	where $\tau$ is the composite map 
	$$(\vert U\vert^{\frac{1}{2}}\otimes \vert W\vert^{\frac{1}{2}})\otimes(\vert U'\vert^{\frac{1}{2}}\otimes\vert W'\vert^{\frac{1}{2}})\overset{\cong}\longrightarrow(\vert U\vert^{\frac{1}{2}}\otimes \vert U'\vert^{\frac{1}{2}})\otimes(\vert W\vert^{\frac{1}{2}}\otimes\vert W'\vert^{\frac{1}{2}})\xrightarrow{\xi\otimes\eta}\vert U\oplus U'\vert^{\frac{1}{2}}\otimes \vert W\oplus W'\vert^{\frac{1}{2}}.$$
\end{lemma}

\begin{proof}
	It suffices to prove that $\zeta\circ (\theta\otimes\theta')=\theta''\circ\tau$. To this end, choose bases $\{u_1,\ldots,u_m\}$, $\{u_1',\ldots,u_n'\}$, $\{w_1,\ldots,w_p\}$, and $\{w_1',\ldots,w_q'\}$ of $U$, $U'$, $W$, and $W'$, respectively. Let $\{v_1,\ldots,v_{m+p}\}$ be a basis of $V$ satisfying $v_k=i(u_k)$ for all $k\in\{1,\ldots,m\}$ and $j(v_{\ell+m})=w_{\ell}$ for all $\ell\in\{1,\ldots,p\}$. We may analogously find a basis $\{v_1',\ldots,v_{n+q}'\}$ of $V'$ satisfying $v_k'=i'(u_k')$ for all $k\in\{1,\ldots,n\}$ and $j'(v_{\ell+n})=w_{\ell}'$ for all $\ell\in\{1,\ldots,q\}$. Observe that 
	$$\mathcal{B}:=\{(v_1,0),\ldots,(v_{m+p},0),(0,v_1'),\ldots,(0,v_{n+q}')\}$$ is a basis of $V\oplus V'$, and that a suitable reordering of basis vectors produces the alternative basis
	$$\mathcal{B}':=\{(v_1,0),\ldots,(v_m,0),(0,v_1'),\ldots,(0,v_n'),(v_{m+1},0),\ldots,(v_{m+p},0),(0,v_{n+1}'),\ldots,(0,v_{n+q}')\}.$$
	
	Now suppose that $\rho\in\vert U\vert^{\frac{1}{2}}$, $\rho'\in\vert U'\vert^{\frac{1}{2}}$, $\nu\in\vert W\vert^{\frac{1}{2}}$, and $\nu'\in \vert W'\vert^{\frac{1}{2}}$. If we regard $\mathcal{B}$ and $\mathcal{B}'$ as tuples of vectors in $V\oplus V'$, then
	\begin{align*}
	(\zeta\circ(\theta\otimes\theta'))(\rho\otimes\nu\otimes\rho'\otimes\nu')(\mathcal{B})&=\theta(\rho\otimes\nu)(v_1,\ldots,v_{m+p})\theta'(\rho'\otimes\nu')(v_1',\ldots,v_{n+q}')\\
	&=\rho(u_1,\ldots,u_m)\nu(w_1,\ldots,w_p)\rho'(u_1',\ldots,u_n')\nu'(w_1',\ldots,w_q')
	\end{align*}
	and
	\begin{align*}
	&(\theta''\circ\tau)(\rho\otimes\nu\otimes\rho'\otimes\nu')(\mathcal{B}')\\ &=\xi(\rho\otimes\rho')((u_1,0),\ldots,(u_m,0),(0,u_1'),\ldots,(0,u_n'))\eta(\nu\otimes\nu')((w_1,0),\ldots,(w_p,0),(0,w_1'),\ldots,(0,w_q'))\\
	&=\rho(u_1,\ldots,u_m)\rho'(u_1',\ldots,u_n')\nu(w_1,\ldots,w_p)\nu'(w_1',\ldots,w_q').
	\end{align*}
	It follows that
	$$(\zeta\circ(\theta\otimes\theta'))(\rho\otimes\nu\otimes\rho'\otimes\nu')(\mathcal{B})=(\theta''\circ\tau)(\rho\otimes\nu\otimes\rho'\otimes\nu')(\mathcal{B}').$$ Since $\mathcal{B}'$ is obtained by reordering the vectors in $\mathcal{B}$, this amounts to the statement
	$$(\zeta\circ(\theta\otimes\theta'))(\rho\otimes\nu\otimes\rho'\otimes\nu')(\mathcal{B})=(\theta''\circ\tau)(\rho\otimes\nu\otimes\rho'\otimes\nu')(\mathcal{B}).$$ We conclude that $$(\zeta\circ(\theta\otimes\theta'))(\rho\otimes\nu\otimes\rho'\otimes\nu')=(\theta''\circ\tau)(\rho\otimes\nu\otimes\rho'\otimes\nu'),$$ completing the proof. 
\end{proof}

\subsection{The BKS density}\label{Subsection: The BKS density}
Let $M$ be a manifold equipped with an arbitrary two-form $\omega\in\Omega^2(M)$. Consider the open submanifold of points at which $\omega$ is non-degenerate, i.e.
$$M^{\circ}:=\{m\in M:\omega_m\text{ is non-degenerate}\}.$$ It follows that  $(T_mM,\omega_m)$ is a symplectic vector space for all $m\in M^{\circ}$. This suggests the following generalization of Lagrangian submanifolds in symplectic geometry (cf. Definition \ref{Definition: quasi-Hamiltonian Lagrangian}).

\begin{definition}
	We call a submanifold $\Lambda\subseteq M$ \textit{Lagrangian} if it satisfies the following conditions:
	\begin{itemize}
		\item[(i)] $\Lambda\subseteq M^{\circ}$;
		\item[(ii)] $T_m\Lambda$ is a Lagrangian subspace of $(T_mM,\omega_m)$ for all $m\in\Lambda$.
	\end{itemize} 
\end{definition}  

Let $\Lambda_1,\Lambda_2\subseteq M$ be Lagrangian submanifolds that intersect cleanly in $M$. Fix $m\in\Lambda_1\cap\Lambda_2$, and let $T_m\Lambda_1+T_m\Lambda_2$ be the subspace of $T_mM$ generated by $T_m\Lambda_1$ and $T_m\Lambda_2$. One may restrict $\omega_m$ to a bilinear form on $T_m\Lambda_1+T_m\Lambda_2$, and easily verify that this restricted form descends to a symplectic form on $\frac{T_m\Lambda_1+T_m\Lambda_2}{T_m(\Lambda_1\cap\Lambda_2)}$. This new symplectic form induces a half-density on $\frac{T_m\Lambda_1+T_m\Lambda_2}{T_m(\Lambda_1\cap\Lambda_2)}$, which in turn forms a basis of $\bigg\vert \frac{T_m\Lambda_1+T_m\Lambda_2}{T_m(\Lambda_1\cap \Lambda_2)}\bigg\vert^{\frac{1}{2}}$. We thereby obtain a canonical isomorphism
\begin{equation}\label{Equation: Canonical2}\bigg\vert \frac{T_m\Lambda_1+T_m\Lambda_2}{T_m(\Lambda_1\cap \Lambda_2)}\bigg\vert^{\frac{1}{2}}\overset{\cong}\longrightarrow\mathbb{C}.\end{equation}

Now consider the exact sequence
\begin{equation}\label{Equation: Second exact}
0\longrightarrow T_m(\Lambda_1\cap\Lambda_2)\overset{\alpha}\longrightarrow T_m\Lambda_1+T_m\Lambda_2\overset{\beta}\longrightarrow\frac{T_m\Lambda_1+T_m\Lambda_2}{T_m(\Lambda_1\cap\Lambda_2)}\longrightarrow 0,
\end{equation}
where $\alpha$ is inclusion and $\beta$ is the quotient map. Let us also consider the external direct sum $T_m\Lambda_1\oplus T_m\Lambda_2$ of $T_m\Lambda_1$ and $T_m\Lambda_2$. The two linear maps $$\gamma:T_m(\Lambda_1\cap\Lambda_2)\longrightarrow T_m\Lambda_1\oplus T_m\Lambda_2,\quad v\mapsto (v,v),\quad v\in T_m(\Lambda_1\cap\Lambda_2)$$ and $$\delta:T_m\Lambda_1\oplus T_m\Lambda_2\longrightarrow T_m\Lambda_1+T_m\Lambda_2,\quad (v_1,v_2)\mapsto v_1-v_2,\quad (v_1,v_2)\in T_m\Lambda_1\oplus T_m\Lambda_2$$
then fit into an exact sequence
\begin{equation}\label{Equation: First exact}0\longrightarrow T_m(\Lambda_1\cap\Lambda_2)\overset{\gamma}\longrightarrow T_m\Lambda_1\oplus T_m\Lambda_2\overset{\delta}\longrightarrow T_m\Lambda_1+T_m\Lambda_2\longrightarrow 0.\end{equation} 

We now observe that

\begin{subequations}\label{Equation: List}
	\begin{align}
	\vert T_m\Lambda_1\vert^{\frac{1}{2}}\otimes\vert T_m\Lambda_2\vert^{\frac{1}{2}} & \cong \vert T_m\Lambda_1\oplus T_m\Lambda_2\vert^{\frac{1}{2}}\hspace{175pt}\text{[by \eqref{Equation: Half-density iso}]}\label{Equation: a''}\\
	& \cong \vert T_m(\Lambda_1\cap\Lambda_2)\vert^{\frac{1}{2}}\otimes\vert T_m\Lambda_1+T_m\Lambda_2\vert^{\frac{1}{2}}\hspace{89pt}\text{[by \eqref{Equation: Half-density iso} and \eqref{Equation: First exact}]}\label{Equation: b''}\\
	& \cong  \vert T_m(\Lambda_1\cap\Lambda_2)\vert^{\frac{1}{2}}\otimes\vert T_m(\Lambda_1\cap\Lambda_2)\vert^{\frac{1}{2}}\otimes\bigg\vert\frac{T_m\Lambda_1+T_m\Lambda_2}{T_m(\Lambda_1\cap\Lambda_2)}\bigg\vert^{\frac{1}{2}}\text{[by \eqref{Equation: Half-density iso} and \eqref{Equation: Second exact}]}\label{Equation: c''}\\
	& \cong \vert T_m(\Lambda_1\cap\Lambda_2)\vert\hspace{190pt}\text{[by \eqref{Equation: Canonical2}]}\label{Equation: d''}.
	\end{align}
\end{subequations}
These considerations yield a canonical isomorphism
\begin{equation}\label{Equation: Canonical isomorphism}
\Phi_m:\vert T_m\Lambda_1\vert^{\frac{1}{2}}\otimes\vert T_m\Lambda_2\vert^{\frac{1}{2}}\overset{\cong}\longrightarrow \vert T_m(\Lambda_1\cap\Lambda_2)\vert.
\end{equation}
If $\rho_1$ and $\rho_2$ are half-densities on $\Lambda_1$ and $\Lambda_2$, respectively, then the formula
$$D(\rho_1,\rho_2)_m:=\Phi_m((\rho_1)_m\otimes(\rho_2)_m),\quad m\in\Lambda_1\cap\Lambda_2$$
defines a density $D(\rho_1,\rho_2)$ on $\Lambda_1\cap\Lambda_2$.

\begin{definition}
	One calls $D(\rho_1,\rho_2)$ the \textit{BKS density} on $\Lambda_1\cap\Lambda_2$ associated to $\rho_1$ and $\rho_2$.
\end{definition}

\subsection{Some technical results}
Let $M$ be a manifold endowed with a two-form $\omega\in\Omega^2(M)$, and suppose that $\Lambda_1,\Lambda_2\subseteq M$ are cleanly intersecting Lagrangian submanifolds. Each fixed $m\in\Lambda_1\cap\Lambda_2$ has an associated inclusion $\theta:T_m(\Lambda_1\cap\Lambda_2)\oplus T_m(\Lambda_1\cap\Lambda_2)\longrightarrow T_m\Lambda_1\oplus T_m\Lambda_2$ and map $$\vartheta:T_m\Lambda_1\oplus T_m\Lambda_2\longrightarrow\frac{T_m\Lambda_1+T_m\Lambda_2}{T_m(\Lambda_1\cap\Lambda_2)},\quad (v_1,v_2)\mapsto [v_1-v_2],\quad (v_1,v_2)\in T_m\Lambda_1\oplus T_m\Lambda_2,$$ where $[v_1-v_2]$ denotes the class of $v_1-v_2$ in $\frac{T_m\Lambda_1+T_m\Lambda_2}{T_m(\Lambda_1\cap\Lambda_2)}$. We then have an exact sequence
\begin{equation}\label{Equation: Newest exact}
0\longrightarrow T_m(\Lambda_1\cap\Lambda_2)\oplus T_m(\Lambda_1\cap\Lambda_2)\overset{\theta}\longrightarrow T_m\Lambda_1\oplus T_m\Lambda_2\overset{\vartheta}\longrightarrow\frac{T_m\Lambda_1+T_m\Lambda_2}{T_m(\Lambda_1\cap\Lambda_2)},\end{equation} so that
\begin{subequations}
	\begin{align}
	\vert T_m\Lambda_1\oplus T_m\Lambda_2\vert^{\frac{1}{2}} & \cong \vert T_m(\Lambda_1\cap\Lambda_2)\oplus T_m(\Lambda_1\cap\Lambda_2)\vert^{\frac{1}{2}}\otimes \bigg\vert\frac{T_m\Lambda_1+T_m\Lambda_2}{T_m(\Lambda_1\cap\Lambda_2)}\bigg\vert^{\frac{1}{2}}\hspace{15pt}\text{[by \eqref{Equation: Half-density iso} and \eqref{Equation: Newest exact}]}\label{Equation: a'}\\
	& \cong \vert T_m(\Lambda_1\cap\Lambda_2)\vert^{\frac{1}{2}}\otimes \vert T_m(\Lambda_1\cap\Lambda_2)\vert^{\frac{1}{2}}\otimes \bigg\vert\frac{T_m\Lambda_1+T_m\Lambda_2}{T_m(\Lambda_1\cap\Lambda_2)}\bigg\vert^{\frac{1}{2}}\hspace{3pt}\text{[by \eqref{Equation: Half-density iso}]}\label{Equation: b'}  
	\end{align}
\end{subequations}

\begin{lemma}\label{Lemma: Intermediate}
	The isomorphism $$\vert T_m\Lambda_1\oplus T_m\Lambda_2\vert^{\frac{1}{2}}\overset{\cong}\longrightarrow\vert T_m(\Lambda_1\cap\Lambda_2)\vert^{\frac{1}{2}}\otimes \vert T_m(\Lambda_1\cap\Lambda_2)\vert^{\frac{1}{2}}\otimes \bigg\vert\frac{T_m\Lambda_1+T_m\Lambda_2}{T_m(\Lambda_1\cap\Lambda_2)}\bigg\vert^{\frac{1}{2}}$$
	obtained by composing \eqref{Equation: b''} and \eqref{Equation: c''} equals that obtained by composing \eqref{Equation: a'} and \eqref{Equation: b'}.
\end{lemma}

\begin{proof}
	Let us write
	\begin{subequations}
		\begin{align*}
		& \psi: \vert T_m\Lambda_1\oplus T_m\Lambda_2\vert^{\frac{1}{2}}\overset{\cong}\longrightarrow\vert T_m(\Lambda_1\cap\Lambda_2)\vert^{\frac{1}{2}}\otimes\vert T_m\Lambda_1+T_m\Lambda_2\vert^{\frac{1}{2}}\\
		& \varphi: \vert T_m(\Lambda_1\cap\Lambda_2)\vert^{\frac{1}{2}}\otimes\vert T_m\Lambda_1+T_m\Lambda_2\vert^{\frac{1}{2}}\overset{\cong}\longrightarrow \vert T_m(\Lambda_1\cap\Lambda_2)\vert^{\frac{1}{2}}\otimes \vert T_m(\Lambda_1\cap\Lambda_2)\vert^{\frac{1}{2}}\otimes \bigg\vert\frac{T_m\Lambda_1+T_m\Lambda_2}{T_m(\Lambda_1\cap\Lambda_2)}\bigg\vert^{\frac{1}{2}}\\
		& \psi': \vert T_m\Lambda_1\oplus T_m\Lambda_2\vert^{\frac{1}{2}}\overset{\cong}\longrightarrow\vert T_m(\Lambda_1\cap\Lambda_2)\oplus T_m(\Lambda_1\cap\Lambda_2)\vert^{\frac{1}{2}}\otimes \bigg\vert\frac{T_m\Lambda_1+T_m\Lambda_2}{T_m(\Lambda_1\cap\Lambda_2)}\bigg\vert^{\frac{1}{2}},\text{ and}\\
		& \varphi':\vert T_m(\Lambda_1\cap\Lambda_2)\oplus T_m(\Lambda_1\cap\Lambda_2)\vert^{\frac{1}{2}}\otimes \bigg\vert\frac{T_m\Lambda_1+T_m\Lambda_2}{T_m(\Lambda_1\cap\Lambda_2)}\bigg\vert^{\frac{1}{2}}\overset{\cong}\longrightarrow\vert T_m(\Lambda_1\cap\Lambda_2)\vert^{\frac{1}{2}}\otimes \vert T_m(\Lambda_1\cap\Lambda_2)\vert^{\frac{1}{2}}\otimes \bigg\vert\frac{T_m\Lambda_1+T_m\Lambda_2}{T_m(\Lambda_1\cap\Lambda_2)}\bigg\vert^{\frac{1}{2}}  
		\end{align*}
	\end{subequations} 
	for the isomorphisms \eqref{Equation: b''}, \eqref{Equation: c''}, \eqref{Equation: a'}, and \eqref{Equation: b'}, respectively. Our objective is to prove that $\varphi\circ\psi=\varphi'\circ\psi'$. In what follows, we verify the equivalent statement that $\psi^{-1}\circ\varphi^{-1}=(\psi')^{-1}\circ(\varphi')^{-1}$. 
	
	We begin by making a few choices, the first being a basis $\{u_1,\ldots,u_m\}$ of $T_m(\Lambda_1\cap\Lambda_2)$. Let us also consider the subspaces of $T_m\Lambda_1\oplus T_m\Lambda_2$ defined by $W':=T_m(\Lambda_1\cap\Lambda_2)\oplus\{0\}$ and $W'':=\{0\}\oplus T_m(\Lambda_1\cap\Lambda_2)$. It follows that $$\{w_1',\ldots,w_m'\}:=\{(u_1,0),\ldots,(u_m,0)\}\quad\text{and}\quad \{w_1'',\ldots,w_m''\}:=\{(0,u_1),\ldots,(0,u_m)\}$$ are bases of $W'$ and $W''$, respectively. At the same time, choose a complement $V_1$ (resp. $V_2$) of $T_m(\Lambda_1\cap\Lambda_2)$ in $T_m\Lambda_1$ (resp. $T_m\Lambda_2$) and set $$W:=V_1\oplus V_2\subseteq T_m\Lambda_1\oplus T_m\Lambda_2.$$ Fix a basis $\{v_1,\ldots,v_n\}$ (resp. $\{\overline{v}_1,\ldots,\overline{v}_n\}$) of $V_1$ (resp. $V_2$), noting that
	$$\{w_1,\ldots,w_{2n}\}:=\{(v_1,0),\ldots,(v_n,0),(0,\overline{v}_1),\ldots,(0,\overline{v}_n)\}$$ is a basis of $W$.
	
	Let us now describe $\rho:=(\psi^{-1}\circ\varphi^{-1})(\rho_1\otimes\rho_2\otimes\rho_3)$ for arbitrary $\rho_1,\rho_2\in\vert T_m(\Lambda_1\cap\Lambda_2)\vert^{\frac{1}{2}}$ and $\rho_3\in \bigg\vert\frac{T_m\Lambda_1+T_m\Lambda_2}{T_m(\Lambda_1\cap\Lambda_2)}\bigg\vert^{\frac{1}{2}}$. To this end, recall the map $$\delta:T_m\Lambda_1\oplus T_m\Lambda_2\longrightarrow T_m\Lambda_1+T_m\Lambda_2$$ from Section \ref{Subsection: The BKS density}. Note that $\delta$ identifies $W$ with a complement $\delta(W)$ of $T_m(\Lambda_1\cap\Lambda_2)$ in $T_m\Lambda_1+T_m\Lambda_2$. We thereby obtain
	$$\varphi^{-1}(\rho_1\otimes\rho_2\otimes\rho_3)=\rho_1\otimes\rho',$$
	where $\rho'\in\vert T_m\Lambda_1+T_m\Lambda_2\vert^{\frac{1}{2}}$ and
	\begin{align*}
	\rho'(u_1,\ldots,u_m,\delta(w_1),\ldots,\delta(w_{2n})) & =\rho_2(u_1,\ldots,u_m)\rho_3(\beta(\delta(w_1)),\ldots,\beta(\delta(w_{2n})))\\
	& = \rho_2(u_1,\ldots,u_m)\rho_3(\vartheta(w_1),\ldots,\vartheta(w_{2n})).
	\end{align*}
	We also observe that $W+W''$ is a complement of $\mathrm{image}(\alpha)$ in $T_m\Lambda_1\oplus T_m\Lambda_2$, where $$\alpha:T_m(\Lambda_1\cap\Lambda_2)\longleftrightarrow T_m\Lambda_1+T_m\Lambda_2$$ is the inclusion. Noting that $\rho:=\psi^{-1}(\rho_1\otimes\rho')$, our last sentence implies that
	\begin{equation}\label{Equation: Block1}\begin{aligned}\rho(\alpha(u_1),\ldots,\alpha(u_m),w_1'',\ldots,w_m'',w_1,\ldots,w_{2n}) & =\rho_1(u_1,\ldots,u_m)\rho'(\delta(w_1''),\ldots,\delta(w_m''),\delta(w_1),\ldots,\delta(w_{2n}))\\
	& = \rho_1(u_1,\ldots,u_m)\rho'(-u_1,\ldots,-u_m,\delta(w_1),\ldots,\delta(w_{2n}))\\
	& = \rho_1(u_1,\ldots,u_m)\rho'(u_1,\ldots,u_m,\delta(w_1),\ldots,\delta(w_{2n}))\\
	& = \rho_1(u_1,\ldots,u_m)\rho_2(u_1,\ldots,u_m)\rho_3(\vartheta(w_1),\ldots,\vartheta(w_{2n})).
	\end{aligned}
	\end{equation}
	
	We next describe $\nu:=((\psi')^{-1}\circ(\varphi')^{-1})(\rho_1\otimes\rho_2\otimes\rho_3)$. Note that
	$$(\varphi')^{-1}(\rho_1\otimes\rho_2\otimes\rho_3)=\nu'\otimes\rho_3,$$ where $\nu'\in\vert T_m(\Lambda_1\cap\Lambda_2)\oplus T_m(\Lambda_1\cap\Lambda_2)\vert^{\frac{1}{2}}$ and
	\begin{align*}
	\nu(w_1',\ldots,w_m',w_1'',\ldots,w_m'') & = \rho_1(u_1,\ldots,u_m)\rho_2(u_1,\ldots,u_m).
	\end{align*} 
	We also observe that $W$ is a complement of $T_m(\Lambda_1\cap\Lambda_2)\oplus T_m(\Lambda_1\cap\Lambda_2)$ in $T_m\Lambda_1\oplus T_m\Lambda_2$. Since $\nu=(\psi')^{-1}(\nu'\otimes\rho_3)$, the previous sentence implies that
	\begin{equation}\label{Equation: Block2}
	\begin{aligned}
	\nu(w_1',\ldots,w_m',w_1'',\ldots,w_m'',w_1,\ldots,w_{2n}) & = \nu'(w_1',\ldots,w_m',w_1'',\ldots,w_m'')\rho_3(\vartheta(w_1),\ldots,\vartheta(w_{2n}))\\
	& = \rho_1(u_1,\ldots,u_m)\rho_2(u_1,\ldots,u_m)\rho_3(\vartheta(w_1),\ldots,\vartheta(w_{2n})).
	\end{aligned}
	\end{equation}
	Now consider the linear automorphism $A:T_m\Lambda_1\oplus T_m\Lambda_2\longrightarrow T_m\Lambda_1\oplus T_m\Lambda_2$ that sends the ordered basis $\{w_1',\ldots,w_m',w_1'',\ldots,w_m'',w_1,\ldots,w_{2n}\}$ to the ordered basis $\{\alpha(u_1),\ldots,\alpha(u_m),w_1'',\ldots,w_m'',w_1,\ldots,w_{2n}\}$. This automorphism satisfies $\det(A)=1$, so that
	$$\nu(w_1',\ldots,w_m',w_1'',\ldots,w_m'',w_1,\ldots,w_{2n})=\nu(\alpha(u_1),\ldots,\alpha(u_m),w_1'',\ldots,w_m'',w_1,\ldots,w_{2n}).$$ Along with \eqref{Equation: Block1} and \eqref{Equation: Block2}, this implies that $\rho$ and $\nu$ have the same value on a basis of $T_m\Lambda_1\oplus T_m\Lambda_2$. It follows that $\rho=\nu$, completing the proof. 
\end{proof}

\subsection{A reformulation of the BKS density}
We again take $M$ to be a manifold equipped with a two-form $\omega\in\Omega^2(M)$ and cleanly intersecting Lagrangian submanifolds $\Lambda_1,\Lambda_2\subseteq M$. Fix a point $m\in\Lambda_1\cap\Lambda_2$ and let $V_1$ and $V_2$ be complements of $T_m(\Lambda_1\cap\Lambda_2)$ in $T_m\Lambda_1$ and $T_m\Lambda_2$, respectively, i.e.
\begin{equation}\label{Equation: Splittings}
T_m\Lambda_1=T_m(\Lambda_1\cap\Lambda_2)\oplus V_1\quad\text{and}\quad T_m\Lambda_2=T_m(\Lambda_1\cap\Lambda_2)\oplus V_2.
\end{equation}
Assume that $V_1\cap V_2=\{0\}$ and recall the map $$\delta:T_m\Lambda_1\oplus T_m\Lambda_2\longrightarrow T_m\Lambda_1+T_m\Lambda_2$$ from Section \ref{Subsection: The BKS density}.

\begin{lemma}\label{Lemma: Linear algebra2}
	The map $$\delta\big\vert_{V_1\oplus V_2}:V_1\oplus V_2\longrightarrow T_m\Lambda_1+T_m\Lambda_2$$ is injective, and $\delta(V_1\oplus V_2)$ is a complement of $T_m(\Lambda_1\cap\Lambda_2)$ in $T_m\Lambda_1+T_m\Lambda_2$.
\end{lemma}

\begin{proof}
	Since $V_1\cap V_2=\{0\}$, $V_1\oplus V_2$ has a trivial intersection with $\mathrm{ker}(\delta)$. It follows that $\delta\big\vert_{V_1\oplus V_2}$ is injective and
	\begin{align*}\dim(\delta(V_1\oplus V_2)) & = \dim(V_1)+\dim(V_2)\\
	& =\bigg(\dim(T_m\Lambda_1)-\dim(T_m(\Lambda_1\cap\Lambda_2))\bigg)+\bigg(\dim(T_m\Lambda_2)-\dim(T_m(\Lambda_1\cap\Lambda_2))\bigg)\\
	& = \dim(T_m\Lambda_1+T_m\Lambda_2)-\dim(T_m(\Lambda_1\cap\Lambda_2)). 
	\end{align*}
	We are therefore reduced to proving that $T_m(\Lambda_1\cap\Lambda_2) + \delta(V_1\oplus V_2)= T_m\Lambda_1+T_m\Lambda_2$. At the same time, we observe that $\delta(V_1\oplus V_2)=V_1+V_2$. This yields
	\begin{align*}
	T_m(\Lambda_1\cap\Lambda_2) + \delta(V_1\oplus V_2) & = T_m(\Lambda_1\cap\Lambda_2)+V_1+V_2\\
	& = (T_m(\Lambda_1\cap\Lambda_2)+V_1)+(T_m(\Lambda_1\cap\Lambda_2)+V_2)\\
	& = T_m\Lambda_1+T_m\Lambda_2,
	\end{align*}
	completing the proof.
\end{proof}

Now recall the map
$$\beta:T_m\Lambda_1+T_m\Lambda_2\longrightarrow \frac{T_m\Lambda_1+T_m\Lambda_2}{T_m(\Lambda_1\cap\Lambda_2)}$$ from Section \ref{Subsection: The BKS density}.

\begin{corollary}\label{Corollary: Computational}
	The map
	\begin{equation}\label{Equation: Composite}(\beta\circ\delta)\big\vert_{V_1\oplus V_2}:V_1\oplus V_2\longrightarrow\frac{T_m\Lambda_1+T_m\Lambda_2}{T_m(\Lambda_1\cap\Lambda_2)}\end{equation} is an isomorphism.
\end{corollary}

We have canonical isomorphisms
\begin{align}
T_m\Lambda_1\oplus T_m\Lambda_2 & \cong \bigg(T_m(\Lambda_1\cap\Lambda_2)\oplus V_1\bigg)\oplus \bigg(T_m(\Lambda_1\cap\Lambda_2)\oplus V_2\bigg)\label{Equation: Can1}\\
& \cong \bigg(T_m(\Lambda_1\cap\Lambda_2)\oplus T_m(\Lambda_1\cap\Lambda_2)\bigg)\oplus\bigg(V_1\oplus V_2\bigg)\label{Equation: Can2} 
\end{align}
It follows that
\begin{subequations}\label{Equation: List2}
	\begin{align}
	\vert T_m\Lambda_1\vert^{\frac{1}{2}}\otimes\vert T_m\Lambda_2\vert^{\frac{1}{2}} & \cong (\vert T_m(\Lambda_1\cap\Lambda_2)\vert^{\frac{1}{2}}\otimes\vert V_1\vert^{\frac{1}{2}})\otimes(\vert T_m(\Lambda_1\cap\Lambda_2)\vert^{\frac{1}{2}}\otimes\vert V_2\vert^{\frac{1}{2}})\hspace{10pt}\text{[by \eqref{Equation: Half-density iso} and \eqref{Equation: Splittings}]}\\
	& \cong  \vert T_m(\Lambda_1\cap\Lambda_2)\oplus T_m(\Lambda_1\cap\Lambda_2)\vert^{\frac{1}{2}}\otimes\vert V_1\oplus V_2\vert^{\frac{1}{2}}\hspace{52pt}\text{[by \eqref{Equation: Half-density iso}]}\\
	& \cong  \vert T_m(\Lambda_1\cap\Lambda_2)\oplus T_m(\Lambda_1\cap\Lambda_2)\vert^{\frac{1}{2}}\otimes \bigg\vert\frac{T_m\Lambda_1+T_m\Lambda_2}{T_m(\Lambda_1\cap\Lambda_2)}\bigg\vert^{\frac{1}{2}}\hspace{18pt}\text{[by Corollary \ref{Corollary: Computational}]}\\
	& \cong \vert T_m(\Lambda_1\cap\Lambda_2)\vert\hspace{195pt}\text{[by \eqref{Equation: Multiplication iso},\eqref{Equation: Half-density iso}, and \eqref{Equation: Canonical2}]}\label{Equation: d'''}.
	\end{align}
\end{subequations}
Composing the isomorphisms \eqref{Equation: List2} results in an isomorphism
\begin{equation}\label{Equation: Nicest}
\Phi_m^{V_1,V_2}:\vert T_m\Lambda_1\vert^{\frac{1}{2}}\otimes\vert T_m\Lambda_2\vert^{\frac{1}{2}}\overset{\cong}\longrightarrow \vert T_m(\Lambda_1\cap\Lambda_2)\vert.
\end{equation}

\begin{proposition}\label{Proposition: Two BKS}
	We have $\Phi_m^{V_1,V_2}=\Phi_m$.
\end{proposition}

\begin{proof}
	Let $\gamma_1:T_m\Lambda_1\longrightarrow V_1$ and $\gamma_2:T_m\Lambda_2\longrightarrow V_2$ denote the projections induced by the decompositions $T_m\Lambda_1=T_m(\Lambda_1\cap\Lambda_2)\oplus V_1$ and $T_m\Lambda_2=T_m(\Lambda_1\cap\Lambda_2)\oplus V_2$, respectively. We then have a commutative diagram
	$$\begin{tikzcd}
	0 \arrow[r] & T_m(\Lambda_1\cap\Lambda_2)\oplus T_m(\Lambda_1\cap\Lambda_2) \arrow[r] \arrow[d, "\mathrm{id}"]& T_m\Lambda_1\oplus T_m\Lambda_2\arrow[r, "\gamma_1\oplus\gamma_2"]\arrow[d, "\mathrm{id}"] & V_1\oplus V_2 \arrow[r] \arrow[d, "\beta' "] & 0 \\
	0 \arrow[r] & T_m(\Lambda_1\cap\Lambda_2)\oplus T_m(\Lambda_1\cap\Lambda_2) \arrow[r] & T_m\Lambda_1\oplus T_m\Lambda_2\arrow[r, "\overline{\delta}"] & \frac{T_m\Lambda_1+T_m\Lambda_2}{T_m(\Lambda_1\cap\Lambda_2)} \arrow[r] & 0
	\end{tikzcd},$$
	with exact rows, where $\beta'$ is the isomorphism obtained by restricting $\beta$ to $V_1\oplus V_2\subseteq T_m\Lambda_1\oplus T_m\Lambda_2$ and $\overline{\delta}$ is the map to $\frac{T_m\Lambda_1+T_m\Lambda_2}{T_m(\Lambda_1\cap\Lambda_2)}$ induced by $\delta$. This induces a commutative diagram
	\begin{equation}\label{Equation: CD'1}
	\begin{tikzcd}
	\vert T_m\Lambda_1\oplus T_m\Lambda_2\vert^{\frac{1}{2}}\arrow[d, "\mathrm{id}"] \arrow[r, "f_1"] & \vert T_m(\Lambda_1\cap\Lambda_2)\oplus T_m(\Lambda_1\cap\Lambda_2)\vert^{\frac{1}{2}}\otimes\bigg\vert \frac{T_m\Lambda_1+T_m\Lambda_2}{T_m(\Lambda_1\cap\Lambda_2)}\bigg\vert^{\frac{1}{2}}\arrow[d, "f_3"] \\
	\vert T_m\Lambda_1\oplus T_m\Lambda_2\vert^{\frac{1}{2}}\arrow[r, "f_2"]
	& \vert T_m(\Lambda_1\cap\Lambda_2)\oplus T_m(\Lambda_1\cap\Lambda_2)\vert^{\frac{1}{2}}\otimes\vert V_1\oplus V_2\vert^{\frac{1}{2}}
	\end{tikzcd},
	\end{equation}
	where $f_3=\mathrm{id}\otimes (\beta')^*$ and the upper (resp. lower) horizontal map is induced by the lower (resp. upper) row in the previous diagram. On the other hand, applying Lemma \ref{Lemma: Linear algebra2} to the exact sequences
	$$0\longrightarrow T_m(\Lambda_1\cap\Lambda_2)\longrightarrow T_m\Lambda_1\overset{\gamma_1}\longrightarrow V_1\longrightarrow 0\quad\text{and}\quad  0\longrightarrow T_m(\Lambda_1\cap\Lambda_2)\longrightarrow T_m\Lambda_2\overset{\gamma_2}\longrightarrow V_2\longrightarrow 0$$
	gives rise to a commutative diagram
	\begin{equation}\label{Equation: CD'2}
	\begin{tikzcd}
	\vert T_m\Lambda_1\vert^{\frac{1}{2}}\otimes \vert T_m\Lambda_2\vert^{\frac{1}{2}}\arrow[d, "f_4"] \arrow[r, "f_5"] & \vert T_m\Lambda_1\oplus T_m\Lambda_2\vert^{\frac{1}{2}}\arrow[d, "f_7"] \\
	(\vert T_m(\Lambda_1\cap\Lambda_2)\vert^{\frac{1}{2}}\otimes \vert V_1\vert^{\frac{1}{2}})\otimes(\vert T_m(\Lambda_1\cap\Lambda_2)\vert^{\frac{1}{2}}\otimes\vert V_2\vert^{\frac{1}{2}})\arrow[r, "f_6"]
	& \vert T_m\Lambda_1\oplus T_m\Lambda_2\vert^{\frac{1}{2}}\otimes \vert V_1\oplus V_2\vert^{\frac{1}{2}}.
	\end{tikzcd}.
	\end{equation}
	We also have a commutative diagram
	\begin{equation}\label{Equation: CD'3}
	\begin{tikzcd}
	\vert T_m(\Lambda_1\cap\Lambda_2)\oplus T_m(\Lambda_1\cap\Lambda_2)\vert^{\frac{1}{2}}\otimes\bigg\vert \frac{T_m\Lambda_1+T_m\Lambda_2}{T_m(\Lambda_1\cap\Lambda_2)}\bigg\vert^{\frac{1}{2}} \ar[dd, "f_8"] \ar[dr, "f_9"] & & \\
	& \vert T_m(\Lambda_1\cap\Lambda_2)\vert,\\
	\vert T_m(\Lambda_1\cap\Lambda_2)\vert^{\frac{1}{2}}\otimes\vert T_m(\Lambda_1\cap\Lambda_2)\vert^{\frac{1}{2}}\otimes\bigg\vert \frac{T_m\Lambda_1+T_m\Lambda_2}{T_m(\Lambda_1\cap\Lambda_2)}\bigg\vert^{\frac{1}{2}} \ar[ur, "f_{10}"] & &
	\end{tikzcd}
	\end{equation}
	where $f_8$, $f_9$, and $f_{10}$ are the isomorphisms defined in \eqref{Equation: b'}, \eqref{Equation: d'''}, and \eqref{Equation: d''}, respectively. Lemma \ref{Lemma: Intermediate} then combines with the definition of $\Phi_m$ to imply that $\Phi_m=f_{10}\circ f_8\circ f_1\circ f_5$, while the definition of $\Phi_m^{V_1,V_2}$ gives $\Phi_m^{V_1,V_2}=f_9\circ f_3^{-1}\circ f_6\circ f_4$. The equality $\Phi_m^{V_1,V_2}=\Phi_m$ now follows easily from the fact that \eqref{Equation: CD'1}, \eqref{Equation: CD'2}, and \eqref{Equation: CD'3} commute. 
\end{proof}

\subsection{A result on half-densities and $G$-invariance}
Recall the Lie-theoretic notation and conventions established in Section \ref{Subsection: The basics}, as well as the BKS density $D(\rho_1,\rho_2)$ from Section \ref{Subsection: The BKS density}. This density is compatible with group actions in the following sense.  

\begin{lemma}\label{Lemma: G-invariance of densities}
	Let $M$ be a manifold equipped with a smooth $G$-action and a $G$-invariant two-form $\omega\in\Omega^2(M)^G$. Suppose that $\Lambda_1,\Lambda_2\subseteq M$ are $G$-invariant Lagrangian submanifolds having a clean intersection. If $\rho_1$ and $\rho_2$ are $G$-invariant half-densities on $\Lambda_1$ and $\Lambda_2$, respectively, then $D(\rho_1,\rho_2)$ is a $G$-invariant half-density on $\Lambda_1\cap\Lambda_2$. 
\end{lemma}

\begin{proof}
	Fix $g\in G$ and $m\in M$. Observe that $g$ acts on $M$ through a diffeomorphism $M\longrightarrow M$ whose differential defines isomorphisms 
	\begin{subequations}\label{Equation: Isomorphism list}
		\begin{align}
		& T_{m}\Lambda_1\overset{\cong}\longrightarrow T_{g\cdot m}\Lambda_1,\label{Equation: a}\\
		& T_{m}\Lambda_2\overset{\cong}\longrightarrow T_{g\cdot m}\Lambda_2,\label{Equation: b}\\
		& T_{m}(\Lambda_1\cap\Lambda_2)\overset{\cong}\longrightarrow T_{g\cdot m}(\Lambda_1\cap\Lambda_2)\label{Equation: c}\\
		& T_{m}\Lambda_1+T_{m}\Lambda_2\overset{\cong}\longrightarrow T_{g\cdot m}\Lambda_1+T_{g\cdot m}\Lambda_2,\label{Equation: d}\\
		& \frac{T_{m}\Lambda_1+T_{ m}\Lambda_2}{T_{m}(\Lambda_1\cap\Lambda_2)}\overset{\cong}\longrightarrow\frac{T_{g\cdot m}\Lambda_1+T_{g\cdot m}\Lambda_2}{T_{g\cdot m}(\Lambda_1\cap\Lambda_2)}.\label{Equation: e}  
		\end{align}
	\end{subequations}
	Note that \eqref{Equation: a}, \eqref{Equation: b}, and their direct sum fit into the commutative diagram
	\begin{equation}\label{Equation: First diag}\begin{tikzcd}
	0\arrow[r] & T_{m}\Lambda_1 \arrow[r] \arrow[d, "\cong"] & T_{m}\Lambda_1\oplus T_{m}\Lambda_2 \arrow[r] \arrow[d, "\cong"] & T_{m}\Lambda_2 \arrow[r] \arrow[d, "\cong"]  & 0 \\
	0\arrow[r]
	& T_{g\cdot m}\Lambda_1 \arrow[r] & T_{g\cdot m}\Lambda_1\oplus T_{g\cdot m}\Lambda_2 \arrow[r] & T_{g\cdot m}\Lambda_2 \arrow[r] & 0
	\end{tikzcd},\end{equation} where $T_m\Lambda_1\longrightarrow T_m\Lambda_1\oplus T_m\Lambda_2$ and $T_{g\cdot m}\Lambda_1\longrightarrow T_{g\cdot m}\Lambda_1\oplus T_{g\cdot m}\Lambda_2$ are the usual inclusions and $T_m\Lambda_1\oplus T_m\Lambda_2\longrightarrow T_m\Lambda_2$ and $T_{g\cdot m}\Lambda_1\oplus T_{g\cdot m}\Lambda_2\longrightarrow T_{g\cdot m}\Lambda_2$ are the usual projections. The isomorphisms \eqref{Equation: Isomorphism list} also connect \eqref{Equation: Second exact} and \eqref{Equation: First exact} at the point $m$ to \eqref{Equation: Second exact} and \eqref{Equation: First exact} at the point $g\cdot m$, respectively. A more precise statement is that we have the commutative diagrams
	\begin{equation}\label{Equation: Second diag}\begin{tikzcd}
	0\arrow[r] & T_{m}(\Lambda_1\cap\Lambda_2) \arrow[r] \arrow[d, "\cong"] & T_{m}\Lambda_1\oplus T_{m}\Lambda_2 \arrow[r] \arrow[d, "\cong"] & T_{m}\Lambda_1+T_{m}\Lambda_2 \arrow[r] \arrow[d, "\cong"]  & 0 \\
	0\arrow[r]
	& T_{g\cdot m}(\Lambda_1\cap\Lambda_2) \arrow[r] & T_{g\cdot m}\Lambda_1\oplus T_{g\cdot m}\Lambda_2 \arrow[r] & T_{g\cdot m}\Lambda_1+T_{g\cdot m}\Lambda_2 \arrow[r] & 0
	\end{tikzcd}\end{equation}
	and
	\begin{equation}\label{Equation: Third diag}\begin{tikzcd}
	0\arrow[r] & T_{m}(\Lambda_1\cap\Lambda_2) \arrow[r] \arrow[d, "\cong"] & T_{m}\Lambda_1+T_{m}\Lambda_2 \arrow[r] \arrow[d, "\cong"] & \frac{T_{m}\Lambda_1+T_{g\cdot m}\Lambda_2}{T_{m}(\Lambda_1\cap\Lambda_2)} \arrow[r] \arrow[d, "\cong"]  & 0 \\
	0\arrow[r]
	& T_{g\cdot m}(\Lambda_1\cap\Lambda_2) \arrow[r] & T_{g\cdot m}\Lambda_1+T_{g\cdot m}\Lambda_2 \arrow[r] & \frac{T_{g\cdot m}\Lambda_1+T_{g\cdot m}\Lambda_2}{T_{g\cdot m}(\Lambda_1\cap\Lambda_2)} \arrow[r] & 0
	\end{tikzcd}.\end{equation}
	By applying Lemma \ref{Lemma: Linear algebra} to \eqref{Equation: First diag}, \eqref{Equation: Second diag}, and \eqref{Equation: Third diag}, we obtain the commutative diagrams
	\begin{equation}\label{Equation: CD1}
	\begin{tikzcd}
	\vert T_{g\cdot m}\Lambda_1\vert^{\frac{1}{2}}\otimes\vert T_{g\cdot m}\Lambda_2\vert^{\frac{1}{2}}\arrow[d, "\cong"] \arrow[r, "\cong"] & \vert T_{g\cdot m}\Lambda_1\oplus T_{g\cdot m}\Lambda_2\vert^{\frac{1}{2}}\arrow[d, "\cong"] \\
	\vert T_m\Lambda_1\vert^{\frac{1}{2}}\otimes \vert T_m\Lambda_2\vert^{\frac{1}{2}}\arrow[r, "\cong"]
	& \vert T_m\Lambda_1\oplus T_m\Lambda_2\vert^{\frac{1}{2}}
	\end{tikzcd},\end{equation} 
	
	\begin{equation}\label{Equation: CD2}
	\begin{tikzcd}
	\vert T_{g\cdot m}(\Lambda_1\cap\Lambda_2)\vert^{\frac{1}{2}}\otimes\vert T_{g\cdot m}\Lambda_1+T_{g\cdot m}\Lambda_2\vert^{\frac{1}{2}}\arrow[d, "\cong"] \arrow[r, "\cong"] & \vert T_{g\cdot m}\Lambda_1\oplus T_{g\cdot m}\Lambda_2\vert^{\frac{1}{2}}\arrow[d, "\cong"] \\
	\vert T_{m}(\Lambda_1\cap\Lambda_2)\vert^{\frac{1}{2}}\otimes\vert T_{m}\Lambda_1+T_{m}\Lambda_2\vert^{\frac{1}{2}}\arrow[r, "\cong"]
	& \vert T_m\Lambda_1\oplus T_m\Lambda_2\vert^{\frac{1}{2}}
	\end{tikzcd},\end{equation}
	and
	\begin{equation}\label{Equation: CD3}
	\begin{tikzcd}
	\vert T_{g\cdot m}(\Lambda_1\cap\Lambda_2)\vert^{\frac{1}{2}}\otimes\bigg\vert\frac{T_{g\cdot m}\Lambda_1+T_{g\cdot m}\Lambda_2}{T_{g\cdot m}(\Lambda_1\cap\Lambda_2)}\bigg\vert^{\frac{1}{2}}\arrow[d, "\cong"] \arrow[r, "\cong"] &  \vert T_{g\cdot m}\Lambda_1+T_{g\cdot m}\Lambda_2\vert^{\frac{1}{2}} \arrow[d, "\cong"] \\
	\vert T_{m}(\Lambda_1\cap\Lambda_2)\vert^{\frac{1}{2}}\otimes\bigg\vert\frac{T_{m}\Lambda_1+T_{m}\Lambda_2}{T_{m}(\Lambda_1\cap\Lambda_2)}\bigg\vert^{\frac{1}{2}} \arrow[r, "\cong"] &  \vert T_{m}\Lambda_1+T_{m}\Lambda_2\vert^{\frac{1}{2}}
	\end{tikzcd}.\end{equation}
	
	Let us again consider the diffeomorphism $M\longrightarrow M$ through which $g$ acts on $M$. Since this diffeomorphism preserves $\omega$, \eqref{Equation: e} is an isomorphism of symplectic vector spaces.
	It follows that
	\begin{equation}\label{Equation: CD4}
	\begin{tikzcd}
	\bigg\vert\frac{T_{g\cdot m}\Lambda_1+T_{g\cdot m}\Lambda_2}{T_{g\cdot m}(\Lambda_1\cap\Lambda_2)}\bigg\vert^{\frac{1}{2}}\arrow[d, "\cong"] \arrow[r, "\cong"] & \mathbb{C}\arrow[d, "\mathrm{id}_{\mathbb{C}}"] \\
	\bigg\vert\frac{T_{m}\Lambda_1+T_{m}\Lambda_2}{T_{m}(\Lambda_1\cap\Lambda_2)}\bigg\vert^{\frac{1}{2}} \arrow[r, "\cong"] & \mathbb{C}
	\end{tikzcd}
	\end{equation}
	commutes, where the leftmost vertical map is induced by \eqref{Equation: e} and the horizontal maps are \eqref{Equation: Canonical2} at the points $m$ and $g\cdot m$. We also observe that
	\begin{equation}\label{Equation: CD5}
	\begin{tikzcd}
	\vert T_{g\cdot m}(\Lambda_1\cap\Lambda_2)\vert^{\frac{1}{2}}\otimes\vert T_{g\cdot m}(\Lambda_1\cap\Lambda_2)\vert^{\frac{1}{2}}\arrow[d, "\cong"] \arrow[r, "\cong"] & \vert T_{g\cdot m}(\Lambda_1\cap\Lambda_2)\vert\arrow[d, "\cong"] \\
	\vert T_{m}(\Lambda_1\cap\Lambda_2)\vert^{\frac{1}{2}}\otimes\vert T_{m}(\Lambda_1\cap\Lambda_2)\vert^{\frac{1}{2}} \arrow[r, "\cong"] & \vert T_{m}(\Lambda_1\cap\Lambda_2)\vert
	\end{tikzcd}
	\end{equation} 
	commutes, where the vertical maps are induced by \eqref{Equation: c} and the horizontal maps are specific instances of \eqref{Equation: Multiplication iso}.  
	
	Now recall that \eqref{Equation: Canonical isomorphism} is obtained by composing the isomorphisms \eqref{Equation: List}. With this point in mind, one can combine \eqref{Equation: CD1}, \eqref{Equation: CD2}, \eqref{Equation: CD3}, \eqref{Equation: CD4}, and \eqref{Equation: CD5} to produce the following commutative diagram:
	\begin{equation}
	\begin{tikzcd}
	\vert T_{g\cdot m}\Lambda_1\vert^{\frac{1}{2}}\otimes\vert T_{g\cdot m}\Lambda_2\vert^{\frac{1}{2}}\arrow[d, "\cong"] \arrow[r, "\psi'"] & \vert T_{g\cdot m}(\Lambda_1\cap\Lambda_2)\vert\arrow[d, "\cong"] \\
	\vert T_{m}\Lambda_1\vert^{\frac{1}{2}}\otimes\vert T_{m}\Lambda_2\vert^{\frac{1}{2}} \arrow[r, "\psi"] & \vert T_{m}(\Lambda_1\cap\Lambda_2)\vert,
	\end{tikzcd}
	\end{equation}
	where horizontal maps are \eqref{Equation: Canonical isomorphism} at $m$ and $g\cdot m$, the leftmost vertical map comes from \eqref{Equation: a} and \eqref{Equation: b}, and the rightmost vertical map is induced by \eqref{Equation: c}. Since $\rho_1$ and $\rho_2$ are $G$-invariant, $(\rho_1)_m\otimes(\rho_2)_m$ is the image of $(\rho_1)_{g\cdot m}\otimes(\rho_2)_{g\cdot m}$ under the leftmost vertical map.  It follows that $\psi((\rho_1)_m\otimes(\rho_2)_m)$ is the image of $\psi'((\rho_1)_{g\cdot m}\otimes(\rho_2)_{g\cdot m})$ under the rightmost vertical map. One also has $\psi((\rho_1)_m\otimes(\rho_2)_m)=D(\rho_1,\rho_2)_m$ and $\psi'((\rho_1)_{g\cdot m}\otimes(\rho_2)_{g\cdot m})=D(\rho_1,\rho_2)_{g\cdot m}$, which combine with the previous sentence to show that $D(\rho_1,\rho_2)$ is $G$-invariant.
\end{proof}
     
\section*{Notation}
\begin{itemize}
	\item $Q(M)$ --- quantization of $M$
	\item $M_1\overset{\Lambda}\Longrightarrow M_2$ --- Lagrangian relation
	\item $M^{\circ}$ --- open submanifold of $M$ on which a given two-form is non-degenerate
	\item $D(\rho_1,\rho_2)$ --- BKS density determined by the half-densities $\rho_1$ and $\rho_2$
	\item $G$ --- compact, connected, simply-connected, simple Lie group
	\item $n$ --- dimension of $G$
	\item $r$ --- rank of $G$
	\item $\theta^L$ --- left-invariant Maurer--Cartan form on $G$
	\item $\theta^R$ --- right-invariant Maurer--Cartan form on $G$
	\item $G_{\text{reg}}$ --- regular elements of $G$
	\item $\mathcal{C}$ --- conjugacy class in $G$
	\item $\Lambda_{\mathcal{C}}$ --- submanifold of points $(g,h)\in D(G)$ satisfying $ghg^{-1}=h$ and $h\in\mathcal{C}$
	\item $G_x$ --- $G$-stabilizer of a point $x$ in a $G$-manifold
	\item $Gx$ --- $G$-orbit of a point $x$ in a $G$-manifold
	\item $\mathfrak{g}$ --- Lie algebra of $G$
	\item $\mathfrak{g}_x$ --- Lie algebra of $G_x$
	\item $\xi_M$ --- fundamental vector field of $\xi\in\mathfrak{g}$ on a $G$-manifold $M$ 
	\item $\Theta_{G_{\xi}}$ --- left-invariant form volume form on $G_{\xi}$ that induces the Haar measure, where $\xi\in\mathfrak{g}$
	\item $\Theta_{\xi}$ --- value of $\Theta_{G_{\xi}}$ at $e\in G_{\xi}$, where $\xi\in\mathfrak{g}$
	\item $\langle\cdot,\cdot\rangle$ --- normalized inner product on $\mathfrak{g}$
	\item $V^{\perp}$ --- orthogonal complement of a subspace $V\subseteq\mathfrak{g}$ with respect to $\langle\cdot,\cdot\rangle$
	\item $\mathfrak{g}_{\text{reg}}$ --- regular elements of $\mathfrak{g}$
	\item $\exp$ --- exponential map $\mathfrak{g}\longrightarrow G$
	\item $\mathrm{Ad}_g(\xi)$ --- adjoint action of $g\in G$ on $\xi\in\mathfrak{g}$
	\item $\mathcal{O}$ --- adjoint orbit of $G$ in $\mathfrak{g}$
	\item $\Lambda_{\mathcal{O}}$ --- submanifold of points $(g,\xi)\in G\times\mathfrak{g}$ satisfying $\mathrm{Ad}_g(\xi)=\xi$ and $\xi\in\mathcal{O}$
	\item $\omega_{\mathcal{O}}$ --- KKS symplectic form on $\mathcal{O}$
	\item $\omega_{\xi}$ --- value of $\omega_{\mathcal{O}}$ at $\xi\in\mathcal{O}$
	\item $\Omega_{\mathcal{O}}$ --- highest non-zero wedge power of $\omega_{\mathcal{O}}$
	\item $\Omega_{\xi}$ --- value of $\Omega_{\mathcal{O}}$ at $\xi\in\mathcal{O}$
	\item $\rho_{\mathcal{O}}$ --- Guillemin--Sternberg half-density on $\Lambda_{\mathcal{O}}$
	\item $\psi_{\mathcal{O}}$ --- phase function on $\Lambda_{\mathcal{O}}$
	\item $\rho_{\mathcal{C}}$ --- half-density on $\Lambda_{\mathcal{C}}$
	\item $\psi_{\mathcal{C}}$ --- phase function on $\Lambda_{\mathcal{C}}$
	\item $R$ --- Dehn twist automorphism $D(G)\longrightarrow D(G)$ defined by $R(g,h)=(g,hg^{-1})$
	\item $\nu_{\mathcal{C}}$ --- half-density on $R(\Lambda_{\mathcal{C}})$ determined by $\mathcal{O}$
	\item $\vartheta_{\mathcal{C}}$ --- phase function on $R(\Lambda_{\mathcal{C}})$
	\item $\omega_{D(G)}$ --- two-form on the quasi-Hamiltonian $G$-space $D(G)$
	\item $\mu_{D(G)}$ --- moment map on the quasi-Hamiltonian $G$-space $D(G)$
	\item $\mu_{D(G)}^{-1}(e)_{\text{reg}}$ --- intersection of $\mu_{D(G)}^{-1}(e)$ and $D(G)_{\text{reg}}$ in $D(G)$
	\item $\overline{\omega_{D(G)}}$ --- symplectic form on $\mu_{D(G)}^{-1}(e)_{\text{reg}}/G$
	\item $T$ --- maximal torus in $G$
	\item $\mathfrak{t}$ --- Lie algebra of $T$
	\item $W$ --- Weyl group of $(G,T)$
	\item $\Phi$ --- roots of $(G,T)$
	\item $\Delta$ --- chosen set of simple roots in $\Phi$
	\item $\ell(w)$ --- length of $w\in W$ with respect to $\Delta$
	\item $T_{\text{reg}}$ --- intersection of $T$ and $G_{\text{reg}}$ in $G$
	\item $\omega_{T\times T}$ --- symplectic form on $T\times T$
	\item $\overline{\omega_{T\times T}}$ --- symplectic form on $(T\times T_{\text{reg}})/W$
\end{itemize} 
\bibliographystyle{acm} 
\bibliography{CrW}

\end{document}